\documentclass[a4paper,11pt]{article}
\usepackage[utf8]{inputenc}
\usepackage{amssymb,amsmath,mathrsfs,amsthm}
\usepackage{verbatim}
\usepackage{indentfirst} 
\usepackage{geometry}
\usepackage{bbm}
\usepackage{tikz-cd}

\usepackage{hyperref} 
\hypersetup{colorlinks=true,allcolors=blue}
\usepackage{hypcap}

\usepackage{multicol}
\usepackage{etoolbox}
\usepackage[refpage]{nomencl}

\patchcmd{\thenomenclature}
{\section*{\nomname}}
{\begin{multicols}{2}[\section*{\nomname}]}
	{}{}
	\appto\endthenomenclature{\end{multicols}}

\newcommand{\nomentry}[2]{
  \nomenclature{#1}{#2}%
}

\newtheorem{thm}{Theorem}[section]
\newtheorem*{rem*}{Remark}
\newtheorem*{defi*}{Definition}
\newtheorem{defi}[thm]{Definition}
\newtheorem{prop}[thm]{Proposition}
\newtheorem{cor}[thm]{Corollary}
\newtheorem{exa}[thm]{Example}
\newtheorem{lem}[thm]{Lemma}
\newtheorem{rem}[thm]{Remark}

\numberwithin{equation}{section}

\renewcommand{\bf}[1]{\mathbf{#1}}
\renewcommand{\rm}[1]{\mathrm{#1}}
\renewcommand{\cal}[1]{\mathcal{#1}}
\renewcommand{\frak}[1]{\mathfrak{#1}}
\newcommand{\bb}[1]{\mathbb{#1}}
\newcommand{\scr}[1]{\mathscr{#1}}
\newcommand{\bm}[1]{\mathbbm{#1}}

\def\PP{\scr P}
\def\B{\mathbbm 1}
\def\R{\bb R}
\def\Z{\bb Z}
\def\P{\scr P}
\def\bp{\bb P}

\newcommand{\dd}{\mathrm{d}}

\newcommand{\supp}{\rm{supp}}

\newcommand{\mucon}[1]{\mu^{*#1}}
\newcommand{\bfh}[1]{\bf{h}_{#1}}

\def\slr{G}
\def\slrn{\rm{SL}_{\rank+1}(\R)}
\def\sltwo{\rm{SL}_2(\R)}
\def\gltwo{\rm{GL}_2(\R)}

\def\sigmah {\sigmahh^n}
\def\sigmahh {Y}
\def\tisig{\tilde{\sigmahh}^n}
\def\fren {\tau}
\def\l{\langle }
\def\r{\rangle }
\def\sg{\rm{m}}
\def\sign{\rm{sg}}
\def\diag{\rm{diag}}
\def\fgeven{f_{\bf g}}
\def\bfgh{\bf g\!\leftrightarrow\!\bf h}
\def\tbfgh{T\bfgh}
\def\expec{\beta}
\def\epss{\epsilon}
\def\rank{m}
\def\mollifier{\tau}
\def\fini{\Gamma}
\def\ad{\rm{Ad}}
\def\g{g}
\def\gap{\gamma}
\def\nupione{{l_1}}
\def\nupitwo{{l_2}}
\def\sp{\rm{Span}}
\def\ceta{\bb C_\eta}

\def\constant{C_A}
\def\realpart{\varpi}

\def\open{J}
\def\k{z}
\def\up{\sharp}
\def\epsilonone{\epsilon_1}
\def\epsilonzer{\epsilon_0}
\def\epsilontwo{\epsilon_2}
\def\Cone{C_2}
\def\epsilonthr{\epsilon_3}
\def\epsilonfour{\epsilon_4}

\begin{document}
\bibliographystyle{alpha}
\title{\textbf{Fourier decay, Renewal theorem and Spectral gaps for random walks on split semisimple Lie groups }}
\author{Jialun LI}
\date{}
\maketitle


\begin{abstract}
	We establish an exponential error term for the renewal theorem in the context of products of random matrices, which is surprising compared with classical abelian cases. A key tool is the Fourier decay of the Furstenberg measures on the projective spaces, which is a higher dimensional generalization of a recent work of Bourgain-Dyatlov.
\\

	\textbf{Résumé}: On établit un terme d'erreur exponentiel dans le théorème de renouvellement dans le cadre de produits de matrices aléatoires, qui est inattendu par rapport au cas classique abélien. L'outil clef est le décroissance de Fourier de mesures de Furstenberg sur les espaces projectifs, qui est une généralisation en dimension supérieure d'un travail récent de Bourgain-Dyatlov.
\end{abstract}

\section{Introduction}

Let $V$ be a nontrivial finite-dimensional irreducible real algebraic representation of a $\R$-split semisimple algebraic real Lie group $G$ (for example $G=\slrn$ and $V=\R^{\rank+1}$ with $\rank\geq 1$). 
Let $\mu$ be a Borel probability measure on $G$ and let $\Gamma_\mu$ be the subgroup generated by the support of $\mu$. We call $\mu$ Zariski dense if $\Gamma_\mu$ is a Zariski dense subgroup of $G$. This means that the measure $\mu$ does not concentrate on any proper algebraic subgroup of $G$. 
We also need the hypothesis of finite exponential moment. If $V$ is a faithful representation of $G$, the definition of exponential moment is that there exists $\epsilon$ positive such that 
\[\int_G\|g\|^\epsilon\dd\mu(g)<\infty, \]
where $\|g\|$ is the operator norm of $g$ acting on $V$.
For the general case, please see Definition \ref{defi:exponential moment}. From now on, we always suppose that the measure $\mu$ is Zariski dense with a finite exponential moment. 

For any natural number $n$, let $\mucon{n}$ be the $n$-times convolution of the measure $\mu$. Let $X_1,\cdots,X_n$ be i.i.d. random variables in $G$ with the same distribution $\mu$, then $\mucon{n}$ is the distribution of the product $X_1X_2\cdots X_n$. People are interested in generalising results about sum of random variables, such as the law of large numbers or the central limit theorem, to the norms or coefficients of products of random matrices.
The pioneers of the study of products of random matrices are Furstenberg, Kesten, Guivarc'h,.... 
For example, Furstenberg proved the law of large numbers, which states that there exists a Lyapunov constant $\sigma_{V,\mu}>0$ such that almost surely
$$\lim_{n\rightarrow\infty } \frac{1}{n}\log\|X_1\cdots X_n\|=\sigma_{V,\mu}.$$

\subsection*{Renewal theorem}
Our first result is about exponential error term in the renewal theorem.
The renewal theorem was first introduced in sum of random variables by Blackwell and in products of random matrices by Kesten \cite{kesten1973random} \cite{kesten1974renewal}, where he applied the renewal theorem to study the solution of random differential equations. The renewal theorem has also applications outside probability theory, for example its application to the decay of correlation by Sarig \cite{sarig} and its application to the asymptotic analysis of certain counting functions arising in the geometry of discrete groups by Lalley \cite{lalley1989renewal}.

Let $\|\cdot \|$ be a good norm\footnote{When $G=\slrn$ and $V=\R^{\rank+1}$, any euclidean norm on $\R^{\rank+1}$ is a good norm. For the definition, please see Definition \ref{defi:good norm}.} on $V$. Recall that for $g$ in $G$, we define $\|g\|$ to be its operator norm on $V$. For a compactly supported continuous function $f$ on $\bb R$ and a real number $t$, we define the renewal sum for norms by
\begin{align*}
R_Pf(t):=\sum_{n=0}^{+\infty}\int_Gf(\log\|g\|-t)\dd\mu^{*n}(g).
\end{align*}
If $f$ is positive, this sum is obviously well defined. In fact, using the positiveness of the Lyapunov constant and the Large deviation principle, we can show that this sum is always well defined.

Let $X=\bp V$ be the real projective space of $V$, which is the set of lines of $V$. Then we have a group action of $G$ on $X$. 
We define the cocycle function $\sigma:G\times X\rightarrow \R$ by, for $x=\R v$ in $X$ and $g$ in $G$, 
\begin{equation}\label{coc}
\sigma(g,x)=\log\frac{\|gv\|}{\|v\|}.
\end{equation}
For a compactly supported continuous $f$ on $\bb R$, the renewal sum for cocycles is defined by
\[Rf(x,t):=\sum_{n=0}^{+\infty}\int_G f(\sigma(g,x)-t)\dd\mucon{n}(g), \text{ for }x\in X \text{ and }t\in\R. \]
The limit law for norms and the limit law for cocycles are closely related.

The renewal theorem gives us a phenomenon of equidistribution when the time $t$ is large enough. The main result (due to Guivarc'h and Le Page \cite{guivarchlepage2015}) is that for a compactly supported continuous function $f$ when the time $t$ tends to infinite, the renewal sum $Rf(x,t)$ tends to $\frac{1}{\sigma_{V,\mu}}\int f$, where $\sigma_{V,\mu}$ is the Lyapunov constant.

Here is one of our main results.
\begin{thm}\label{thm:renewal}	Let $\bf G$ be a connected semisimple algebraic group defined and split over $\R$ and let $G=\bf G(\R)$ be its group of real points.
	Let $\mu$ be a Zariski dense Borel probability measure on $\slr$ with a finite exponential moment. Let $V$ be a nontrivial finite-dimensional irreducible real algebraic representation of $\slr$ with a good norm.
	There exists $\epsilon>0$ such that for $f\in C_c^{2}(\bb R)$ and $t\in\R$, we have
	\begin{equation*}
	Rf(x,t)=\frac{1}{\sigma_{V,\mu}}\int_{-t}^{\infty}f(u)\dd Leb(u)+O_f( e^{-\epsilon|t|}),
	\end{equation*}
	and
	\begin{equation*}
	R_Pf(t)=\frac{1}{\sigma_{V,\mu}}\int_{-t}^{\infty}f(u)\dd Leb(u)+O_f( e^{-\epsilon|t|}),
	\end{equation*}
	where $O_f$ depends on the support and some Sobolev norm of $f$.
\end{thm}
\begin{rem}
We should compare this result with the renewal theorem on $\R$ (the classical abelian case). If $\lambda$ is a measure on $\R$ whose support is finite, then the error term in the renewal theorem for the real random walk induced by $\lambda$ is never exponential. So Theorem \ref{thm:renewal} says that the behaviour of products of random matrices is more regular.

One application of Theorem \ref{thm:renewal} is the Fourier decay of self-affine measures in \cite{li2019affine}, where the exponential error term is crucially used to obtain a polynomial rate in the decay of Fourier transform. 

Our result is an improvement of a result of Boyer \cite{boyer2016renewalrd}, where the error term is polynomial on $t$.
\end{rem}

 We hope this type of result and its idea of proof will be helpful to obtain some exponential error terms in the orbital counting problems of higher rank. For instance in \cite{lalley1989renewal}, \cite{quint2005groupes} and \cite{sambarino2015hyperconvex}, they are interested in the asymptotic growth of $\#\{\gamma\in\Gamma|\ d(\gamma o,o)\leq R \}$, where $o$ is the base point in $\slrn/\rm{SO}(\rank+1)$ and $\Gamma$ is a discrete subgroup of $\slrn$. In \cite{GMR19}, they relate the number of integer solutions of Markoff-Hurwitz equations to a counting problem of discrete subsemigroup of $\slrn$, whose limit set is known as the Rauzy gasket.

This type of error term is usually proved using some spectral gap property. 

\subsection*{Spectral gap}
Equip $\bp V$ with a Riemannian distance. For $\gamma$ positive, let $C^\gamma(\bp V)$ be the space of $\gamma$-H\"older functions on $\bp V$. We introduce the transfer operator, which is an analogue of the characteristic function for real random variables.
\begin{defi*}
	For $z$ in $\bb C$ with the real part $|\Re z|$ small enough, let $P_{z}$ be the operator on the space of continuous functions on $\bp V$, which is given by
	\[P_{z}f(x)=\int_G e^{z\sigma(g,x)}f(gx)\dd\mu(g), \text{ for }x\in\bp V, \]
	where the cocycle $\sigma(g,x)$ is defined in \eqref{coc}.
\end{defi*}
We keep the assumption that $\mu$ is a Zariski dense Borel probability measure on $\slr$ with a finite exponential moment. The use of this transfer operator on the products of random matrices has been introduced by Guivarc'h and Le Page. Due to the property of exponential moment, when $|\Re z|$ is small enough, the operator $P_z$ preserves the Banach space $C^\gamma(\bp V)$ for $\gamma>0$ small enough. Due to the contracting action of $G$ on $X$, for $z$ in a small ball centred at $0$, the spectral radius of $P_z$ on $C^\gamma(\bp V)$ is less than $1$ except at $0$. Due to the non-arithmeticity of $\Gamma_\mu$, on the imaginary line, the operator $P_z$ also has spectral radius less than $1$ except at $0$. These were used to give limit theorems for products of random matrices by Guivarc'h and Le Page (Please see \cite{lepage_limite_1982} and \cite{benoistquint}). We will prove a uniform spectral gap of the transfer operator with parameter $z$ near the imaginary line.
\begin{thm}\label{thm:spegaprep}
Under the same assumptions as in Theorem \ref{thm:renewal},
	for every $\gamma>0$ small enough, there exists $\delta>0$ such that for all $|b|>1$ and $|a|$ small enough the spectral radius of $P_{a+ib}$ acting on $C^{\gamma}(\bb PV)$ satisfies
	\[\rho(P_{a+ib})<1-\delta.\]
\end{thm}
Even in the case $\rm{SL}_2(\R)$, the result is new and only known in some special cases; when $\mu$ is supported on a finite number of elements of $\rm{SL}_2(\R)$ and these elements generate a Schottky semigroup, this result is due to Naud \cite{naud2005expanding}. When $\mu$ is absolutely continuous with respect to the Haar measure on $\sltwo$, this result can be obtained directly using high oscillations. 

This result should be compared with similar results for random walks on $\bb R$. Let $\lambda$ be a Borel probability measure on $\bb R$ with finite support. Then
\[\liminf_{|b|\rightarrow \infty}|1-\hat \lambda(ib)|=0, \]
which is totally different from our case and where $\hat{\lambda}(z)$ is the Laplace transform of the measure $\lambda$, given by 
\[\hat{\lambda}(z)=\int_{\R} e^{zx}\dd\lambda(x). \]
The proof is direct. Let $\{x_1,\dots,x_l\}$ be the support of $\lambda$. Then $\hat{\lambda}(ib)=\sum_{1\leq j\leq l}\lambda(x_j)e^{ibx_j}$, and we only need to find $b$ such that all the terms are uniformly near $1$. Using the fact that $\liminf_{b\rightarrow \infty}d_{\bb R^l}(b(x_1,\dots, x_l),2\pi\bb Z^l)=0$, we have the claim. 

An analogous result is valid if we replace the projective space $\bp V$ by the flag variety $\P$. Let $\PP$ be the full flag variety of $\slr$ and let $\frak{a}$ be a Cartan subspace of the Lie algebra $\frak g$ of $G$. For $g\in G$ and $\eta\in\P $, let $\sigma(g,\eta)$ be the Iwasawa cocycle, which takes values in $\frak a$. We fix a Riemannian distance on $\P$. We can similarly define the space of $\gamma$-H\"older functions $C^\gamma(\P)$. Let $\realpart, \vartheta$ be in $\frak a^*$. For a continuous function $f$ on $\P$ and $|\realpart|$ small enough, the transfer operator $P_{\realpart+i\vartheta}$ on the flag variety is defined by
\[P_{\realpart+i\vartheta}f(\eta)=\int_G e^{(\realpart+i\vartheta)\sigma(g,\eta)}f(g\eta)\dd\mu(g). \]
We will prove that
\begin{thm}[Spectral gap]\label{thm:spegap}
	Let $\bf G$ be a connected semisimple algebraic group defined and split over $\R$ and let $G=\bf G(\R)$ be its group of real points.
	Let $\mu$ be a Zariski dense Borel probability measure on $\slr$ with finite exponential moment. 
	For every $\gamma>0$ small enough, there exists $\delta>0$ such that for all $\vartheta$, $\realpart$ in $\frak a^*$ with $|\vartheta|>1$ and $|\realpart|$ small enough the spectral radius of $P_{\realpart+i\vartheta}$ acting on $C^{\gamma}(\PP)$ satisfies
	\[\rho(P_{\realpart+i\vartheta})<1-\delta.\]
\end{thm}
This is a higher dimensional generalisation of Theorem \ref{thm:spegaprep}, because the parameter $\vartheta$ is in $\frak a^*$, whose dimension is equal to the real rank of $G$. It seems possible that this result implies a version of local limit theorem with exponentially shrinking targets in the Cartan subspace $\frak a$, similar to the work of Petkov and Stoyanov \cite{PS12}, where they used the spectral gap of transfer operator to obtain an asymptotic counting of lengths of closed geodesics lying in an exponentially shrinking targets.
This will be done in a joint work in progress with Cagri Sert.

\subsection*{Fourier decay}

The key ingredient in the proof of Theorem \ref{thm:spegaprep} and \ref{thm:spegap} is the following Fourier decay property of the $\mu$-stationary measure on the flag variety $\P$. 

On the space $X=\P$, a Borel probability measure $\nu$ is $\mu$-stationary if we have $$\nu=\mu*\nu:=\int_G g_*\nu\dd\mu(g),$$ where $g_*\nu$ is the pushforward of $\nu$ by the action of $g$ on $X$. By a theorem of Furstenberg, under Zariski dense condition there is a unique $\mu$-stationary probability measure $\nu$ on $X$ and it is also called the Furstenberg measure. This measure was introduced by Furstenberg when he established the law of large numbers for products of random matrices. The properties of the $\mu$-stationary measure are also important in other limit theorems for products of random matrices.

For a $\gamma$-H\"older function $f$ on $X$, we define $c_\gamma(f)=\sup_{x\neq x'}\frac{|f(x)-f(x')|}{d(x,x')^\gamma}$. We start with the case $G=\sltwo$. 
\begin{thm}\label{thm:foutwo}
	Let $\mu$ be a Zariski dense Borel probability measure on $\sltwo$ with a finite exponential moment. Let $X=\bb P(\bb R^2)$ and let $\nu$ be the $\mu$-stationary measure on $X$.
	
	For every $\gamma>0$, there exist $\epsilonzer > 0, \epsilonone > 0$ depending on $\mu$ such that the following holds. For $\xi>0$ large enough and any pair of real functions $\varphi\in C^2(X)$, $r\in C^\gamma(X)$ such that $|\varphi'|\geq \xi^{-\epsilonzer }$ on the support of $r$, $\|r\|_\infty\leq 1$ and 
	\begin{equation*}
	\|\varphi\|_{C^2}+c_\gamma(r) \leq \xi^{\epsilonzer },
	\end{equation*}
	then 
	\begin{equation*}
	\left|\int e^{i\xi \varphi(x)}r(x)\dd\nu(x)\right|\leq \xi^{-\epsilonone }.
	\end{equation*}
\end{thm}

\begin{rem}
	As a corollary of Theorem \ref{thm:foutwo}, we obtain a polynomial decay of the Fourier coefficients of $\nu$, that is
	\[ |\hat{ \nu}(k)|=O(|k|^{-\epsilonone}).\]  
	This generalizes the recent work of Bourgain-Dyatlov \cite{bourgain2017fourier} on Patterson-Sullivan measures and a work of myself in \cite{li2017fourier}, where a qualitative version is proved. For more background on the decay of Fourier coefficients see \cite{li2017fourier}.
\end{rem}

Theorem \ref{thm:foutwo} is a particular case of a more general result: Theorem \ref{thm:foudec} below.
In order to state the Fourier decay on the flag variety, we need to introduce a special condition. Let $r$ be a continuous function on $\P$ and let $C>1$. For a $C^2$ function $\varphi$ on $\P$, we say that $\varphi$ is $(C,r)$ good if it satisfies some assumptions on the Lipschitz norm and derivatives, which will be defined later (Definition \ref{defi:C r good}). 
Due to some technical problem, we will only prove a simply connected case in Section \ref{sec:proof}. (For example the group $\bf{SL}_{\rank+1}$ is simply connected but $\bf{PGL}_{\rank+1}$ is not.) The general case will be proved in Appendix \ref{sec:semisimple} by a covering argument.
\begin{thm}[Fourier decay]\label{thm:foudec}
		Let $\bf G$ be a connected $\R$-split reductive $\R$-group whose semisimple part is simply connected and let $G=\bf G(\R)$ be its group of real points.
		Let $\mu$ be Zariski dense Borel probability measure on $\slr$ with finite exponential moment. Let $\nu$ be the $\mu$-stationary measure on the flag variety $\P$.
	
	For every $\gamma>0$, there exist $\epsilonzer > 0,\epsilonone >0$ depending on $\mu$ such that the following holds. For  $\xi>0$ large enough and any pair of real functions $\varphi\in C^2(\PP)$, $r\in C^\gamma(\PP)$ such that $\varphi$ is $(\xi^{\epsilonzer },r)$ good, $\|r\|_\infty\leq 1$ and $c_\gamma(r) \leq \xi^{\epsilonzer }$,
	then 
	\begin{equation}\label{equ:fourier decay}
	\left|\int_\P e^{i\xi \varphi(\eta)}r(\eta)\dd\nu(\eta)\right|\leq \xi^{-\epsilonone }.
	\end{equation}
\end{thm}
\begin{rem}
	The decay rate only depends on the constants in the large deviation principles and the regularity of stationary measures. This should be compared with \cite{bourgain2017fourier}, where the spectral gap and the decay rate only depend on the dimension of the Patterson-Sullivan measure.
	
	A similar Fourier decay for the Lie group $\rm{SL}_2(\bb C)$ is established in \cite{lnpsl2c} for Patterson-Sullivan measures, which cannot be treated by our method due to the non splitness of $\rm{SL}_2(\bb C)$. It would also be interesting to establish a similar Fourier decay for the group $\rm{SL}_2(\bb Q_p)$ and the stationary measure on $\bp^1_{\bb Q_p}$.
\end{rem}
When $G=\rm{SL}_2(\R)$, the $(C,r)$ goodness is exactly the assumption of $\varphi$ in Theorem \ref{thm:foutwo}, which is natural for having a Fourier decay. Theorem \ref{thm:foudec} clearly implies Theorem \ref{thm:foutwo}.
%


It is interesting that the three objects, the polynomial rate in Fourier decay, the exponential error term in the renewal theorem and the spectral gap are roughly equivalent. In \cite{li2017fourier}, we use the renewal theorem to prove the Fourier decay. But in this manuscript, we use the polynomial rate in Fourier decay to prove the spectral gap, and then use the spectral gap to prove the exponential error term in the renewal theorem. In a highly related setting, convex cocompact surfaces, we can compare our three objects with more geometric objects. The Fourier decay was recently studied by Bourgain-Dyatlov; the spectral gap can be interpreted as the zero free region of the Selberg zeta function or the gap of the eigenvalues of the Laplace operator on the surface; the renewal theorem is replaced by the counting problem of the lattice points or the primitive closed geodesics. For the relation between these three objects of convex cocompact surfaces, see Borthwick \cite{borthwick2007spectral} and the references there.

\subsection*{Paper organization}

In Section \ref{sec:lie groups}, we will study the action of the group on the tangent bundle of the flag variety. In our higher rank case, the action is not conformal and we need to understand the contraction or the dilatation rate in different directions, which is a difficulty compared with rank one case. We will also recall basic properties of random walks on Lie groups.

Section \ref{sec:noncon} is devoted to the study of non-concentration condition, which is the main input for the use of discretized sum-product estimates (Proposition \ref{prop:sum-product}, a generalized version of a result of Bourgain \cite{bourgain2010discretized}). We need to verify certain measures on $\R^\rank$ are not concentrated on any affine subspaces of $\R^\rank$. The key idea is to use linear algebra to transfer the problem to an estimate of volume, which gives non-concentration for all affine subspaces simultaneously (Corollary \ref{cor:3.6}). Then we apply representation theory and the Guivarc'h regularity to establish the non-concentration condition.

In Section \ref{sec:proof}, our main results are proved.
The proof of Theorem \ref{thm:foudec} follows the similar strategy as in \cite{bourgain2017fourier}, but we are in higher rank case and we need to use almost all the preparations in previous sections. Then we use some idea of Dolgopyat to derive Theorem \ref{thm:spegaprep} and \ref{thm:spegap} from Theorem \ref{thm:foudec}. Finally, Theorem \ref{thm:renewal} will be obtained by using Theorem \ref{thm:spegaprep} rather straight forward.

In Appendix, we explain how to obtain the Fourier decay for semisimple groups from Theorem \ref{thm:foudec}, which only holds for groups with simply connected semisimple part.


\subsection*{Notation}
We will make use of some classical notation: for two real functions $f$ and $g$, we write $f=O(g), f\ll g$ or $g\gg f$ if there exists a constant $C>0$ such that $|f|\leq Cg$, where $C$ only depends on the ambient group $G$ and the measure $\mu$. We write $f\asymp g$ if $f\ll g$ and $g\ll f$.  We write $f=O_\epsilon(g), f\ll_\epsilon g$ or $g\gg_\epsilon f$ if the constant $C$ depends on an extra parameter $\epsilon>0$.

We always use $0<\delta<1$ to denote an error term and $0<\beta<1$ to denote the magnitude. We will use $\delta^{C_0}$ with $C_0$ only depends the group $G$ and measure $\mu$, this constant $C_0$ may change from line to line, while still being denoted the same.
If $\delta^{C_0}f\leq g\leq \delta^{-C_0}f$, then we say that $f$ and $g$ are of the same size. 

\subsection*{Acknowledgement}
This is part of the author’s Ph.D. thesis, written under the supervision
of Jean-François Quint at the University of Bordeaux. The author gratefully acknowledges the many helpful suggestions and stimulating conversations of Jean-François Quint during the preparation of the article.

The author would also like to thank the referee for many helpful comments and corrections which greatly improve the readability of this article.

\section{Random walks on Reductive groups}
\label{sec:lie groups}
The representation theory of algebraic groups is more clear than the representation theory of Lie groups. We will use the vocabulary of algebraic groups. In this manuscript, we always assume that $\bf G$ is a connected $\R$-split reductive $\R$-group. We are interested in two particular cases:
\begin{itemize}
	\item the semisimple part is simply connected in algebraic group sense,
	\item the group $\bf G$ is actually semisimple.
\end{itemize}
If we need these further assumptions, we will say it at the beginning of the section or in the statement.
Please see \cite{helgason1979differential}, \cite{borel1990linear} and \cite{benoistquint} for more details. 

We write $\bf G$ for an algebraic group, and $G=\bf G(\R)$ for its group of real points, equipped with the Lie group topology (analytical topology). All the representations are nontrivial finite-dimensional real algebraic and with a good norm.
\subsection{Reductive groups and representations}
\subsubsection*{Reductive groups}
Let $\bf G$ be a connected $\R$-split reductive $\R$-group. Let $\bf A$ be a maximal $\R$-split torus in $\bf G$. Because $\bf G$ is $\R$-split, the group $\bf A$ is also the maximal torus of $\bf G$ and the centralizer of $\bf A$ in $\bf G$ is $\bf A$. Let $\bf C$ be the connected component of the centre of $\bf G$, which is contained in the maximal torus $\bf A$. The semisimple part of $\bf G$ is the derived group $\scr D\bf G=[\bf G,\bf G]$. \nomentry{$\scr D\bf G$}{}
Let $\bf B$ be the subtorus of $\bf A$ given by $\bf A\cap \scr D\bf G$. The dimension of $\bf A$ and $\bf B$ are called the reductive rank and the semisimple rank of $\bf G$, respectively. We write $r$ and $\rank$ for the reductive rank and the semisimple rank.

Because we are dealing with real groups, we will use transcendental methods to describe the structure of $\bf G$. Let $G,A,B$ and $C$ be the group of real points of $\bf G,\bf A,\bf B$ and $\bf C$. Let $\theta$ be a Cartan involution of $G$ which satisfies $\theta(A)=A$ and such that the set of fixed points $K=\{g\in G |\ \theta(g)=g \}$ is a maximal compact subgroup of $G$. \nomentry{$\theta$}{}
 Let $\frak g, \frak k$, $\frak a,\frak b$ and $\frak c$ be the Lie algebra of $G, K, A, B$ and $C$, respectively. Then $\frak a=\frak b\oplus\frak c$ due to $\frak g=\scr D\frak g\oplus\frak c$. We write $\exp$ for the exponential map from $\frak a$ to $A$. 
We also write $\theta$ for the differential of the Cartan involution, whose set of fixed points is $\frak k$ and which equals $-id$ on $\frak a$.

For $X,Y$ in $\frak g$, the Killing form is defined as
\[K(X,Y)=tr(\rm{ad}X \rm{ad}Y).  \] 
The Killing form is positive definite on $\frak b$ and negative definite on $\frak k$. Endowed with the Killing form, the Lie algebra $\frak b$ and its dual $\frak b^*$ become Euclidean spaces.

\subsubsection*{Root systems and the Weyl group}
The spaces $\frak b^*$ and $\frak c^*$ are seen as subspaces of $\frak a^*$, which takes value zero on $\frak c$ and $\frak b$, respectively.
Let $R$ be the root system of $\frak g$ with respect to $\frak a$, that is the set of nontrivial weights of the adjoint action of $\frak a$ on $\frak g$. It is actually a subset of $\frak b^*$. Because $\frak c$ is in the centre of $\frak g$, its adjoint action on $\frak g$ is trivial. Fix a choice of positive roots $R^+$. Let $\Pi$ be the collection of primitive simple roots of $R^+$. Let $\frak a^+$ be the Weyl chamber defined by 
$\{X\in\frak a|\alpha(X)\geq 0,\ \forall \alpha\in\Pi \}$.
Let $\frak a^{++}$ be the interior of the Weyl chamber defined by
 $\{X\in\frak a|\alpha(X)> 0,\ \forall \alpha\in\Pi \}$.
Using the root system, we have a decomposition of $\frak g$ into eigenspaces of $\frak a$,
\[\frak g=\frak z\oplus\bigoplus_{\alpha\in R}\frak g^{\alpha}, \]
where $\frak z$ is the centralizer of $\frak a$ and $\frak g^{\alpha}$ is the eigenspace given by 
\[ \frak g^{\alpha}=\{X\in \frak g|\ [Y,X]=\alpha(Y)X \text{ for all }Y\in \frak a  \}. \]
Since the group $\bf G$ is split, we know that $\frak a=\frak z$ and that $\frak g^\alpha$ are of dimension 1.

Recall that for every root $\alpha$ in $R$, there is an orthogonal symmetry $s_\alpha$ which preserves $R$ and $s_\alpha(\alpha)=-\alpha$. For $\alpha\in R$, let $H_\alpha$ be the unique element in $\frak b$ such that $s_\alpha(\alpha')=\alpha'-\alpha'(H_\alpha)\alpha$ for $\alpha'\in\frak b^*$. The set $\{H_\alpha|\ \alpha\in R \}$ is called the set of dual roots in $\frak b$. Since the Cartan involution $\theta$ equals $-id$ on $\frak a$, this implies $\theta\frak g^\alpha=\frak g^{-\alpha}$ for $\alpha\in R$. Using the Killing form, we can prove that $[\frak g^\alpha,\frak g^{-\alpha}]= \R H_\alpha$ (See \cite[Cha. 4, Theorem 2]{serre2012complex} for more details). Hence, there is a unique choice (up to sign) $X_\alpha\in \frak g^\alpha,\ Y_\alpha\in \frak g^{-\alpha}$ for $\alpha$ positive such that 
\[[X_\alpha,Y_\alpha]=H_\alpha\text{ and }\theta(X_\alpha)=-Y_\alpha. \]
Let $K_\alpha=X_\alpha-Y_\alpha$. Due to $\theta K_\alpha=K_\alpha$, the element $K_\alpha$ is in $\frak k$.

Let $W$ be the Weyl group of $R$. Then the group $W$ acts simply transitively on the set of Weyl chambers. Let $w_0$ be the unique element in $W$ which sends the Weyl chamber $\frak a^+$ to the Weyl chamber $-\frak a^+$. Let $\iota=-w_0$ be the opposition involution. The Weyl group also acts on $\frak a^*$ by the dual action. Let $N_G(A)$ be the normalizer of $A$ in $G$. An element in $N_G(A)/A$ induces an automorphism on the tangent space $\frak a$. This gives an isomorphism from $N_G(A)/A$ to the Weyl group $W$. Hence $w_0$ can be realized as an element in $G/A$ and its action on $\frak a$ is given by conjugation.

\subsubsection*{The Iwasawa cocycle}
Let $\frak n=\oplus_{\alpha\in R^+ }\frak g^\alpha$ and $\frak n^-=\oplus_{\alpha\in R^+}\frak g^{-\alpha} $. They are nilpotent Lie algebras. Let $\bf N$ be the connected algebraic subgroup of $\bf G$ with Lie algebra $\frak n$. The group $\bf N$ is normalized by $\bf A$. Let $\bf P=\bf A\ltimes\bf N$ be a minimal parabolic subgroup. The flag variety $\P$ is defined to be the set of conjugacy classes of $P$ under the action of $G$. Since the normalizer of $P$ in $G$ is itself, we obtain an isomorphism
\[ G/P\rightarrow \P. \]
We write $\eta_o$ for the subgroup $P$ seen as a point in $\P$. Let $M$ be the subgroup of $A$, whose elements have order at most two. Since $A$ is isomorphic to $(\R^*)^r$, we know that $M\simeq (\Z/2\Z)^r$ and $A=M\times A_e$, where $A_e=\exp(\frak a)$ is the analytical connected component of $A$ and $A_e\simeq(\R_{>0})^r$. 

For a real Lie group $L$, let $L^o$ be the analytical connected component of the identity element in $L$.
\begin{lem}\label{lem:km} We have $K=K^o M$.
\end{lem}
\begin{proof}
	On the one hand, by Matsumoto's theorem \cite{matsumoto64gpreel} (\cite[Théorème 14.4]{boreltits65reductif}), we have $G=G^o A=G^o M$. Hence the group $M$ intersects each connected component of $G$. 
	The set $K^o M$ intersects each connected component of $G$ and its intersection with $G^o$ contains $K^o$. By definition of $K$, we know that $K\subset K^o M$.
	
	On the other hand, we know that $K\supset K^o M$, because the group $M$ equals to $A\cap K$ due to \cite[Lemme 4.2]{benoist2005convexes}. This lemma can also be proved directly by considering the action of the Cartan involution on $M$. The proof is complete.
\end{proof}
We have an Iwasawa decomposition of $G$ given by 
\[G=KAN. \]
The action of $K$ on $\P$ is transitive. Hence $\P$ is a compact manifold. 
By Lemma \ref{lem:km}, we have $G=KA_eMN=KA_eN$.
This is a bijection between $G$ and $K\times A_e\times N$. Then we can define the Iwasawa cocycle $\sigma$ from $G\times \P$ to $\frak a$. Let $\eta$ be in $\P$ and $g$ be in $G$. By the transitivity of $K$, there exists $k\in K$ such that $\eta=k\eta_o$. By the Iwasawa decomposition, there exists a unique element $\sigma(g,\eta)$ in $\frak a$ such that
\[gk\in K\exp(\sigma(g,\eta))N. \]
We can verify that this is well defined and $\sigma$ is an additive cocycle, that is for $g,h$ in $G$ and $\eta$ in $\P$
\[\sigma(gh,\eta)=\sigma(g,h\eta)+\sigma(h,\eta). \]

Due to the direct sum $\frak a=\frak b\oplus\frak c$, we can decompose the Iwasawa cocycle into the semisimple part and the central part of the cocycle, that is
\begin{equation*}
	\sigma(g,\eta)=\sigma_{ss}(g,\eta)+c(g),
\end{equation*}
where $\sigma_{ss}$ is in $\frak b$ and $c(g)$ in $\frak c$. The central part $c(g)$ does not depend on $\eta$, because the map
\[G\rightarrow G/\scr D G \]
kills the semisimple part and the restriction of this map to $C_e=\exp(\frak c)$ is injective. Moreover, since the Iwasawa cocycle is additive, the central part is also additive. That is for $g,h$ in $G$
\[c(gh)=c(g)+c(h). \]
\subsubsection*{The Cartan decomposition}
The Cartan decomposition says that $G=KA^+K$, where $A^+$ is the image of the Weyl chamber $\frak a^+$ under the exponential map. For $g$ in $G$, by Cartan decomposition, we can write $g=k_ga_g\ell_g$ with $k_g,\ell_g$ in $K$ and $a_g$ in $A^+$. The element $a_g$ is unique and there is a unique element $\kappa(g)$ in $\frak a^+$ such that $a_g=\exp(\kappa(g))$. We call $\kappa(g)$ the Cartan projection of $g$. Then $\kappa(g^{-1})=\iota\kappa(g)$, where $\iota$ is the opposition involution. Since $A$ is contained in $P$, we can define $\zeta_o=w_0\eta_o$, where the element $w_0$ in the Weyl group is seen as an element in $G/A$ (As an element in $\P$, $\zeta_o$ is the opposite parabolic group with respect to $P$ and $A$). Let $\eta^M_g=k_g\eta_o$
\nomentry{$\eta^M_g$}{}
 and $\zeta^m_g=\ell_g^{-1}\zeta_o$.
 \nomentry{$\zeta^m_g$}{}
When $\kappa(g)$ is in $\frak a^{++}$, it is uniquely defined, independently of the choice of $k_g$ and $\ell_g$. 

We can also define a unique decomposition of $\kappa(g)$ into semisimple part and central part.
Due to $\kappa(g)=\sigma(g,\ell^{-1}_g\eta_o)=\sigma_{ss}(g,\ell^{-1}_g\eta_o)+c(g)$, we have 
\[\kappa(g)=\kappa_{ss}(g)+c(g). \]

\subsubsection*{Dominant weights}

Let $\rm X(\bf A)$ and $\rm X(\bf B)$ be the character groups of $\bf A$ and $\bf B$, respectively. We will identify $\rm X(\bf A)$ and $\rm X(\bf B)$ as discrete subgroups of $\frak a^*$ and $\frak b^*$ by taking differential. The elements of $\frak a^*$ in $\rm X(\bf A)$ are called weights. All the roots are weights, because they come from adjoint action of $A$ on $\frak g^\alpha$.

Since $\{H_\alpha\}_{\alpha\in\Pi}$ is a basis of $\frak b$, let $\{\tilde\omega_\alpha\}_{\alpha\in\Pi}$ be the dual basis, they are called the fundamental weights. 

\textbf{For the general case}: we only know that $X(\bf B)$ is a finite index subgroup of $\oplus_{\alpha\in\Pi}\Z\tilde\omega_\alpha$. Let 
\begin{equation*}
\tilde\chi_{\alpha}=n_\alpha\tilde\omega_\alpha, 
\end{equation*}
where $n_\alpha$ is the smallest positive natural number such that $n_\alpha\tilde\omega_\alpha$ is in $X(\bf B)$. Since $\bf B$ is a closed subgroup of a split torus $\bf A$, every character on $\bf B$ extends to a character on $\bf A$. Let $\chi_\alpha$ be an extension of $\tilde\chi_{\alpha}$ to $\bf A$. We fix this choice. For semisimple case, we know that $\chi_{\alpha}=\tilde{\chi}_\alpha$.

\textbf{If the derived group $\scr D\bf G$ is simply connected}, we have
\begin{equation*}
	\rm X(\bf B)=\oplus_{\alpha\in\Pi}\Z\tilde\omega_\alpha.
\end{equation*}
Hence for $\alpha\in\Pi$, we have
\begin{equation*}
\chi_\alpha\in\rm X(\bf A) \text{ and }\chi_\alpha|_{\frak b}=\tilde\omega_\alpha.
\end{equation*}
We write $\omega_\alpha$ for the element in $\frak a^*$ which is another extension of $\tilde\omega_\alpha$ and vanishes on $\frak c$, that is
\begin{equation*}
	\omega_\alpha|_{\frak b}=\tilde\omega_\alpha\text{ and }\omega_\alpha|_{\frak c}=0.
\end{equation*}

Recall that a weight is a dominant weight, if for every $w$ in the Weyl group $W$, the difference $\chi-w(\chi)$ is a sum of positive roots. 
\begin{lem}\label{lem:dominant}
 For every $\alpha\in\Pi$, the weight $\chi_{\alpha}$ is a dominant weight.
\end{lem}
\begin{proof}
The action of the Weyl group on $\frak c^*$, the space of linear functionals which vanish on $\frak b$, is trivial. We know that
\[\chi_\alpha-w(\chi_\alpha)\in\frak b^*. \]
Because $\tilde\omega_\alpha$ is a fundamental weight, we have $\chi_\alpha-w(\chi_\alpha)|_{\frak b}=n_\alpha(\tilde{\omega}_\alpha-w(\tilde\omega_\alpha))$
equals a sum of positive roots. 
\end{proof}

\subsubsection*{Representations and highest weights}
Let $(\rho,V)$ be a representation of $G$. The set of restricted weights $\Sigma(\rho)$ of the representation is the set of elements $\omega$ in $\frak a^*$ such that the eigenspace
\[V^\omega=\{v\in V|\forall X\in\frak a,\ \dd\rho(X)v=\omega(X)v \} \]
is nonzero, where $\dd \rho$ is the tangent map of $\rho$ from $\frak g$ to $End(V)$. By definition, we see that $\omega$ is the differential of a character on $\bf A$, which is a weight. We define a partial order on the restricted weights: For $\omega_1,\omega_2$ in $\Sigma(\rho)$,
\[\omega_1\geq\omega_2 \Leftrightarrow \omega_1-\omega_2 \text{ is a sum of positive roots.} \]
If $\omega$ is in $\Sigma(\rho)$, then we say that $\omega$ is a weight of $V$ and a vector $v$ in $V^\omega$ is said to have weight $\omega$. We call $\rho$ proximal if there exists $\chi$ in $\Sigma(\rho)$ which is greater than the other weights and such that $V^\chi$ is of dimension 1. We should pay attention that a proximal representation is not supposed to be irreducible. An advantage of the splitness of $G$ is that all the irreducible representations are proximal, which will be extensively used later on.

Suppose that $(\rho,V)$ is an irreducible representation. Let $\chi\in\frak a^*$ be the highest weight of $(\rho,V)$. We write $V_{\chi,\eta}=\rho(g)V^{\chi}$
\nomentry{$V_{\chi,\eta}$}{}
for $\eta=g\eta_o$, which is well defined because the parabolic subgroup $P$ fixes the subspace $V^\chi$. This gives a map from $\P$ to $\bp V$ by
\begin{equation}\label{equ:ppV}
\P\rightarrow \bp V,\ \eta\mapsto V_{\chi,\eta}.
\end{equation}

In the case of split reductive groups, for a character $\chi$ on $\bf A$, there exists an irreducible algebraic representation with highest weight $\chi$ if and only if $\chi$ is a dominant weight \cite{tits_rep_1971}. Let 
$$\Theta_\rho=\{\alpha\in\Pi:\ \chi-\alpha\text{ is a weight of }\rho \}.$$
\nomentry{$\Theta_\rho$}{}By Lemma \ref{lem:dominant}, we have  
  \begin{lem}\label{lem:tits}
  Let $(\rho_\alpha,V_\alpha) _{\alpha\in\Pi}$ be a family of representations such that the highest weight of $\rho_\alpha$ is $\chi_\alpha$. Then we have $\Theta_{\rho_\alpha}=\{\alpha \}$ and
  	the product of the maps given by \eqref{equ:ppV}
  	\begin{align*}
  	\P\longrightarrow &\prod_{\alpha\in\Pi}\bb PV_\alpha,\ \ \eta\mapsto (V_{\chi_\alpha,\eta})_{\alpha\in\Pi},
  	\end{align*}
  	is an embedding of $\P$ into the product of projective spaces.
  \end{lem}
  \begin{lem}\label{lem:weight structure}
  	Let $(\rho,V)$ be an irreducible representation of $G$ with highest weight $\chi$. Then $\Theta_\rho=\{\alpha \}$ is equivalent to say that $\chi(H_\alpha)> 0$ for only one simple root $\alpha$.
  \end{lem}
  \begin{proof}
  	 Consider the representation of the Lie algebra $\frak s_\alpha=<H_\alpha,X_\alpha,Y_\alpha>$ on $v$ of highest weight. By the classification of the representations of $\frak {sl}_2$, we know that $Y_\alpha v\neq 0$ if and only if $\chi(H_\alpha)>0$. The vector $Y_\alpha v$ is the only way to obtain a vector of weight $\chi-\alpha$ by \cite[Chapter 7, Proposition 2]{serre2012complex}. The proof is complete.
  \end{proof}
  %

  \begin{defi}[Super proximal representation]\label{defi:super proximal}
  	Let $(\rho,V)$ be an irreducible representation of $G$ with highest weight $\chi$. We call $V$ super proximal if the exterior square $\wedge^2V$ is also proximal. This is equivalent to $\Theta_\rho=\{\alpha \}$, and $V^{\chi-\alpha}$ is of dimension 1 for some simple root $\alpha$. 
  \end{defi}
  \begin{lem}\label{lem:super proximal}
  	If the highest weight $\chi$ of an irreducible representation satisfies $\chi(H_\alpha)> 0$ for only one simple root $\alpha$, then this representation is super proximal.
  \end{lem}
  \begin{proof}
 Because the central part of $G$ preserves eigenspaces of $A$. It is also an irreducible representation of the semisimple part. It will be thus sufficient to prove the semisimple case.
  	
  	 Let $\alpha$ be the simple root. Let $v$ be a nonzero vector with highest weight $\chi$. By \cite[Chapter 7, Proposition 2]{serre2012complex}, the representation $V$ is generated by vectors $Y_{\beta_1}\cdots Y_{\beta_k}v$, where $\beta_1,\dots,\beta_k$ are positive roots. Hence a vector of weight $\chi-\alpha$ can only be obtained by $Y_\alpha v$. The dimension of $V^{\chi-\alpha}$ is no greater than 1. Since $\chi-\alpha$ is a weight due to Lemma \ref{lem:weight structure}, the proof is complete.
  \end{proof}

For $\chi\in\frak a^*$, if it is a weight, we will use $\chi^\up$ to denote its corresponding algebraic character in $\rm X(\bf A)$.
By the definition of eigenspace $V^\chi$, we have
\begin{lem}\label{lem:sign m}
	Let $(\rho,V)$ be an irreducible representation of $G$. Let $\chi^\up$ be an algebraic character of $A$. For $a$ in $A$ and $v\in V^{\chi}$, we have
	\[\rho(a)v=\chi^\up(a)v. \]
\end{lem}
This lemma will be used to determine the sign in Section \ref{sec:sign group}. 

\subsubsection*{Representations and good norms}
\begin{defi}\label{defi:good norm}
Let $\|\cdot \|$ be an euclidean norm on a representation $(\rho, V)$ of $G$. We call $\|\cdot\|$ a good norm if $\rho(A)$ is symmetric and $\rho(K)$ preserves the norm. 
\end{defi}
By \cite{helgason1979differential}, \cite[Lemma 6.33]{benoistquint}, good norms exist on every representation of $G$. One advantage of good norm is that for $v,u$ in $V$ and $g$ in $G$
\[\l\rho(g)v,u \r=\l v,\rho(\theta(g^{-1})u)\r, \]
where $\theta$ is the Cartan involution. The above equation is true because it is true for $g$ in $A$ and $K$. This means that for good norm we have 
\begin{equation}\label{equ:good norm}
^t\rho(g)=\rho(\theta(g^{-1})).
\end{equation}
The application \eqref{equ:ppV} enables us to get information on $\P$ from the representations. For an element $g$ in $GL(V)$, let $\|g\|$ be its operator norm. 
\begin{lem}\label{lem:flapro}
	Let $G$ be a connected reductive $\R$-group. Let $(\rho,V)$ be an irreducible linear representation of $G$ with good norm. Let $\chi$ be the highest weight of $V$. For $\eta$ in $\P$ and a non-zero vector $v\in V_{\chi,\eta} $, we have
	\begin{align}\label{equ:representation cocycle}
	\frac{\|\rho(g)v\|}{\|v\|}=\exp(\chi\sigma(g,\eta)),\\
	\label{equ:representation cartan}
	\|\rho (g)\|=\exp(\chi\kappa(g)).
	\end{align}
\end{lem}
Please see \cite[Lemma 8.17]{benoistquint} for the proof.

\subsubsection*{Examples}
	\label{exa:slr}
	For the group $\bf{GL}_{\rank+1}$, the maximal torus $A$ can be taken as the diagonal subgroup and the Lie algebra $\frak a$ is the set of diagonal matrices. The Lie algebra $\frak b$ is the subset of $\frak a$ with trace zero. For $X$ in $\frak a$, we write $X=\diag(x_1,\dots,x_{\rank+1})$ with $x_i\in \R$. The Lie algebra $\frak c=\{X\in\frak a|\ x_1=x_2=\cdots=x_{\rank+1}  \}$. 
	Let $\lambda_i$ in $\frak a^*$ be the linear map given by $\lambda_i(X)=x_i$. The root system $R$ is given by 
	\[R=\{\lambda_i-\lambda_j| i\neq j,\text{ and }i,j\in\{1,\dots,\rank+1\}  \}. \]
	A choice of positive roots is $\lambda_i-\lambda_j$ with $i<j$. The set of simple roots is $\Pi=\{\lambda_i-\lambda_{i+1}| i=1,\dots, \rank \}$. Let $\alpha_i=\lambda_i-\lambda_{i+1}$. The Weyl chamber is
	\[\frak a^+=\{X\in \frak a|x_1\geq x_2\geq \cdots\geq x_{\rank+1}  \}. \]
	The fundamental weights are $\tilde\omega_{\alpha_i}=\lambda_1+\cdots+\lambda_i$ for $i=1,\dots,\rank$ on $\frak b$. The weight $\chi_{\alpha_i}$ has the same form as $\tilde\omega_{\alpha_i}$ in $\frak a$. The weight $\omega_{\alpha_i}$ is equal to $\chi_{\alpha_i}-\frac{i}{\rank+1}(\lambda_1+\cdots+\lambda_{\rank+1})$. The representations $V_{\alpha_i}$ are given by $V_{\alpha_i}=\wedge^i\R^{\rank+1}$ for $i=1,\dots,\rank$. The maximal compact subgroup $K$ is $\rm{O}(\rank+1)$ and the parabolic group $P$ is the upper triangular subgroup and $N$ is the subgroup of $P$ with all the diagonal entries equal to $1$. The flag variety $\P$ is the set of all flags
	\[W_1\subset W_2\subset \cdots \subset W_{\rank}, \]
	where $W_i$ is a subspace of $\R^{\rank+1}$ of dimension $i$. 
	
	Let $\epsilon_{i,j}$ be the square matrix of dimension $\rank +1$ with the only nonzero entry at the $i$-th row and $j$-th column, which equals 1. The element $H_{\alpha_i}$ is $\epsilon_{i,i}-\epsilon_{i+1,i+1}$. The element $X_{\alpha_i},Y_{\alpha_i}$ are given by $\epsilon_{i,i+1}, \epsilon_{i+1,i}$. The Cartan involution $\theta$ is the additive inverse of the transpose, that is $\theta(X)=-^tX$ for $X$ in $\frak a$.
	
	The Weyl group $W$ is isomorphic to the symmetric group $\scr S_{\rank+1}$. The action on $\frak a$ is simply given by the permutation of coordinates and the element $w_0$ sends $X=\diag(x_1,\dots,x_{\rank+1})$ to $w_0X=\diag(x_{\rank+1},\dots,x_1)$.

\subsection{Linear actions on vector spaces}
\label{sec:linear action}
Let $V$ be a vector space with an euclidean norm. Then we have an induced norm on its dual space $V^*$, tensor products $\otimes^jV$ and exterior powers $\wedge^jV$. 

For $x=\R v,x'=\R v'$ in $\bp V$, we define the distance between $x,x'$ by 
\begin{equation}\label{equ:distance x x'}
d(x,x')=\frac{\|v\wedge v'\|}{\|v\|\|v'\|}.
\nomentry{$d(x,x')$}{}
\end{equation}
This distance has the advantage that it behaves well under the action of $GL(V)$. See for example Lemma \ref{lem:gBmg}.
For $y=\R f$ in $\bp V^*$, let $y^\perp=\bp(\ker f)\subset \bp V$ be a hyperplane in $\bp V$. For $x=\R v$ in $\bp V$, we define the distance of $x$ to $y^\perp$ by
\begin{equation*}
	\delta(x,y)=\frac{|f(v)|}{\|f\|\|v\|},
	\nomentry{$\delta(x,y$)}{}
\end{equation*} 
which is explained by $\delta(x,y)=d(x,y^\perp)=\min_{x'\in y^\perp}d(x,x')$. Let $K_V$ be the compact group of $GL(V)$ which preserves the norm. Let $A_V^+$ be the set of diagonal elements such that $\{a=\diag(a_1,\cdots, a_d)| a_1\geq a_2\geq\cdots\geq a_d \}$, under the basis $\{e_1,\cdots, e_d \}$. Let $A_V^{++}$ be the interior of $A_V^+$. For $g$ in $GL(V)$, by the Cartan decomposition we can choose 
\begin{equation}\label{equ:kgaglg}
g=k_ga_g\ell_g\text{, where }a_g\in A_V^+\text{ and }k_g,\ell_g\in K_V.
\end{equation}
Let $e_i^*$ be the dual basis.
Let $x^M_g=\R k_ge_1$
\nomentry{$x^M_g$}{}
 and $y^m_g=\R\,^t\ell_g e_1^*$
 \nomentry{$y^m_g$}{}
  be the density points of $g$ on $\bp V$ and $\,^t g$ on $\bp V^*$, which is unique and independent of the choice of basis when $a_g$ is in $A_V^{++}$.
For $r>0$ and $g$ in $GL(V)$, let
\begin{align*}
b^M_{V,g}(r)&=\{x\in\bb PV|d(x,x^M_g)\leq r \},\\ \nomentry{$b^M_{V,g}(r)$}{}
B^m_{V,g}(r)&=\{x\in\bb PV|\delta(x,y^m_g)\geq r \}.\nomentry{$B^m_{V,g}(r)$}{}
\end{align*}
These two sets play an important role when we want to get some ping-pong property. The elements in the set $B^m_{V,g}(r)$ have distance at least $r$ to the hyperplane determined by $y^m_g$. 

\subsubsection*{Distance and norm}\label{sec:distance}
We start with a general $g$ in $GL(V)$, where $V$ is a finite-dimensional vector space with euclidean norm. We need some technical distance control. These are quantitative versions of the same controls in \cite[Lemma 2.5, 4.3, 6.5]{quint2002mesures}.

For $g$ in $GL(V)$ and $x=\R v\in\bp V$, we define an additive cocycle $\sigma_V:GL(V)\times \bp V\rightarrow \R$ by
\begin{equation}\label{equ:sigma V}
\sigma_V(g,x)=\log\frac{\|gv\|}{\|v\|}.
\end{equation}
This is called a cocycle, because for $g,h$ in $G$ we have
\[\sigma_V(gh,x)=\sigma_V(g,hx)+\sigma_V(h,x). \]
\begin{lem}[Lemma 14.2 in \cite{benoistquint}]\label{lem:cocycle} For any $g$ in $GL(V)$ and $x$ in $\bb PV$, we have
	\begin{align}\label{equ:coccar}
	\delta(x,y^m_g)&\leq \frac{\|gv\|}{\|g\|\|v\|}\leq 1.
	\end{align}
\end{lem}

For $g$ in $GL(V)$, let $\gamma_{1,2}(g):=\frac{\|\wedge^2g\|}{\|g\|^2}$. \nomentry{$\gamma_{1,2}(g)$}{}

\begin{lem}\label{lem:gBmg}
	Let $\delta>0$. For $g$ in $GL(V)$, if $\expec=\gamma_{1,2}(g)\leq\delta^2$, then 
	\begin{itemize}
		\item the action of $g$ on $B^m_{V,g}(\delta)$ is $\expec\delta^{-2}$-Lipschitz and
		\[gB^m_{V,g}(\delta)\subset b^M_{V,g}(\expec\delta^{-1})\subset b^M_{V,g}(\delta), \]
		\item
		the restriction of the real valued function $\sigma_V(g,\cdot)$ on $B^m_{V,g}(\delta)$ is $2\delta^{-1}$-Lipschitz.
	\end{itemize}
\end{lem}
\begin{proof}
	Due to \cite[Lem 14.2]{benoistquint},
	\[d(gx,x^M_g)\delta(x,y^m_g)\leq\gamma_{1,2}(g)=\beta. \]
	Hence
	\[d(gx,x^M_g)\leq\beta\delta(x,y^m_g)^{-1}\leq \beta\delta^{-1}, \]
	which implies the inclusion.
	
	For $x=\R v$ and $x'=\R v'$ in $B^m_{V,g}(\delta)$, by \eqref{equ:coccar}, we have
	\[d(gx,gx')=\frac{\|gv\wedge gv'\|}{\|v\wedge v'\|}\frac{\|v\wedge v'\|}{\|v\|\|v'\|}\frac{\|v\|\|v'\|}{\|gv\|\|gv'\|}\leq\gamma_{1,2}(g)d(x,x')\delta^{-2} , \]
	which implies the Lipschitz property of $g$.
	
	For the Lipschitz property of $\sigma_V(g,\cdot)$, please see \cite[Lemma 17.11]{benoistquint}.
\end{proof}
For two different points $x=\R v$ and $x'=\R v'$ in $\bp V$, we write $x\wedge x'=\R (v\wedge v')\in \bp (\wedge^2 V)$. 
\begin{lem}\label{lem:distance} For any $g$ in $GL(V)$ and two different points $x=\bb Rv, x'=\bb R v'$ in $\bb PV$, we have
	\begin{align}\label{equ:gxgx'l}
	\gamma_{1,2}(g)\delta(x\wedge x',y^m_{\wedge^2g})&\leq \frac{d(gx,gx')}{d(x,x')}.
	\end{align}
\end{lem}
\begin{proof} By definition and \eqref{equ:coccar}, we have
	\begin{align*}
	d(gx,gx')&=\frac{\|gv\wedge gv'\|}{\|v\wedge v'\|}\frac{\|v\wedge v'\|}{\|v\|\|v'\|}\frac{\|v\|\|v'\|}{\|gv\|\|gv'\|}\geq\gamma_{1,2}(g)\delta(x\wedge x',y_{\wedge^2g}^m)d(x,x').
	\end{align*}
	The proof is complete.
\end{proof}

\subsection{Actions on Flag varieties}
\label{sec:actionflag}
\subsubsection*{Representations and Density points}
Now, suppose that $V$ is a representation of $G$ with a good norm. Recall that $V^\chi$ is the eigenspace of the highest weight $\chi$. Let $V^*$ be the dual space of $V$. The representation of $G$ on $V^*$ is the dual representation given by: for $g\in G$ and $f\in V^*$, let $\rho^*(g)f=\,^t\rho(g^{-1})f$. This definition gives
\begin{equation}\label{equ:dual}
\rho^*(g)f(\rho(g)v)= \,^t\rho(g^{-1})f(\rho(g)v)=f(v), 
\end{equation}
for $f$ in $V^*$ and $v$ in $V$. Then the highest weight of $V^*$ is $\iota\chi$. The following results explain the relation between different definitions by using combinatorial information on root systems and representations.
\begin{lem}
	For an irreducible representation $V$ of $G$ with highest weight $\chi$, 
	\begin{equation}\label{equ:woweight}
	V_{\chi,\zeta_o}=V^{w_0\chi}.
	\end{equation}
\end{lem}
\begin{proof}
	This can be verified as follows: For $X$ in $\frak a$ and $v$ in $V^\chi$,
	\[\dd\rho(X)\rho(w_0)v=\rho(w_0)\dd\rho(w_0X)v=\chi(w_0X)\rho(w_0)v=(w_0\chi)(X)\rho(w_0)v.  \]
	The proof is complete.
\end{proof} 
\begin{lem}\label{lem:ym rhog}
	Let $V$ be a proximal representation of $G$ with highest weight $\chi$. Then we have
	\begin{equation}\label{equ:ym rhog}
	x^M_{\rho(g)}=\rho(k_g)V^\chi\text{ and } y^m_{\rho(g)}=\,^t\rho(\ell_g)(V^*)^{-\chi}. 
	\end{equation}
	If $V$ is irreducible, then we have 
	$$x^M_{\rho(g)}=V_{\chi,\eta^M_g}\text{ and }y^m_{\rho(g)}=V^*_{\iota\chi, \zeta^m_g}.$$
\end{lem}
\begin{proof}
Let $\{  e_1,\dots,e_d \}$ be an orthonormal basis of $V$ composed of eigenvectors of $\rho(A)$ such that $e_1\in V^\chi$. Then $\rho(A)$ is diagonal. For $\exp(X)\in A^+$, since $\chi$ is the highest weight, we have
\[a_1=\exp(\chi(X))\geq a_2,\dots, a_d. \]
By the definition of a good norm, $\rho(K)$ preserves the norm. Hence for $g$ in $G$, the formula $\rho(g)=\rho(k_g)\rho(a_g)\rho(\ell_g)$ is a decomposition which satisfies \eqref{equ:kgaglg} in the previous paragraph with some permutation of $\{e_2,\dots, e_d \}$. But these permutations do not change the density points. Hence we have $x^M_{\rho(g)}=\R\rho(k_g)e_1=\rho(k_g)V^\chi$. If $V$ is irreducible we have $x^M_{\rho(g)}=V_{\chi,\eta^M_g}$. 

 In the dual space, we can verify that $e_1^*$ has weight $-\chi$, which is the lowest weight in weights of $V^*$. By the same argument as in $\bp V$, we have 
\begin{equation*}
y^m_{\rho(g)}=\R\,^t\rho(\ell_g)e_1^*=\,^t\rho(\ell_g)(V^*)^{-\chi}. 
\end{equation*}
We also have a map from $\P$ to $\bp V^*$. Hence by \eqref{equ:woweight} with representation $V^*$ and weight $\iota\chi$, we know
$V^*_{\iota\chi,\zeta_o}=(V^*)^{w_0\iota\chi}=(V^*)^{-\chi}$.
For $\zeta=g\zeta_o$ in $\P$, by definition, 
\begin{equation}\label{equ:iotachi}
V^*_{\iota\chi,\zeta}=gV^*_{\iota\chi,\zeta_o}=g(V^*)^{-\chi}. 
\end{equation}
Since $V$ is irreducible, by \eqref{equ:iotachi} we have $y^m_{\rho(g)}=\,^t\rho(\ell_g)(V^*)^{-\chi}=\rho^*(\ell_g^{-1})(V^*)^{-\chi}=V^*_{\iota\chi, \zeta^m_g}$.
\end{proof}

\subsubsection*{Distance on Flag varieties}
For $\alpha$ in $\Pi$, we abbreviate $V_{\chi_\alpha,\eta},V^*_{\iota\chi_\alpha,\zeta}$ to $V_{\alpha,\eta}, V^*_{\alpha,\zeta}$. \nomentry{$V_{\alpha,\eta}$}{}For $g$ in $G$, by Lemma \ref{lem:ym rhog}, we find $x^M_{\rho_\alpha(g)}=V_{\alpha,\eta^M_g}$ and $y^m_{\rho_\alpha(g)}=V^*_{\alpha,\zeta^m_g}$. 
For $\eta,\eta'$ in $\P$, let 
\[d_\alpha(\eta,\eta')=d(V_{\alpha,\eta},V_{\alpha,\eta'}) \]
\nomentry{$d_\alpha(\eta,\eta')$}{}%
be its distance between their images in $\bp V_\alpha$.
We define a distance on the flag variety. It is the maximal distance induced by projections,
\begin{equation}\label{equ:distance eta eta'}
d(\eta,\eta')=\max_{\alpha\in \Pi}d(V_{\alpha,\eta},V_{\alpha,\eta'}).
\nomentry{$d(\eta,\eta')$}{}
\end{equation}
We have another embedding of the flag variety
\[\P\rightarrow \prod_{\alpha\in\Pi}\bb P(V^*_\alpha). \]
For $\zeta=k\zeta_o\in\P$, by definition, we have $V^*_{\alpha,\zeta}=kV^*_{\alpha,\zeta_o}$.
For $\eta\in\P$ and $\zeta\in \P$, we set
\[\delta(\eta,\zeta)=\min_{\alpha\in\Pi}\delta(V_{\alpha,\eta},V^*_{\alpha,\zeta}). \]
\nomentry{$\delta(\eta,\zeta)$}{}%
In particular, because the images of $\eta_o,\zeta_o$ in $\bp V_\alpha,\bp V^*_\alpha$ are $V^{\chi_\alpha}, (V^*)^{-\chi_\alpha}$, we know $\delta(V_{\alpha,\eta_o},V^*_{\alpha,\zeta_o})=\delta(V^{\chi_\alpha},(V^*)^{-\chi_\alpha})=1$, and then
\begin{equation}\label{equ:delta etao}
	\delta(\eta_o,\zeta_o)=1.
\end{equation}
 We write 
 $$b^M_{V_\alpha,g}(r)=\{x\in\bb PV_\alpha|d(x,x^M_{\rho_\alpha(g)})\leq r \},$$  $$B^m_{V_\alpha,g}(r)=\{x\in\bb PV_\alpha|\delta(x,y^m_{\rho_\alpha(g)})\geq r \}.$$ 
 They are subsets of $\bp V_\alpha$.
 Write 
 $$b^M_g(r)=\{\eta\in\PP|\forall\alpha\in\Pi,\  V_{\alpha,\eta}\in b^M_{V_\alpha,g}(r) \}=\{\eta\in\PP|d(\eta,\eta^M_g)\leq r \},\nomentry{$b^M_g(r)$}{}$$
  $$B^m_g(r)=\{\eta\in\PP|\forall\alpha\in\Pi,\  V_{\alpha,\eta}\in B^m_{V_\alpha,g}(r) \}=\{\eta\in\P|\delta(\eta,\zeta^m_g)\geq r \}.  \nomentry{$B^m_g(r)$}{}$$ 
 They are subsets of $\PP$.
\subsubsection*{Distance and norm}
We need a multidimensional version of the lemmas in Section \ref{sec:linear action}. They are about the similar quantities on flag varieties. The idea is to use all the representations $\rho_\alpha$. There exists $C_1>0$ such that for any element $X$ in $\frak b$, we have
\begin{equation}\label{equ:norrep}
\frac{1}{C_1}\sup_{\alpha\in \Pi}|\chi_\alpha(X)|\leq \|X\|\leq C_1 \sup_{\alpha\in \Pi}|\chi_\alpha(X)|.
\end{equation}
Using Lemma \ref{lem:flapro}, \eqref{equ:norrep} and $\sigma(g,\eta)-\kappa(g)\in\frak b$, we deduce the following two lemmas from Lemma \ref{lem:cocycle} and Lemma \ref{lem:gBmg}
\begin{lem}\label{lem:iwacar}
	For $g$ in $G$ and $\eta$ in $\P$,
	\begin{equation*}
	\|\sigma(g,\eta)-\kappa(g)\|\leq C_1 |\log \delta(\eta,\zeta^m_g)|.
	\end{equation*}
\end{lem}
For $g$ in $G$ and $\alpha\in\Pi$, by Lemma \ref{lem:flapro}, \[\gamma_{1,2}(\rho_\alpha(g))=\frac{\|\wedge^2\rho_\alpha (g)\|}{\|\rho_\alpha(g)\|^2}=e^{(2\chi_\alpha-\alpha-2\chi_\alpha)\kappa(g)}=e^{-\alpha\kappa(g)}.\]
Let 
\begin{equation}\label{equ:gap of g}
\gap(g)=\sup_{\alpha\in\Pi}e^{-\alpha\kappa(g)}. \nomentry{$\gap(g)$}{}
\end{equation}	
We call it the gap of $g$.
\begin{lem}\label{lem:gBmgmul}
	Let $\delta>0$. For $g$ in $G$, if $\expec=\gap(g)\leq\delta^2$, then 
	\begin{itemize}
		\item the action of $g$ on $B^m_{g}(\delta)$ is $\expec\delta^{-2}$-Lipschitz and
		\[gB^m_{g}(\delta)\subset b^M_{g}(\expec\delta^{-1})\subset b^M_{g}(\delta), \]
		\item
		the restriction of the $\frak a$-valued function $\sigma(g,\cdot)$ on $B^m_{g}(\delta)$ is $O(\delta^{-1})$-Lipschitz.
	\end{itemize}
\end{lem}
These properties tell us that the action of an element $g$ on a large set of the flag variety $\P$ behaves like uniformly contracting map.

We also need to compare the distance on the projective space and the flag variety. Recall that the map from $\P$ to $\bp V$ is defined in \eqref{equ:ppV}.
\begin{lem}\label{lem:profla}
	Let $(\rho,V)$ be an irreducible representation of $G$ with highest weight $\chi$. There exists a constant $C>0$ depending on the chosen norm such that for $\eta,\eta'$ in $\P$,
	\begin{equation}\label{equ:profla}
	d(V_{\chi,\eta},V_{\chi,\eta'})\leq C d(\eta,\eta').
	\end{equation}
\end{lem}
The intuition is that a differentiable map between two compact Riemannian manifolds is Lipschitz. For more details, please see Lemma \ref{lem:equivalence distance P} in Appendix \ref{sec:equi distance}.
\subsection{Actions on the tangent bundle of the Flag variety}
\label{sec:tangent}
In this section, we will study the action of $G$ on the tangent bundle of $\P$. Recall that $\P\simeq G/P$ is the flag variety and $P=AN$ is a parabolic subgroup.

We first study the tangent bundle of the homogeneous space
$$\P_0=G/A_eN.\nomentry{$\P_0$}{}$$ 
Recall that $A_e$ is the analytical connected component of $A$, given by $\exp(\frak a)$. Note that the left action of $K$ on $\P_0$ is simply transitive (due to the Iwasawa decomposition in split case). 
Let $\k_o$ be the base point $A_eN$ in $\P_0$. We can identify the left $K$-invariant vector fields as 
$$T_{\k_o}\P_0=T_{\k_o}(G/A_eN)\simeq \frak g/\frak p.$$
Hence the tangent bundle of $\P_0$ has an isomorphism
\[T\P_0\simeq \P_0\times \frak g/\frak p, \]
that is because we can identify the tangent space at $\k_o$ and $\k=k\k_o$ by the left action of $k$. We denote by $(\k,Y)$ a point of $T\P_0$ where $\k$ is in $\P_0$ and $Y$ is in $\frak g/\frak p$. We use elements in $\frak n^-=\oplus_{\alpha\in R^+}\frak g^{-\alpha}$ as representative elements in $\frak g/\frak p$. 

Then we describe the left action of $G$ on $T\P_0$. Take $Y$ in $\frak g^{-\alpha}$ and $\k=k\k_o$ in $\P_0$. For $g$ in $G$, by the Iwasawa decomposition we have a unique $k'$ in $K$ and a unique $\sigma(g,k)$ in $\frak a$ such that $gk=k'p\in k'\exp(\sigma(g,k))N$, where $p\in A_eN$. Here $\sigma(g,k)$ is understood as $\sigma(g,k\eta_o)$. Due to 
$$gk\exp(tY)\k_o=k'p\exp(tY)\k_o=k'\exp(t\ad_pY)\k_o,$$ 
by taking derivative at $t=0$, the left action of $g$ on the tangent vector $(\k,Y)$ satisfies
\[L_g(\k,Y)=(\k',\ad_pY), \]
where $\k'=k'\eta_o$ and $\ad$ is the adjoint action of $P$ on $\frak g/\frak p$. 

Now we restrict our attention to simple roots. Let $\alpha$ be a \textbf{simple root}. Due to $Y\in \frak g^{-\alpha}$, we have $\ad_NY\subset Y+\frak a+\frak n$, which implies that the unipotent part $N$ acts trivially on $(\frak g^{-\alpha}+\frak p)/\frak p$. Due to $p\in \exp(\sigma(g,k))N$, we have
\begin{equation}\label{equ:adjoint P}
	\ad_pY=\exp(-\alpha\sigma(g,k))Y \text{ on }(\frak g^{-\alpha}+\frak p)/\frak p.
\end{equation} 
This means that the line bundle $\P_0\times \frak g^{-\alpha}$ is stable under the left action of $G$, and we call it the $\alpha$-bundle.

The flag variety $\P$ is a quotient of $\P_0$ by the right action of group $M$, due to $A=MA_e$. We use $\pi$ to denote the quotient map. The right action of $M$ also induces an action on the tangent bundle. For $(\k,Y)$ in $T\P_0$ and $m$ in $M$, by $k\exp(tY)m\k_o=km\exp(t\ad_{m^{-1}}Y)\k_o$, we have
\begin{equation}\label{equ:right M}
	R_m(k\k_o,Y)=(km\k_o,\ad_{m^{-1}}Y).
\end{equation}
Descending to the quotient implies that the tangent bundle of $\P$ satisfies
\[T\P\simeq \P_0\times_M\frak g/\frak p, \]
which is the quotient space of $\P_0\times \frak g/\frak p$ by the equivalence relation generated by the action of $M$, \eqref{equ:right M}. As $M$ is a subgroup of $A$, its adjoint action preserves the line $\frak g^{-\alpha}$ in $\frak g/\frak p $. Hence the $\alpha$-bundle on $\P_0$ descends to a line bundle on $\P$. The integral curves of the $\alpha$-bundle on $\P_0$ are closed, and we call them $\alpha$-circles on $\P_0$. At a point $\k=k\k_o$ in $\P_0$, it is given by 
\begin{equation}\label{equ:alpha circle}
\gamma_\alpha:\R \rightarrow \P_0,\ t\mapsto k\exp(tK_\alpha)\k_o.
\end{equation}
This can be verified directly, because the tangent vector of the curve at time $t$ is $(\gamma_\alpha(t),K_\alpha)=(\gamma_\alpha(t),Y_\alpha)$, due to the definition of $\frak g/\frak p$,  which belongs to the $\alpha$-bundle. The one parameter subgroup $\{\exp(tK_\alpha):t\in\R\}$ is a compact subgroup of $G$, which is isomorphic to $SO(2)$. We call it $O_\alpha$.

Under the right action of $M$, the $\alpha$-circles on $\P_0$ descend to the $\alpha$-circles on $\P$.
\begin{lem}\label{lem:alpha circle}
	Let $\chi$ be a dominant weight such that $\chi(H_\alpha)=0$ and let $(\rho,V)$ be an irreducible representation with highest weight $\chi$. Then the image of an $\alpha$-circle in $\bp V$ is a point.
	
	Suppose in addition that the semisimple part of $\bf G$ is simply connected. The image of the $\alpha$-circle containing $\eta=k\eta_o$ in $\bp V_{\alpha}$ is the projective line generated by $\rho_\alpha(k)V^{\chi_{\alpha}}$ and $\rho_\alpha(k)V^{\chi_\alpha-\alpha}$.
\end{lem}
\begin{proof}
	Since the $\alpha$-bundle is left $K$-invariant, the set of $\alpha$-circles are also left $K$-invariant. It is sufficient to consider the $\alpha$-circle containing $\eta_o$. By \eqref{equ:alpha circle} and \eqref{equ:ppV}, the image of $\alpha$-circle is given by $\rho(O_\alpha)V^\chi$.
	
	Consider the Lie algebra $\frak s_\alpha$ generated by $H_\alpha,X_\alpha,Y_\alpha$, which is isomorphic to $\frak{sl}_2$. For $v$ in $V^\chi$, we have $\dd\rho(H_\alpha)v=\chi(H_\alpha)v$. Due to the classification of the irreducible representation of $\frak{sl}_2$, the irreducible representation $V_1$ of $\frak{s}_\alpha$ generated by $V^\chi$ is of dimension $\chi(H_\alpha)+1$.
	
	When $\chi$ satisfies $\chi(H_\alpha)=0$, the above argument implies that $V_1$ is a trivial representation and $\rho(O_\alpha)$ acts trivially on $V_1$.
	Hence the image of the $\alpha$-circle is a point.
	
	For simply connected case, by $\chi_\alpha(H_\alpha)=\tilde\omega_\alpha(H_\alpha)=1$, the same argument implies that $V_1$ is of dimension 2. Another eigenspace of $V_1$ is $V^{\chi_{\alpha}-\alpha}$. The group $\rho(O_\alpha)$ acts as $SO(2)$ on $V_1$, which implies the result.
\end{proof}
\begin{rem}
	If the reader knows the partial flag variety $\P_{\Pi-\{\alpha \}}$, then the $\alpha$-circle is simply the fibre of the quotient map $\P\rightarrow \P_{\Pi-\{\alpha \}}$. This point of view also implies Lemma \ref{lem:alpha circle}.
\end{rem}
Generally, the $\alpha$-bundle on $\P$ is non trivial in the sense of line bundle.
\begin{exa}
	Let $G$ be $\rm{SL}_3(\bb R)$. Recall that
	\[\frak a=\{X=\diag(x_1,x_2,x_3)|\ x_1+x_2+x_3=0,\ x_1,x_2,x_2\in\R \}, \] 
	and $\alpha_1,\alpha_2$ are two simple roots given by $\alpha_1=\lambda_1-\lambda_2$ and $\alpha_2=\lambda_2-\lambda_3$. The group $M$ is $\{e,\diag(1,-1,-1),\diag(-1,1,-1),\diag(-1,-1,1) \}\simeq (\bb Z/2\bb Z)^2$. We have
	\[\ad_{\diag(1,-1,-1)}Y_{\alpha_1}=\alpha_1^\up(\diag(1,-1,-1))Y_{\alpha_1}=-Y_{\alpha_1}. \]
	In this case the action of $M$ is nontrivial and it is not a normal subgroup of $K=SO(3)$.	The $\alpha$-bundle on $\P$ restricted to an $\alpha$-circle is roughly a M\"obius band.
	
	In this case, $\alpha_1$-circles are given by $\{W_1\subset W_2 \}$, where $W_2$ is a fixed two dimensional subspace of $\R^3$ and $W_1$ varies in one-dimensional subspaces of $W_2$. On the contrary, $\alpha_2$-circles are given by $\{W_1\subset W_2 \}$ with $W_1$ fixed and $W_2$ varying in two planes which contain $W_1$. From this description, we can easily see the $G$-invariance of the set of $\alpha$-circles.
\end{exa} 
It is better to work on $\P_0$, where the $\alpha$-bundle is trivial. One difficulty is that in the covering space $\P_0$, we need to capture the missing information of the group $M$. More precisely, for $h$ in $G$ and $\k, \k'$ in $\P_0$ if $h\pi(\k),h\pi(\k')$ are close, we do not know whether $h\k,h\k'$ are close or not. This will be answered at the end of Section \ref{sec:sign group}.

\begin{rem}
	In an abstract language as in \cite[Lemma 4.8]{benoistquint14random}, we have a principal bundle $M\rightarrow \P_0\rightarrow\PP$, where the action of $M$ on $\P_0$ is a right action. We also have a left action of a semigroup $\Gamma$ in $G$ on $\P_0$ and $\PP$ ($\Gamma$ will be taken as $\Gamma_\mu$ in our case). Suppose that we have a $\Gamma$-minimal set $\Lambda_\Gamma$ in $\PP$. The lifting of $\Lambda_\Gamma$ to $\P_0$ has different possibilities. Let $\eta$ be a point in $\Lambda_\Gamma$ and $\k=k\k_o$ be a lifting of $\eta$ in $\P_0$. Let $M_{\k}=\{m\in M| \overline{\Gamma k}m=\overline{\Gamma k} \}$, where the closure is taken in $\P_0$. Then we have a bijection between two sets
	\[\{\Gamma-\text{minimal orbit in }\P_0 \}\longleftrightarrow M_{\k}\backslash M. \]
	In particular, if $\Gamma$ is a semigroup of matrices of positive entries in $\slrn$, then $M_{\k}=\{ e\}$ and $\Gamma$ has the maximal number of minimal orbits in $\P_0$. 
\end{rem}
\subsection{The sign group}
\label{sec:sign group}
Suppose the semisimple part of $\bf G$ is simply connected. Recall the notation for reductive groups and Lie algebras. Let $N^-$ be the subgroup with Lie algebra $\frak n^-$.
We have a Bruhat decomposition of the reductive group $G$(\cite[21.15]{borel1990linear}), where the main part is given by 
\[N^-\times M\times  A_e\times N\rightarrow G. \]
The image $U$ is a Zariski open subset of $G$ and the map is injective. For elements in $U$, we can define a map $\sg$ to the group $M$, mapping an element $g$ to the part of $M$ in the Bruhat decomposition.

A part of $M$ is given by the different analytical connected components of $G$. Let 
$$M_0=M\cap G^o\ and\  M_1=M/M_0$$
be the quotient group. Let $\pi_0(X)$ be the set of connected components of a topological space $X$ and let $\#\pi_0(X)$ be its number of elements.
\begin{lem}\label{lem:mbm}
	We have $M\cap B=M_0$. 
\end{lem}
\begin{proof}
	Recall that $B=A\cap \scr D G$. Due to $\scr D G\supset K^o$,
	\[M\cap B=M\cap (A\cap \scr D G)=M\cap \scr D G\supset M\cap K^o. \]
	Since $M$ is a subset of $K$, we see that $M_0=M\cap G^o=M\cap K\cap G^o=M\cap K^o$. At the same time, since $\scr D \bf G$ is simply connected, the group of real points $\scr DG$ is connected in the Lie group topology. Therefore
	\[M\cap K^o=M_0=M\cap G^o\supset M\cap\scr D G=M\cap B. \]
	The proof is complete.
\end{proof}
Let $\P_1=G/A_eBN$.
\begin{lem}
	The homogeneous space $\P_1$ has the same number of connected components as $\P_0$, that is $\#\pi_0(\P_0)=\#\pi_0(\P_1)=\#\pi_0(M_1)$, and each connected component of $\P_1$ is isomorphic to $\P$ as topological spaces.
\end{lem}
\begin{proof}
	Since $\scr D G$ is connected, we know $A_eBN\subset A_e\scr DG\subset G^o$. The number of connected components of $\P_1$ equals to $\# \pi_0(G)=\#\pi_0(\P_0)$.
	
	The degree of the covering $\P_1\rightarrow \P$ equals to 
	\[\#(A/A_eB)=\#(M/M\cap B) .\]
	By Lemma \ref{lem:mbm}, we have $M\cap B=M_0$. Hence
	\[\#(A/A_eB)=\#(M/M_0)=\#\pi_0(M_1)=\#\pi_0(G)=\#\pi_0(\P_1). \]
	Since $\P$ is connected, the result follows.
\end{proof}
Hence, the $M_1$ part of $m(g)$ can be determined by seeing in which connected component of $G$ the element $g$ is. We want to know for two near elements $g,g'$ in $G$, whether we have $\sg(g)=\sg(g')$ or not. 

In order to study the $M_0$ part, we will use representations defined in Lemma \ref{lem:tits} to give another description of the sign group. This is in the same spirit as the treatment of the sign group $M$ in \cite{benoist2005convexes}.
Let $v_\alpha$ be a non-zero eigenvector with highest weight $\chi_\alpha$ in $V_\alpha$. Let $\sign$ be the sign function on $\R$.
\begin{lem}\label{lem:isomorphism M Z}
	For $g$ in $U$, we have
	\[\sign\l v_\alpha,\rho_\alpha(g)v_\alpha\r=\chi_\alpha^\up(\sg(g)), \]
	where $\chi_\alpha^\up$ is the corresponding algebraic character on $A$ of the weight $\chi_\alpha$.
\end{lem}
\begin{proof}
	Since $v_\alpha$ is $N$-invariant and the Cartan involution $\theta$ maps $N^-$ to $N$, by \eqref{equ:good norm} 
	\begin{align*}
	\l v_\alpha,\rho_\alpha(N^-MA_eN)v_\alpha)&=\l^t\rho_{\alpha}(N^-)v_\alpha,\rho_\alpha(MA_eN)v_\alpha\r=\l\rho_{\alpha}(\theta(N^-))v_\alpha,\rho_\alpha(MA_eN)v_\alpha\r\\
	&=\l\rho_{\alpha}(N)v_\alpha,\rho_\alpha(MA_eN)v_\alpha\r=\l v_\alpha,\rho_\alpha(MA_e)v_\alpha\r.
	\end{align*}
	The action of $A_e$ does not change the sign, hence by Lemma \ref{lem:sign m} we have 
	\[ \sign\l v_\alpha,\rho_\alpha(g)v_\alpha\r=\sign\l v_\alpha,\rho_\alpha(\sg(g))v_\alpha\r=\chi_\alpha^\up(\sg(g)).\]
	The proof is complete.
\end{proof}

We are in the simply connected case and we have $\rm X(\bf B)=\oplus_{\alpha\in\Pi}\Z \tilde{\omega}_\alpha$. Due to $M_0=M\cap B$, we know that $\chi_{\alpha}^\up(m)=\tilde{\omega}_\alpha^\up(m)$ for $m$ in $M_0$. By $B\simeq (\R^*)^{\#\Pi}$, the common kernel of all the characters is the neutral element.
Therefore
\begin{lem}\label{lem:isomorphism M Z1}
	The function $\Pi_{\alpha\in\Pi}\,\chi_\alpha^\up:M_0\rightarrow\R^{\#\Pi}$ given by $$\Pi_{\alpha\in\Pi}\,\chi_\alpha^\up(m)=(\chi_{\alpha}^\up(m))_{\alpha\in\Pi}\ \ \text{ for }m\in M_0,$$ is injective.
\end{lem}
\begin{defi}
	We define the sign function from $G\times G$ to $M\cup \{0 \}$ by
	\begin{equation*}
	\sg(g,g')=\begin{cases}
	\sg(\theta(g^{-1})g') &\text{ if }\,\theta(g^{-1})g'\in U,\\
	0 &\text{ if not,}
	\end{cases}
	\end{equation*}
	where $g,g'$ are in $G$ and $\theta$ is the Cartan involution.\nomentry{$\sg(g,g')$}{}
\end{defi}	
	This definition exploits the relation between $g$ and $g'$. More precisely, for $u,v$ in $V_\alpha$ we have $\l v, \,\rho_\alpha(\theta(g^{-1})g')u\r=\l \rho_\alpha gv, \rho_\alpha g'u\r$, which explains the definition. Due to $\theta(N)=N^-$, the sign function $\sg$ factors through $G/A_eN \times G/A_eN=\P_0\times \P_0$. 
	
	We now explain the sign function for the case $\gltwo$. We only need to consider the representation of $\gltwo$ on $\bb R^2$. Let $v_0=\begin{pmatrix}
	1\\ 0
	\end{pmatrix}$ be a vector with highest weight in $\R^2$. Then 
	\[\l v_0,\theta(g^{-1})g'v_0 \r=\l gv_0,g'v_0 \r,  \]
	which is the inner product of the first column of $g$ and $g'$. The sign function is used to determine whether these two vectors $gv_0,g'v_0$ have an acute angle and whether $g$ and $g'$ are in the same connected component.
	
	By the Bruhat decomposition, we have the following lemma.
\begin{lem}\label{lem:gg'm}
	For $g,g'$ in $G$ and $m$ in $M$, we have
	\begin{equation*}
	\sg(g,g'm)=\sg(gm,g')=\sg(g,g')m.
	\end{equation*}
\end{lem}

\begin{lem}\label{lem:dirgk}
	Take a Cartan decomposition of $g$, that is $g=k_ga_g\ell_g\in KA^+K$. 
	Then for $h$ in $G$,
	\begin{equation*}
		\sg(k_g,gh)=\sg(\ell_g^{-1},h).
	\end{equation*}
\end{lem}
The key observation here is that the sign function is locally constant. Recall that $\zeta_o$ is point in $\P$ and its image in $\bp V^*_\alpha$ is the linear functional on $V_\alpha$ which vanishes on the hyperplane perpendicular to $V^{\chi_{\alpha}}$. Recall that $\delta(\eta,\zeta)=\min_{\alpha\in\Pi}\delta(V_{\alpha,\eta},V^*_{\alpha,\zeta})$ and $d(\eta,\eta')=\max_{\alpha\in\Pi}d(V_{\alpha,\eta},V_{\alpha,\eta'})$.
\begin{lem}\label{lem:U}
	For $g\in G-U$, we have
	\[ \delta(g\eta_o,\zeta_o)=0 \text{ and }d(g\eta_o,\eta_o)=1. \]
\end{lem}
\begin{proof}
	By Bruhat's decomposition, we know that $g\in N^-A_ewN$ for some non-trivial element $w$ in the Weyl group. Then the action of $w$ on $\frak b^*$ is non-trivial, there exists a simple root $\alpha$ such that $w(\chi_{\alpha})\neq\chi_{\alpha}$.
	
	Then $\rho_{\alpha}(g)v_\alpha=\rho_{\alpha}(N^-MA_e(g))\rho_{\alpha}(w)v_\alpha$, where $N^-MA_e(g)$ is the corresponding part of $g$ in the Bruhat decomposition. The vector $\rho_{\alpha}(w)v_\alpha$ is of weight $w(\chi_{\alpha})<\chi_{\alpha}$ by the definition of highest weight and $\rho_{\alpha}(N^-MA_e(g))\rho_{\alpha}(w)v_\alpha$ is a linear combination of vectors of weight less than or equal to $w(\chi_{\alpha})$. Since vectors of different weights are orthogonal, we obtain
	\[\delta(V_{\alpha,g\eta_o},V^*_{\alpha,\zeta_o})=\langle v_\alpha,\rho_{\alpha}(g)v_\alpha\rangle=0 \]
	and
	\[d(V_{\alpha,g\eta_o},V_{\alpha,\eta_o})=\frac{\|\rho_{\alpha}(g)(v_\alpha)\wedge v_\alpha\|}{\|\rho_{\alpha}(g)v_\alpha\| \|v_\alpha\|}=1. \]
	The proof is complete.
\end{proof}
\begin{lem}\label{lem:locally constant}
	For $k_1,k_2,k_3$ in $K$, if $\delta(k_2\eta_o,k_1\zeta_o)> d( k_2\eta_o, k_3\eta_o)$, then
	\begin{equation*}
		\sg(k_1,k_2)=\sg(k_1,k_3)\sg(k_2,k_3).
	\end{equation*}
\end{lem}
\begin{figure}
	\begin{center}
		\begin{tikzpicture}[scale=2]
		\draw [domain=0:2*pi, samples=200] plot ({2*cos(\x r)},{2*sin(\x r)});
		\draw [dotted](0,0) -- (2,0);
		\draw (2,0) node[right]{$v_1$};
		\draw (0,-2) -- (0,2);
		\draw (0,2) node[above]{$v_1^\perp$};
		\draw (0,0) -- (1,{sqrt(3)});
		\draw (1,{sqrt(3)}) node[above]{$v_3$};
		\draw (0,0) -- ({sqrt(3)}, 1);
		\draw ({sqrt(3)},1) node[right]{$v_2$};
		\draw [domain=pi/6:pi/2, samples=200] plot ({0.5*cos(\x r)},{0.5*sin(\x r)});
		\draw (1/4,{sqrt(3)/4}) node[above]{$\vartheta_1$};
		\draw [domain=pi/6:pi/3, samples=200] plot ({cos(\x r)},{sin(\x r)});
		\draw ({sqrt(2)/2},{sqrt(2)/2}) node[above]{$\vartheta_2$};
		\end{tikzpicture}
	\end{center}
	\caption{Angle}\label{fig:angle}
\end{figure}
\begin{proof}
	By Lemma \ref{lem:gg'm}, it is sufficient to consider $k_i\in K^o$.
	By definition and \eqref{equ:dual}, we have $\delta(k_2\eta_o,k_1\zeta_o)=\delta(k_1^{-1}k_2\eta_o,\zeta_o)$ and $\sg(k_1,k_2)=\sg(id,\,k_1^{-1}k_2)$. Hence, we can suppose that $k_1=e$, the identity element in $K$. By $d(k_3^{-1}k_2\eta_o,\eta_o)=d(k_2\eta_o,k_3\eta_o)<1$ and Lemma \ref{lem:U}, we have $m(k_2,k_3)\in M$. Lemma \ref{lem:isomorphism M Z} and Lemma \ref{lem:isomorphism M Z1} imply that it is sufficient to prove that if $\delta(k_2\eta_o,\zeta_o)>d(k_2\eta_o,k_3\eta_o)$ and $\sg(k_2,k_3)=e$, then for every simple root $\alpha$ we have
	\[\sign\l v_\alpha,\rho_\alpha(k_2)v_\alpha\r=\sign\l v_\alpha,\rho_\alpha(k_3)v_\alpha\r. \]
	
	Fix a simple root $\alpha$ in $\Pi$. Abbreviate $v_\alpha,\rho_\alpha(k_2)v_\alpha, \rho_\alpha(k_3)v_\alpha$ to $v_1, v_2,v_3$. Let $\vartheta_1$ be the angle between the vector $v_2$ and the hyperplane $v_1^{\perp}$ and let $\vartheta_2$ be the angle between $v_2$ and $v_3$. Due to $\sg(k_2,k_3)=e$, by Lemma \ref{lem:isomorphism M Z} this implies
	$$0<\l v_1,\,k_2^{-1}k_3v_1 \r=\l k_2v_1,k_3v_1\r=\l v_2, v_3\r,$$
	hence the angle $\vartheta_2$ is acute. The image of $\zeta_0$ in $\bp V^*_\alpha$ is given by $\R \l v_1,\cdot\r$. The hypothesis $\delta(k_2\eta_o,\zeta_o)>d(k_2\eta_o,k_3\eta_o)$ implies that 
	$$\sin\vartheta_1= \l v_1,v_2\r > \| v_2\wedge v_3\|= \sin\vartheta_2. $$ 
	Hence $\vartheta_2<\vartheta_1$ and $v_2,v_3$ are in the same side of the hyperplane $v_1^\perp$, which implies $\sign \l v_1,v_2\r=\sign \l v_1,v_3\r$. Please see figure \ref{fig:angle}.
\end{proof}
We state a consequence of Lemma \ref{lem:locally constant} which will be used in Section \ref{sec:sumfou} to get independence of certain measures $\lambda_j$.
\begin{lem}\label{lem:ghkk'}
	Let $\delta<1/2$, let $g,h$ be in $G$ and $k,k'$ in $K$. If $h,k,k'$ satisfy 
	\[ d(k\eta_o,k'\eta_o)<\delta,\ k\eta_o,k'\eta_o\in B^m_h(\delta), \eta^M_h\in B^m_g(3\delta)\text{ and }\gap(h)< \delta^2, \]
	then
	\begin{equation*}
		\sg(k_g,ghk)=\sg(\ell_g^{-1},hk')\sg(k,k').
	\end{equation*} 
\end{lem}
\begin{proof}
	By Lemma \ref{lem:gg'm}, it is sufficient to prove the case $\sg(k,k')=e$ and $k,k'$ in $K^o$. By Lemma \ref{lem:dirgk},
	\begin{align}\label{equ:kgghk}
	\sg(k_g,ghk)=\sg(\ell_g^{-1},hk).
	\end{align}	
	Denote $k\eta_o,k'\eta_o$ by $\eta,\eta'$. Then by Lemma \ref{lem:gBmgmul}, we have $h\eta,h\eta'\in b^M_h(\delta)\subset B^m_g(2\delta)$. Hence by $d(h\eta,h\eta')<2\delta\leq \delta(h\eta,\zeta^m_g)=\delta(h\eta,\ell_g^{-1}\zeta_o)$ and Lemma \ref{lem:locally constant}, we have
	\begin{equation}\label{equ:ghkghk'}
		\sg(\ell_g^{-1},hk)=\sg(\ell_g^{-1},hk')\sg(hk,hk'). 
	\end{equation}
	The main point here is to use the following lemma.
	\begin{lem}\label{lem:sg k k'}
		Under the same assumption as in Lemma \ref{lem:ghkk'}, we have
		\[ \sg(hk,hk')=\sg(k,k'). \]
	\end{lem}
	Combined with \eqref{equ:kgghk} and \eqref{equ:ghkghk'}, the proof is complete.
\end{proof}
\begin{proof}[Proof of Lemma \ref{lem:sg k k'}]
	Without loss of generality, suppose that $\sg(k,k')=e$. Due to $k\eta_o\in B^m_h(\delta)$, we can chose a $\ell_h$ in the Cartan decomposition $h=k_ha_h\ell_h$ such that $\sg(\ell_h^{-1},k)=e$. By Lemma \ref{lem:locally constant}, the hypothesis that $\delta(k\eta_o,\ell_h^{-1}\zeta_o)>\delta>d(k\eta_o,k'\eta_o)$ implies $\sg(\ell_h^{-1},k')=\sg(\ell_h^{-1},k)=e$. By Lemma \ref{lem:dirgk}, we conclude that
	$e=\sg(k_h,hk)=\sg(\ell_h^{-1},k)=\sg(\ell_h^{-1},k')=\sg(k_h,hk')$. Here we need a distance $d_0$ on $\P_0$, which is defined in Appendix \ref{sec:equi distance}. Let $\k=k\k_o$ and $\k'=k'\k_o$. By Lemma \ref{lem:p0 p} with $z_h=k_hz_o$
	\begin{equation}\label{equ:hkhk'}
	d_0(h\k,h\k')\leq d_0(h\k,\k_h)+d_0(\k_h,h\k')\leq d(hk\eta_o,\eta^M_h)+d(\eta^M_h,hk'\eta_o). 
	\end{equation}
	Hence by \eqref{equ:hkhk'}, we have $d_0(h\k,h\k')\leq 2\delta<1$, which implies $\sg(hk,hk')=e$ due to Lemma \ref{lem:p0 p}. 
\end{proof}
The proof of Lemma \ref{lem:sg k k'} also says that if $\k,\k'$ are close and away from the bad subvariety defined by $h$, and if the gap of $h$ is large, then $h\k,h\k'$ are also close.
\subsection{Derivative}
\label{sec:derfla}
Suppose in addition that the semisimple part of $\bf G$ is simply connected except Lemma \ref{lem:lift varphi}.
Let $\varphi$ be a $C^1$ function on $\P_0$. We will give some property of the directional derivative of $\varphi$. We write $\partial_\alpha\varphi$
\nomentry{$\partial_\alpha\varphi$}{}
 for the directional derivative $\partial_{Y_\alpha}\varphi$, where $\alpha$ is a simple root. 
\begin{defi}[Arc length] Let $\k_1,\k_2$ be two points in the same $\alpha$-circle in $\P_0$. If $\sg(\k_1,\k_2)=e$, we define the arc length distance between $\k_1,\k_2$ by
	\[d_A(\k_1,\k_2):=\arcsin d(\pi \k_1,\pi \k_2).\nomentry{$d_A(\k_1,\k_2)$}{}\]
\end{defi}
\begin{rem}
	This is a restriction of the left $K$-invariant distance, which can be induced by the $K$-invariant Riemann metric $d_1$ in the appendix.
\end{rem}
\begin{lem}[The Newton-Leibniz formula]
	Let $\k_1,\k_2$ be two points in the same $\alpha$-circle on $\P_0$ such that $\sg(\k_1,\k_2)=e$. Let $u=d_A(\k_1,\k_2)$ and let $\gamma:[0,u]\rightarrow \P_0$ be the curve in the $\alpha$-circle connecting $\k_1,\k_2$ with unit speed (in the sense of arc length). Then for $g$ in $G$
	\begin{equation}\label{equ:newlei}
		\varphi(g \k_1)-\varphi(g \k_2)=\pm\int_0^u\partial_\alpha\varphi_{g\gamma(s)}e^{-\alpha\sigma(g,\gamma(s))}\dd s,
	\end{equation}
	where the sign only depends on the direction of $\gamma$.
\end{lem}
\begin{rem}\label{rem:direction}
	The $\alpha$-circle already has an orientation given by $Y_\alpha$. The sign is negative if the curve $\gamma$ is negatively oriented.
\end{rem}
\begin{proof}
Without loss of generality, suppose that $\gamma$ is positively oriented. Recall that $K_\alpha=Y_\alpha-X_\alpha$ for $\alpha\in\Pi$. The images of $K_\alpha$ and $Y_\alpha$ coincide in $\frak g/\frak p$.
Then $k_2=k_1\exp(uK_\alpha)$ and $\gamma(s)=k_1\exp(sK_\alpha)\k_o$ for $s\in[0,u]$. By the Newton-Leibniz formula and \eqref{equ:adjoint P} we have
\begin{equation*}
\begin{split}
\varphi(g\k_2)-\varphi(g \k_1)&=\int_0^u\dd\varphi_{g\gamma(s)}\dd g_{\gamma(s)}K_\alpha\dd s=\int_0^u\dd\varphi_{g\gamma(s)}\dd g_{\gamma(s)}Y_\alpha\dd s\\
&=\int_0^u\dd\varphi_{g\gamma(s)}\exp(-\alpha\sigma(g,\gamma(s)))Y_\alpha\dd s=\int_0^u\partial_\alpha\varphi_{g\gamma(s)}e^{-\alpha\sigma(g,\gamma(s))}\dd s.
\end{split}
\end{equation*}
The proof is complete.
\end{proof}

	For $m$ in $M$ and $\alpha$ in $\Pi$, by Lemma \ref{lem:sign m} with the adjoint representation of $G$ on $\frak g$, due to $Y_\alpha\in \frak g^{-\alpha}$, we have $\ad_mY_\alpha=(-\alpha)^\up(m)Y_\alpha=\alpha^\up(m)^{-1}Y_\alpha=\alpha^\up(m)Y_\alpha$. The last equality is due to $\alpha^\up(m)\in\{\pm 1 \}$. Thanks to \eqref{equ:right M}, we have
\begin{lem}\label{lem:kmk}
	Let $m$ be in $M$ and let $\varphi$ be a $C^1$ function on $\P_0$ which is right $M$-invariant. We have for $\k=k\k_o$ in $\P_0$
	\begin{equation*}
		\partial_\alpha\varphi_{km\k_o}=\alpha^\up(m)\partial_\alpha\varphi_\k.
	\end{equation*}
\end{lem}
We say a function $\varphi$ on $\P_0$ is the lift of a function on $\bp V_\alpha$, if there exists a function $\varphi_1$ on $\bp V_\alpha$ such that for $\k=k\k_o\in\P_0$
\[\varphi(\k)=\varphi_1(V_{\alpha,k\eta_o}). \]
By Lemma \ref{lem:alpha circle}, we have
\begin{lem}\label{lem:lift varphi}
	Suppose $\bf G$ is a connected $\R$-split reductive $\R$-group. If $\varphi$ is a $C^1$ function on $\P_0$, which is the lift of a $C^1$ function on $\bp V_\alpha$, then 
	\[\partial_{\alpha'}\varphi=0 \text{ for }\alpha'\neq \alpha,\alpha'\in\Pi. \]
\end{lem}

\subsection{Changing Flags}
\label{sec:chafla}

Suppose in addition that the semisimple part of $\bf G$ is simply connected.
This part is trivial for $\sltwo$, where the flag variety $\bp (\R^2)$ is a single $\alpha$-orbit. In this section, we suppose that the semisimple rank $\rank$ is no less than two.

On the flag variety, we have many directions in the tangent space. Roughly speaking, the action of $g$ is contracting and the contraction speed on $Y_\alpha$ is given by $e^{-\alpha\kappa(g)}$, $\alpha\in R^+$. Due to $\kappa(g)$ being in the Weyl chamber $\frak{a}^+$, the slowest directions are given by simple roots. Other directions are negligible. The main result Lemma \ref{lem:chapoi} is a quantitative version of this intuition.

We have already seen that if two points $\eta,\eta'$ are in the same $\alpha$-circle, then we have a nice formula for the difference of the value of a real function $\varphi$ at $g\eta$ and $g\eta'$, where $g\in G$. We want to compute this for $\eta,\eta'$ in general position. This is a new difficulty in higher rank. 

If we are on the euclidean space $\bf E^n$ and we are only allowed to move along the directions of coordinate vectors, for any two points $x,x'$, we can walk from $x$ to $x'$ with at most $n$ moves. But this is not true for the flag variety $\P$. Suppose that we are only allowed to move along $\alpha$ circles with $\alpha\in\Pi$. Then for two general points $\eta,\eta'$ in $\P$, it takes more than $\rank=\#\Pi$ moves to walk from one point to the other point. We try to move in each $\alpha$ circle at most one time and to make the resulting points as close as possible.

Recall that $V$ is a finite-dimensional vector space with euclidean norm. Let $l=\R(v_1\wedge v_2)$ be a point in $\bp (\wedge^2V)$, which is also a projective line in $\bp V$.
\begin{lem}\label{lem:volume area}
	Let $x=\R w_1$ be a point in $\bb PV$ and $l=\R(v_1\wedge v_2)$ be a projective line in $\bb PV$. Then we have
	\begin{equation*}
	d(l,x):=\min_{x'\in l}d(x',x)=\frac{\|v_1\wedge v_2\wedge w_1\| }{\|v_1\wedge v_2\|\|w_1\|}.
	\end{equation*}
\end{lem}
\begin{proof}
	The geometric meaning of $\|v_1\wedge v_2\wedge w_1\|$ is the volume of the parallelepiped generated by three vectors $v_1,v_2,w_1$. This volume can also be calculated as the product of the area of the parallelogram generated by $v_1$ and $v_2$, that is $\|v_1\wedge v_2\|$, and the distance of $w_1$ to the plane generated by $v_1$ and $v_2$, that is $d(w_1,\sp(v_1,v_2))$. Hence, we have the formula
	\begin{equation}\label{equ:volume area}
	\|v_1\wedge v_2\wedge w_1\|=\|v_1\wedge v_2\|d(w_1,\sp(v_1,v_2)). \end{equation}
	The distance $d(w_1,\sp(v_1,v_2))$ equals $\|w_1\|d(l,x)$, because the geometric sense of $d(l,x)$ is the sine of the angle between the vector $w_1$ and the plane $\sp(v_1,v_2)$.
	Together with \eqref{equ:volume area}, we have the result .
\end{proof}
\begin{lem}\label{lem:line point}
	Let $x$ be a point in $\bb PV$ and $l$ be a projective line in $\bb PV$. If $g\in GL(V)$ satisfies that $\delta(x, y^m_{g}),\delta(l,y^m_{\wedge^2g})>\delta$, then
	\begin{equation*}
	d(gl,gx)\leq \delta^{-2}\gamma_{1,3}(g)d(l,x),
	\end{equation*}
	where $\gamma_{1,3}(g)=\frac{\|\wedge^3g\|}{\|\wedge^2g\|\|g\|}$.
\end{lem}
Compared with Lemma \ref{lem:distance}, we see that with more degree of freedom the contracting speed is significantly greater.
\begin{proof}
	By definition and the fact that $l=\R (v_1\wedge v_2), x=\R w_1$, we have
	\begin{align*}
	d(gl,gx)=\frac{\|\wedge^2g(v_1\wedge v_2)\wedge gw_1\|}{\|\wedge^2g(v_1\wedge v_2)\|\|gw_1\|}\leq \frac{\|\wedge^3g\|\|v_1\wedge v_2\wedge w_1\|}{\|\wedge^2g(v_1\wedge v_2)\|\|gw_1\|},
	\end{align*}
	Then by Lemma \ref{lem:cocycle}, we have
	\begin{equation*}
	d(gl,gx)\leq \frac{\|\wedge^3g\|\|v_1\wedge v_2\wedge w_1\|}{\delta^2\|\wedge^2g\|\|v_1\wedge v_2\|\|g\|\|w_1\|}= \frac{\|\wedge^3g\|}{\delta^2\|\wedge^2g\|\|g\|}d(l,x).
	\end{equation*}
	The proof is complete.
\end{proof}
  We can also prove Lemma \ref{lem:line point} by finding a point $x'=\R v'\in l$ such that $v'\wedge w_1$ is orthogonal to the vector of highest weight in $\wedge^2V$. Then the distance between $gx'$ and $gx$ will be roughly $\gamma_{1,3}(g)$.

We will start to change the flags. Recall that for $\alpha\in\Pi$ and $\eta,\eta'$ in $\P$, the function $d_\alpha(\eta,\eta')$ is the distance between the images of $\eta$ and $\eta'$ in $\bp V_{\alpha}$. If one wants to change a flag in the $\alpha$-circle in $\P$, there are some constraints from the structure of flags. We introduce the following definition which explains the constraint.

By Lemma \ref{lem:alpha circle}, we have
\begin{lem}\label{lem:line orbit}	
	The image of the $\alpha$-circle of $\eta$ in $\bp V_{\alpha}$ is a projective line and we call it $l_{\alpha,\eta}$. Seen as an element in $\bp (\wedge^2V_\alpha)$, the element $l_{\alpha,\eta}$ is actually in $\bp V_{2\chi_\alpha-\alpha}\subset\bp(\wedge^2V_\alpha)$.
\end{lem}
\begin{exa}
	If $G=\slrn$. Let
	\[\eta=\{W_1\subset W_2\subset \cdots\subset W_{\rank+1}=\R^{\rank+1}  \} \]
	be a flag in $\PP$. Recall that $W_r$ is a $r$-dimensional subspace of $\R^{\rank+1}$.Take $W_0=\{0\}$. Let $i_r$ be the natural embedding of the Grassmannian to projective spaces, that is $\bb G_r(\R^{\rank+1})\rightarrow\bb P(\wedge^r\R^{\rank+1})$.
	In this case, we see that 
	$$l_{\alpha_r,\eta}=i_r(W_{r+1}\supset W_r'\supset W_{r-1}),
	$$ 
	as a projective line in $\bp(\wedge^r\R^{\rank+1})$, which is the image of all the $r$-dimensional subspace $W_r'$ of $\R^{\rank+1}$ such that $W_{r-1}\subset W'_r\subset W_{r+1}$. 
\end{exa}
\begin{defi}
	Let $(\eta_0,\eta_1,\dots,\eta_k)$ be a sequence of points in $\P$. We call it a chain if any consecutive elements $\eta_i,\eta_{i+1}$ are in the same $\alpha$-circle for some $\alpha\in \Pi$, and we write $\alpha(\eta_i,\eta_{i+1})$ for this simple root. Let $\Pi(\eta_0,\dots,\eta_k)$ be the set of simple roots appearing in the chain, that is
	\[\Pi(\eta_0,\dots,\eta_k)=\{\alpha\in\Pi|\ \exists i, \alpha(\eta_i,\eta_{i+1})=\alpha \}. \]
\end{defi}
\begin{lem}\label{lem:image Valpha}
	Let $(\eta_0,\dots,\eta_l)$ be a chain and let $\alpha$ be a simple root such that $\alpha\notin\Pi(\eta_0,\dots,\eta_l)$. Then the image of the chain in $\bp V_\alpha$ is a single point, that is 
	\[ V_{\alpha,\eta_j}=V_{\alpha,\eta_0},\ \forall j=1,\dots, l. \]
	If for any $\alpha'\in\Pi(\eta_0,\dots,\eta_l)$, we have that $\alpha+\alpha'$ is not a root, then
	\[ l_{\alpha,\eta_j}=l_{\alpha,\eta_0},\ \forall j=1,\dots, l. \]
\end{lem}
\begin{proof}
	The first equality is direct consequence of Lemma \ref{lem:alpha circle} and the relation $\chi_{\alpha}(H_{\alpha'})=\delta_{\alpha\alpha'}$.
	
	For the second equality, let $\alpha'$ be a simple root in $\Pi(\eta_0,\dots,\eta_l)$. The projective line $l_{\alpha,\eta}$ in $\bp V_\alpha$ is uniquely determined by the image of $\eta$ in $\bp V_{2\chi_\alpha-\alpha}$. Hence we only need to understand the image of $\alpha'$-circle in $\bp V_{2\chi_\alpha-\alpha}$. By definition,
	\[(2\chi_\alpha-\alpha)(H_{\alpha'})=2\delta_{\alpha\alpha'}-\alpha(H_{\alpha'}). \]
	Since $\alpha+\alpha'$ is not a root, we know $\alpha(H_{\alpha'})=0$ and $(2\chi_\alpha-\alpha)(H_{\alpha'})=0$. By Lemma \ref{lem:alpha circle}, the image of $\alpha'$-circle in $\bp V_{2\chi_\alpha-\alpha}$ is point. Hence $l_{\alpha,\eta_0}=l_{\alpha,\eta_1}=\cdots=l_{\alpha,\eta_l}$.
\end{proof}
The Coxeter diagram of an irreducible root system is a tree, modulo the multiplicities of edges. We can find a disjoint union $\Pi_1$ and $\Pi_2$ of vertices such that there is no edge whose two endpoints are in the same $\Pi_i$. 
In the Coxeter diagram, two simple roots $\alpha$, $\alpha'$ are connected by an edge if and only if $\alpha+\alpha'$ is a root. Hence, we have
\begin{lem}\label{lem:root partition}
	We can separate $\Pi$ into a disjoint union $\Pi_1$ and $\Pi_2$ such that for $\alpha,\alpha'$ in the same atom $\Pi_j$,
	\[\alpha+\alpha' \text{ is not a root.} \]
\end{lem}
	Let $\nupione=\#\Pi_1$ and $\nupitwo=\#\Pi_2$.
Now, we state our main result of this part, which will be used in the main approximation (Proposition \ref{prop:mainapprox}). 
\begin{lem}
	\label{lem:chapoi}
	Let $\eta,\eta'$ be two points in $\PP$ and let $g$ be in $G$. If for $\alpha\in \Pi_1$,
	\begin{align*}
	\delta(V_{\alpha,\eta'},y^m_{\rho_\alpha(g)}), \delta(l_{\alpha,\eta},y^m_{\wedge^2\rho_\alpha(g)})>\delta,
	\end{align*}
	for $\alpha\in\Pi_2$,
	\begin{align*}
	\delta(V_{\alpha,\eta},y^m_{\rho_\alpha(g)}), \delta(l_{\alpha,\eta'},y^m_{\wedge^2\rho_\alpha(g)})>\delta,
	\end{align*}
	then we can find two chains $(\eta=\eta_0,\eta_1,\dots ,\eta_\nupione)$ and $(\eta'=\eta_0',\eta_1',\dots,\eta_\nupitwo')$ such that
	\begin{equation}\label{equ:djgj}
	d(g\eta_{j},g\eta_{j+1})=d_\alpha(g\eta_{j},g\eta_{j+1})=d_\alpha(g\eta,g\eta')+O(\delta^{-2}\beta e^{-\alpha\kappa(g)}),
	\end{equation} 
	where $\alpha=\alpha(\eta_j,\eta_{j+1})\in\Pi_1$ and different $j$ correspond to different roots; similarly for $\eta'$ with $\Pi_2$. 
	
	We also have that for all $\alpha\in\Pi$
	\begin{equation}\label{equ:geta1 geta2}
	d_\alpha(g\eta_\nupione,g\eta_\nupitwo')\leq \beta e^{-\alpha\kappa(g)}\delta^{-2},
	\end{equation}
	where $\beta$ is the gap of $g$, that is $\beta=\gap(g)=\max_{\alpha\in\Pi}\{e^{-\alpha\kappa(g)} \}$.
\end{lem}
The point is that the contraction speed $\beta$ implies that the term $\delta^{-2}\beta e^{-\alpha\kappa(g)}$ is of smaller magnitude than $e^{-\alpha\kappa(g)}$. The objective is to walk from $g\eta$ to $g\eta'$ only through $\alpha$ circles and to preserve information of distance. Since we can neglect error term, it is simpler to walk from $g\eta$ to $g\eta_{l_1}$ through some $\alpha$ circles and to walk from $g\eta'$ to $g\eta_{l_2}'$ through the other $\alpha$ circles, which means the corresponding simple roots are different from the one for the first walk. After this operation, the distance between $g\eta_{l_1}$ and $g\eta_{l_2}'$ is negligible, due to \eqref{equ:geta1 geta2}. The distance of the move in the $\alpha$ circle is approximately the distance between the images of $g\eta$ and $g\eta'$ in $\bp V_\alpha$, due to \eqref{equ:djgj}.
\begin{figure}
	\begin{center}
		\begin{tikzpicture}[scale=1.5]
		\draw (-4,0) -- (4,0);
		\draw (-3,0) node[above]{$\alpha_1$-orbit};
		\fill  (-2,0) circle (1.5pt);
		\draw (-2,-0.25) node[below]{$g\eta_0$};
		\draw (3,-1) -- (-2,4);
		\fill (0,2) circle (1.5pt);
		\draw (-1,3.25) node[above]{$g\eta_0'$};
		\draw (3,2) node[above]{$g\eta_1'$};
		\draw (0,0) circle (1.5pt);
		\draw (0,-0.25) node[below]{$g\eta_1$};
		\draw [dotted] (-4,2) -- (4,2);
		\draw [shift={(0,2)}]  plot[domain=0:2.3,variable=\t]({1*1*cos(\t r)+0*2*sin(\t r)},{0*2*cos(\t r)+1*1*sin(\t r)});
		\draw (1,2.5) node[right]{$\alpha_2$-orbit};
		\draw [->] (1,2.00001) -- (1,2);
		\end{tikzpicture}
	\end{center}
	\caption{Changing Flag for $\rm{SL}_3(\R)$}\label{fig:change flag}
\end{figure}
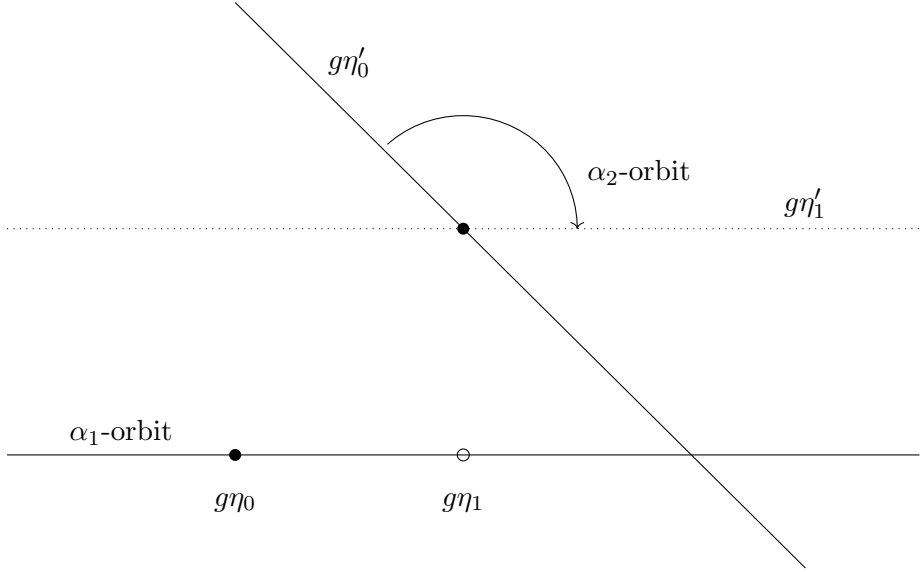
\begin{proof}[Proof of Lemma \ref{lem:chapoi}]
	If we have already found $(\eta_0,\dots,\eta_j)$ and $j<\nupione$, we want to find $\eta_{j+1}$. Let $\alpha\in\Pi_1$ be not in $\Pi(\eta_0,\dots,\eta_j)$. Hence by Lemma \ref{lem:image Valpha},
	\begin{equation}\label{equ:alpha new}
		V_{\alpha,\eta_j}=V_{\alpha,\eta_0}=V_{\alpha,\eta}.
	\end{equation}
	Due to Lemma \ref{lem:root partition} and Lemma \ref{lem:image Valpha},
	we have further
		\begin{equation}\label{equ:l new}
		l_{\alpha,\eta_j}=l_{\alpha,\eta_0}=l_{\alpha,\eta}.
		\end{equation}
	We are in the situation of Lemma \ref{lem:line point} with $V=V_{\alpha}$, $x=V_{\alpha,\eta'}$ and $l=l_{\alpha,\eta}$. Due to the hypothesis, Lemma \ref{lem:line point} and Lemma \ref{lem:line orbit}, we can find $\eta_{j+1}$ in the same $\alpha$-circle of $\eta_j$ such that 
	\begin{equation}\label{equ:eta r+1}
		d_{\alpha}(g\eta_{j+1},g\eta')=d(\rho_{\alpha}gV_{\alpha,\eta_{j+1}},\rho_{\alpha}gV_{\alpha,\eta'}) \leq \delta^{-2}\gamma_{1,3}(\rho_{\alpha}g)\leq \delta^{-2}\beta e^{-\alpha\kappa(g)}.
	\end{equation}
	Hence by \eqref{equ:alpha new} and \eqref{equ:eta r+1},
	\begin{equation*}
		d_{\alpha}(g\eta_{j+1},g\eta_j)=d_{\alpha}(g\eta_{j+1},g\eta)=d_{\alpha}(g\eta,g\eta')+O(\delta^{-2}\beta e^{-\alpha\kappa(g)}),
	\end{equation*}
	which is \eqref{equ:djgj}. Please see Figure \ref{fig:change flag}, where an element in the flag variety is represented by a line with a point.
	
	We need to verify the distance between $g\eta_{\nupione}$ and $g\eta'_{\nupitwo}$. Without loss of generality, suppose that $\alpha\in \Pi_1$. Then by Lemma \ref{lem:image Valpha}, the construction and \eqref{equ:eta r+1},
	\[d_\alpha(g\eta_{\nupione},g\eta'_{\nupitwo})=d_\alpha(g\eta_{\nupione},g\eta')=d_\alpha(g\eta_{j+1},g\eta')\leq \delta^{-2}\beta e^{-\alpha\kappa(g)}, \]
	where $j$ is the unique number such that $\alpha(\eta_j,\eta_{j+1})=\alpha$.
\end{proof}
\begin{rem}\label{rem:funpro}
	In the case $\rm{SL}_3(\R)$, we know that $\wedge^2 V_{\alpha_1}$ and $\wedge^2V_{\alpha_2}$ are isomorphic to $V_{\alpha_2}$ and $V_{\alpha_1}$, respectively. The condition in Lemma \ref{lem:chapoi} is equivalent to $\eta,\eta'$ in $B^m_g(\delta)$.
	
	In the case $\slrn$, the representations $V_r=\wedge^r\R^{\rank+1}$ are fundamental representations. Since $\slrn$ is split, $\wedge^2V_r$ is again proximal, but may not be irreducible. In Lemma \ref{lem:xx'w}, we will proceed to give a control on $y^m_{\wedge^2(\wedge^rg)}$. 
	
	The condition of Lemma \ref{lem:chapoi} is not really important, what we need is that the condition is true with a loss of exponentially small measure when we consider the random walks on $\slr$.
\end{rem}
\begin{lem}	\label{lem:getajgetal}
	Under the same assumptions and constructions as in Lemma \ref{lem:chapoi}, if we also have $\eta,\eta'\in B^m_g(\delta)$, then $g\eta_j,g\eta_l'$ are in $b^M_g(C\beta\delta^{-2})$ for $1\leq j\leq \nupione$ and $1\leq l\leq \nupitwo$, where $C>0$ is a constant only depends on the group $\bf G$.
\end{lem}
\begin{proof}
	By hypothesis, Lemma \ref{lem:gBmgmul} implies that $g\eta,g\eta'\in b^M_g(\beta\delta^{-1})$. By \eqref{equ:djgj},
	\[d(g\eta_j,g\eta_{j+1})\leq 2\beta\delta^{-1}+O(\delta^{-2}\beta e^{-\alpha\kappa(g)})= O(\beta\delta^{-2}). \]
	Hence by induction, we have $g\eta_j\in b^M_g(C\beta\delta^{-2})$ for all $j$. Similarly the result holds for $g\eta_l'$. 
\end{proof}
\subsection{Random walks and Large deviation principles}
\label{sec:lardev}
The study of random walks on projective spaces and flag varieties are connected by representation theory. 

Let $X$ be $\P$ or $\bp V$, where $V$ is an irreducible representation of $G$. There is a natural group action of $G$ on $X$. Let $\mu$ be a Borel probability measure on $G$. Then a Borel probability measure $\nu$ on $X$ is called $\mu$-stationary if
\[\nu=\mu*\nu:=\int_G g_*\nu\dd\mu(g), \]
where $g_*\nu$ is the pushforward measure of $\nu$ under the action of $g$ on $X$.
\begin{lem}[Furstenberg]\label{lem:stauni}
	Let $\mu$ be a Zariski dense Borel probability measure on $G$. There exists a unique $\mu$-stationary probability measure $\nu$ on the flag variety and its images in the projective spaces $\bp V$ are the unique $\mu$-stationary probability measures when $V$ is an irreducible representation of $G$.
\end{lem}
See \cite{furstenberg1973boundary}, \cite[Proposition 10.1]{benoistquint} for more details. In order to distinguish stationary measures on different spaces, we use $\nu_V$ to denote a $\mu$-stationary measure on $\bp V$.
\begin{defi}\label{defi:exponential moment}
	Let $\mu$ be a Zariski dense Borel probability measure on $G$. The measure $\mu$ has a finite exponential moment if there exists $t_0>0$ such that
	\[\int_G e^{t_0\|\kappa(g)\|}\dd\mu(g)<\infty. \]
\end{defi}
\begin{rem*}
	This definition coincides with the definition given in the introduction for matrix groups, because in that case $\log\|g\|=\chi\kappa(g)$ where $\chi$ is the highest weight of a faithful representation. This weight $\chi$ is in the dual cone of $\frak a^+$ and $\chi(X)\gg \|X\|$ for $X$ in $\frak a^+$. %
\end{rem*}
\begin{defi}
	Let $\mu$ be a Zariski dense Borel probability measure with finite exponential moment on $G$. The Lyapunov vector $\sigma_\mu$ is defined as the average of the Iwasawa cocycle
	\begin{equation*}
		\sigma_\mu:=\int_{G\times \P}\sigma(g,\eta)\dd\mu(g)\dd\nu(\eta).
	\end{equation*}
\end{defi}
\begin{lem}\label{lem:poslya}
	Let $\mu$ be a Zariski dense Borel probability measure with finite exponential moment on $G$. Then the Lyapunov vector $\sigma_\mu$ is in $\frak a^{++}$, the interior of the Weyl chamber. Equivalently, for any simple root $\alpha$, we have $\alpha(\sigma_\mu)>0$.
\end{lem}
The maximal positivity of Lyapunov vector in Lemma \ref{lem:poslya} is due to Guivarc'h-Raugi \cite{guivarc1985frontiere} and Goldsheid-Margulis \cite{gol1989lyapunov}. See \cite[Corollary 10.15]{benoistquint} for more details.
Lemma \ref{lem:poslya} will be used to show that the action of $G$ on $\P$ is contracting in Section \ref{sec:sumfou}, where the contraction speed is given by $\beta=\sup_{\alpha\in\Pi}\{e^{-\alpha\sigma_\mu} \}$.

In the following proposition, we give the large deviation principle for the Cartan projection. We keep the assumption that \textbf{$\mu$ is a Zariski dense Borel probability measure on $G$ with a finite exponential moment}.
\begin{prop}\label{prop:lardev1}
	For every $\epsilon>0$ there exist $C,c>0$ such that for all $n\in\bb N$ we have
	\begin{align}
	&\label{equ:lardev1}\mu^{*n}\{g\in G|\ \|\kappa(g)-n\sigma_{\mu}\|\geq n\epsilon \}\leq Ce^{-c\epsilon n},
	\end{align}
\end{prop}
See \cite[Thm 13.17]{benoistquint} for more details. 
\begin{prop}\label{prop:large deviation projective}
	If $(\rho,V)$ is an irreducible representation of $G$, then for every $\epsilon>0$ there exist $C,c$ such that for all $x$ in $\bp V$ and $y$ in $\bp V^*$ and $n\geq 1$ we have
	\begin{align}\label{equ:large deviation projective}
		\mu^{*n}\{g\in G|\ \delta(x,y^m_g)\leq e^{-n\epsilon} \}\leq Ce^{-c\epsilon n},\\
		\nonumber\mu^{*n}\{g\in G|\ \delta(x^M_g,y)\leq e^{-n\epsilon} \}\leq Ce^{-c\epsilon n}.
	\end{align}
\end{prop}
See \cite[Prop 14.3]{benoistquint} for more details. We need $\rho$ to be proximal in Proposition \ref{prop:large deviation projective}, but the representation is automatically proximal due to the splitness of $G$.
\begin{prop}\label{prop:large deviatio flag}
	For every $\epsilon>0$ there exist $C,c$ such that for all $\eta,\,\zeta$ in $\PP$ and $n\geq 1$ we have
	\begin{align}
	\label{equ:xgmx}&\mu^{*n}\{g\in G|\ \delta(\eta^M_g,\zeta)\leq e^{-n\epsilon} \}\leq Ce^{-c\epsilon n},\\
	\label{equ:xgmy}&\mu^{*n}\{g\in G|\ \delta(\eta,\zeta^m_g)\leq e^{-n\epsilon} \}\leq Ce^{-c\epsilon n},
	\end{align}
\end{prop}
Proposition \ref{prop:large deviatio flag} is a multidimensional version of Proposition \ref{prop:large deviation projective}.
\begin{prop}[H\"older regularity]\label{prop:holder regulariyt}
	If $(\rho,V)$ is an irreducible representation of $G$, then there exist constants $C>0,\,c>0$ such that for every $y$ in $\bp V^*$ and $r>0$ we have
	\begin{equation}\label{equ:regularity stataionary measure pro}
	\nu_V(\{x\in\bp V|\ \delta(x,y)\leq r \})\leq Cr^{c}.
	\end{equation}
\end{prop}
The proximality of the representation is also needed in Proposition \ref{prop:holder regulariyt}.
This result is due to Guivarc'h \cite{guivarc1990produits}. See \cite[Thm 14.1]{benoistquint} for more details.
As a corollary of Proposition \ref{prop:holder regulariyt}, we have the following.
\begin{cor}\label{cor:regularity}
	If $(\rho,V)$ is an irreducible representation of $G$ with highest weight $\chi$, then there exist constants $C>0,\,c>0$ such that for every $y$ in $\bp V^*$ and $r>0$ we have
	\begin{equation}\label{equ:regularity stataionary measure}
	\nu(\{\eta\in\P|\ \delta(V_{\chi,\eta},y)\leq r \})\leq Cr^{c}.
	\end{equation}
\end{cor}
\begin{proof}
By Lemma \ref{lem:stauni}, we have
\[ \nu(\{\eta\in\P|\ \delta(V_{\chi,\eta},y)\leq r \})=\nu_V(\{x\in\bp V|\ \delta(x,y)\leq r \}).\]
Hence Corollary \ref{cor:regularity} follows from Proposition \ref{prop:holder regulariyt}.
\end{proof}
For later convenience, we introduce the following definition. Let $C_1'>0$ be a constant such that $|\alpha(X)|\leq C_1\|X\|$ for all simple root $\alpha$ and $X\in\frak b$. Let $\constant=3C_1+C_1'$, where $C_1$ is defined in \eqref{equ:norrep}.
\begin{defi}[Good element]\label{defi:good element} 
	For $n\in\bb N, \epsilon>0$ and $\eta, \zeta\in \P$, we say that an element $h$ is $(n,\epsilon,\eta,\zeta)$ good if 
	\begin{equation}\label{equ:good element}
	\|\kappa(h)-n\sigma_\mu\|\leq \epsilon n/\constant\text{ and }\delta(\eta, \zeta^m_h),\delta(\eta^M_h,\zeta)>2e^{-\epsilon n/\constant}.
	\end{equation}
\end{defi} 
\begin{lem}\label{lem:cocbou expoential}	
	We have that $h$ is $(n,\epsilon,\eta,\zeta)$ good outside an exponentially small set, that is to say there exist $C>0,c>0$ such that 
	\[\mucon{n}\{h \text{ is not }(n,\epsilon,\eta,\zeta)\text{ good.} \}\leq Ce^{-c\epsilon n} .\]
\end{lem}
\begin{proof}
	This is due to the large deviation principle \eqref{equ:lardev1}, \eqref{equ:xgmx} and \eqref{equ:xgmy}.
\end{proof}
\begin{lem}\label{lem:cocbou}
	Let $\delta=e^{-\epsilon n}$ and $\expec=\max_{\alpha\in\Pi}e^{-\alpha\sigma_\mu n}$.
	Suppose that $\epsilon$ is small enough such that $\beta<\delta^3$. If $h$ is $(n,\epsilon,\eta,\zeta^m_g)$ good, then
	\[\gap(h)\leq\expec\delta^{-1}\leq\delta^2\text{ and }\|\sigma(gh,\eta)-\kappa(g)-n\sigma_\mu\|\leq \epsilon n. \]
\end{lem}
\begin{proof}
By hypothesis and $\alpha(n\sigma_\mu-\kappa(h))\leq C_1'\|n\sigma_\mu-\kappa(h)\|\leq \epsilon n$, we have
\[\gap(h)=\max_{\alpha\in\Pi}e^{-\alpha\kappa(h)}=\max_{\alpha\in\Pi} e^{-\alpha n\sigma_\mu}e^{\alpha(n\sigma_\mu-\kappa(h))}\leq \expec\delta^{-1}. \]
By Lemma \ref{lem:gBmgmul}, we have $h\eta\in b^M_h(\gap(h)/\delta)\subset b^M_h(\delta)\subset B^m_{g}(\delta)$. Hence by Lemma \ref{lem:iwacar}
	\begin{align*}
	\|\sigma(gh,\eta)-\kappa(g)-n\sigma_\mu\|&=\|\sigma(g,h\eta)-\kappa(g)+\sigma(h,\eta)-n\sigma_\mu \|\\
	&\leq C_1|\log\delta(h\eta,\zeta^m_{g})|
	+C_1|\log\delta(\eta,\zeta^m_{h})|+\|\kappa(h)-n\sigma_\mu
	\|\leq \epsilon n.
	\end{align*}
	The proof is complete.
	%
\end{proof}
For later usage in Section \ref{sec:noncon}, we will define another notion of goodness.
\begin{defi}\label{defi:good element'}
	For $n\in\bb N, \epsilon>0$ and $\zeta\in \P$, we say that an element $h$ is $(n,\epsilon,\zeta)$ good if 
	\begin{equation}\label{equ:good element'}
	\|\kappa(h)-n\sigma_\mu\|\leq \epsilon n/\constant\text{ and }\delta(\eta^M_h,\zeta)>2e^{-\epsilon n/\constant}.
	\end{equation}
\end{defi} 
	\begin{lem}\label{lem:cocbou'} Suppose the semisimple part of $\bf G$ is simply connected.
		Let $\delta=e^{-\epsilon n}$ and $\expec=\max_{\alpha\in\Pi}e^{-\alpha\sigma_\mu n}$.
		There exists a flag $\eta_\alpha$ in $\P$ which is different from $\eta_o$ only in its image in $\bp V_\alpha$ and
		\begin{equation}\label{equ:norlar1con}
		V_{\alpha,\eta_\alpha}=V^{\chi_\alpha-\alpha}.
		\end{equation}
	If $h$ is $(n,\epsilon,\zeta^m_g)$ good, then for $\eta=l_h^{-1}\eta_\alpha$, we have
		\begin{equation}\label{equ:cocbou'}
		e^{\omega_{\alpha'}(\sigma(gh,\eta)-\kappa(g)-n\sigma_\mu)} \in[\delta,\delta^{-1}] \text{ for } \alpha'\neq\alpha\text{ and }
		e^{\omega_\alpha(\sigma(gh,\eta)-\kappa(g)-n\sigma_\mu)}\leq \beta\delta^{-1}.
		\end{equation}
	\end{lem}
	\begin{proof}
		The existence of $\eta_\alpha$ is guaranteed by Lemma \ref{lem:alpha circle}. In the $\alpha$ circle of $\eta_o$, there exists a point $\eta_\alpha$ whose image in $\bp V_\alpha$ is exactly $V^{\chi_\alpha-\alpha}$. This is the $\eta_\alpha$ that we are looking for.
		
		Without loss of generality, we can suppose that $l_h=e$. The image of $\eta_\alpha$ in $\bp V_{\alpha'}$ is the same as $\eta_o$ if $\alpha'\neq\alpha$. Hence by \eqref{equ:representation cocycle}, we have $\omega_{\alpha'}\sigma(gh,\eta_\alpha)=\omega_{\alpha'}\sigma(gh,\eta_o)$ for $\alpha'\neq\alpha$. By \eqref{equ:delta etao}, that is $\delta(\eta_o,\zeta_o)=1$, the element $h$ is $(n,\epsilon,\eta_o,\zeta^m_g)$ good. By Lemma \ref{lem:cocbou}, we obtain the first part of \eqref{equ:cocbou'}.
		
		The image of $\eta_\alpha$ in $\bp V_\alpha$ is $V^{\chi_\alpha-\alpha}$, whose weight is $\chi_\alpha- \alpha$. Hence by \eqref{equ:representation cocycle}, for $v\in V^{\chi_\alpha-\alpha}$
		\begin{equation}\label{equ:etad chid}
			\chi_\alpha\sigma(h,\eta_\alpha)=\log\frac{\|h v\|}{\|v\|}=\log\frac{\|\exp(\kappa(h))v \|}{\|v\|}=(\chi_\alpha- \alpha)\kappa(h).
		\end{equation}
		By \eqref{equ:representation cocycle} and \eqref{equ:representation cartan}, we have $\chi_\alpha(\sigma(g,h\eta)-\kappa(g))\leq 0$.
		Together with \eqref{equ:etad chid},
		\begin{align*}
			&{\chi_\alpha(\sigma(gh,\eta)-\kappa(g)-n\sigma_\mu)}= 	{\chi_\alpha(\sigma(g,h\eta)-\kappa(g))}+	{\chi_\alpha(\sigma(h,\eta)-n\sigma_\mu)}\\
			\leq &{(\chi_\alpha- \alpha)\kappa(h)-n\chi_\alpha\sigma_\mu}={- n\alpha\sigma_\mu}+{(\chi_{\alpha}-\alpha)(\kappa(h)-n\sigma_\mu)}.
		\end{align*}
		By \eqref{equ:good element'} and $\chi_{\alpha}-\omega_\alpha\in \frak c^*$, the proof is complete.
	\end{proof}
\begin{exa}
	In the case $\slrn$,
			\begin{equation*}
			\eta_{\alpha_d}=\{\R e_1\subset \cdots\subset  \R e_1\oplus\cdots\oplus \R e_{d-1}\subset\R e_1\oplus\cdots\oplus \R e_{d-1}\oplus \R e_{d+1}\subset \cdots \},
			\end{equation*}
			and its image in $\wedge^d(\R^{\rank+1})$ is $\R(v_1\wedge \cdots \wedge v_{d-1}\wedge v_{d+1})$.
\end{exa}

	Let $V$ be a representation of $G$. Let $\bb G_2(V):=\{2\text{-planes  in }V \}$ be the Grassmannian variety of $V$. Let $q_\lambda:\wedge^2V\rightarrow \wedge^2V$
\nomentry{$q_\lambda$}{}
 be the $G$-equivariant projection to the sum of all the irreducible subrepresentations of $\wedge^2V$ with highest weight equal to $\lambda$.
\begin{lem}\label{lem:large rep}
	Let $V$ be an irreducible representation of $G$ with highest weight $\chi$. For a simple root $\alpha$, let $q_{2\chi-\alpha}$ be the $G$-equivariant projection from $\wedge^2V$ to $\wedge^2V$. There exists $c>0$ such that for all $v, v'$ in $V$,
	\begin{equation*}
	\sum_{\alpha\in\Pi}\|q_{2\chi-\alpha}(v\wedge v')\|\geq c \|v\wedge v'\|.
	\end{equation*}
\end{lem}

For the proof we need a lemma similar to \cite[Lemma 3.3]{benoistquint_finite_2012}. 
\begin{lem}\label{lem:g2v W} Under the same assumptions as in Lemma \ref{lem:large rep}, then $\bigcap_{\alpha \in \Pi}\ker(q_{2\chi-\alpha})$ does not contain any pure wedge.
\end{lem}
\begin{proof}	
	Let $W'$ be the intersection of all the kernels, that is $W'=\bigcap_{\alpha \in \Pi}\ker (q_{2\chi-\alpha}) $. The two sets $\bb G_2(V)$ and $\bp W'$ are closed subvarieties of $\bp (\wedge^2V)$ and $G$-invariant. Therefore their intersection is again a $G$-invariant closed subvariety which is complete. Let $B$ be the Borel subgroup of $G$, which is solvable. By \cite[Thm.10.4]{borel1990linear}, the action of a solvable algebraic connected group on a complete variety has fixed points.  We claim that the fixed points of $B$ on $\bb G_2(V)$ are the lines with the highest weight. Then the result follows by the fact that these lines do not belong to $W'$.
	
	Suppose that there exist $v,u$ in $V$ such that $v\wedge u$ is $B$-invariant. We can decompose $v,u$ as a sum $v=\sum_{\lambda}v_\lambda$ and $u=\sum_{\lambda}u_\lambda$. Since we can replace $v,u$ by $bv, bu$ for $b$ in $B$, which augments the weight, we can suppose that the component of highest weight $v_\chi$ is non-zero. Since the dimension of $V^\chi$ is $1$, we can suppose that $u_\chi=0$. Let $\rho\neq\chi$ be a highest weight such that $u_\rho$ is nonzero. The $B$-invariance of $\R (v\wedge u)$ also implies that the action of $X_\alpha$, for $\alpha$ simple roots, fixes the line. Hence $X_\alpha(v\wedge u)=X_\alpha v\wedge u+v\wedge X_\alpha u\in \R v\wedge u$. The weight $\chi+\rho+\alpha$ is higher than all the weights appearing in $v\wedge u$, hence $v_\chi\wedge X_\alpha u_\rho=0$ for all simple roots $\alpha$. This implies that $\rho=\chi-\alpha$ for some simple root $\alpha$. Therefore $v\wedge u$ contains $v_\chi\wedge u_{\chi-\alpha}$. Since $v\wedge u$ is also $A$-invariant, all the components in the weight decomposition have the same weight. Hence $v\wedge u=v_\chi\wedge u_{\chi-\alpha} $, which is a vector of highest weight in $\wedge^2 V$.
\end{proof}
\begin{proof}[Proof of Lemma \ref{lem:large rep}]
	By Lemma \ref{lem:g2v W}, we know that $\frac{\sum_{\alpha\in\Pi}\|q_{2\chi-\alpha}(v\wedge v')\|}{ \|v\wedge v'\|}: \bb G_2(V)\rightarrow \bb R_{\geq 0}$ is a positive continuous function. Since $\bb G_2(V)$ is a compact space, on which a positive continuous function has a lower bound, the result follows.
\end{proof}
We want to prove a large deviation principle for a special reducible representation. This lemma will be used in Lemma \ref{lem:g good} to control $y^m_{\wedge^2g}$ in Lemma \ref{lem:distance} and Lemma \ref{lem:chapoi}. 
\begin{lem}\label{lem:xx'w}
	Let $V$ be a super proximal representation of $G$ (Definition \ref{defi:super proximal}). For $\epsilon>0$ there exist $C,c>0$ such that the following holds. For $x=\bb Rv,x'=\bb Rv'\in\bb PV$ with $x\neq x'$, we have
	\[\mucon{n}\{g\in G|\delta(x\wedge x',y^m_{\wedge^2\rho(g)})<e^{-\epsilon n} \}\leq Ce^{-c\epsilon n} .\]
\end{lem}
Due to Definition \ref{defi:super proximal}, there is only one simple root $\alpha$ such that $q_{2\chi-\alpha}(\wedge^2V)$ is non-zero. Write $\wedge^2V=W\oplus W'$, where $W$ is the irreducible representation generated by the vector corresponding to the highest weight in $\wedge^2V$, and $W'$ is the $G-$invariant complementary subspace. Then $q_{2\chi-\alpha}(\wedge^2V)=W$, and we write $Pr_W=q_{2\chi-\alpha}$.
\begin{proof}[Proof of Lemma \ref{lem:xx'w}]
	By \eqref{equ:ym rhog}, we see that a non-zero vector in $y^m_{\wedge^2g}$ vanishes on $W'$ and $y^m_{\wedge^2g}$ can be seen as an element in $\bp W^*$. We only need to consider the projection of $v\wedge v'$ onto $W$ and use the large deviation principle \eqref{equ:large deviation projective}. By Lemma \ref{lem:large rep},
	\[\delta(x\wedge x',y^m_{\wedge^2 g})=\frac{|f(v\wedge v')|}{\|v\wedge v'\|}=\frac{|f(Pr_W(v\wedge v'))|}{\|Pr_W(v\wedge v')\|}\frac{\|Pr_W(v\wedge v')\|}{\|v\wedge v'\|}\geq c\delta(Pr_W(x\wedge x'),y^m_{\wedge^2g}), \]
	where $f$ is a unit vector in $y^m_{\wedge^2g}$.
	The proof is complete.
\end{proof}
\section{Non-concentration condition}
\label{sec:noncon}
We want to verify the main input for the sum-product estimate, the non-concentration condition, in this section. The semisimple part of $\bf G$ is supposed to be simply connected.

For the first time, the reader can neglect $g$ in the left of $h$ and think of the semisimple case $\slrn$. The main idea of the proof is already there. Adding $g$ is a technical step, which is needed in its application. (We only need an additional condition on $\eta^M_h$ to control $\kappa(gh)$.)
\subsection{Projective, Weak and Strong non-concentration}
Recall that $\rank$ is the semisimple rank of $G$ and $\chi_1,\cdots,\chi_\rank$ are fixed weights, where we change the subscript from $\alpha\in \Pi$ to $i\in \{1,\cdots,\rank \}$. The set $\{\omega_i\}_{1\leq i\leq\rank}$ are the extension of fundamental  weights $\tilde\omega_i$ to $\frak a$ which vanishes on $\frak c$ and the restriction of $\omega_i$ and $\chi_i$ to $\frak b$ coincide with $\tilde\omega_i$.
Recall that $\alpha_1,\cdots ,\alpha_\rank$ are the simple roots of $\frak a^*$. For two vectors $x=(x_1,\cdots,x_\rank)$ and $x'=(x_1',\cdots,x_\rank')$ in $\R^\rank$, we write $xx'$ for their product $(x_1x_1',\cdots,x_\rank x_\rank')$ and $\l x,x'\r=\sum_{1\leq j\leq \rank}x_1x_1'$ for their inner product.

In order to distinguish different objects, we will use capital letter $X$ to denote functions or random variables and use small letter $x$ to denote vectors or indeterminates.

Let $L$ be the $m\times m$ square matrix which changes the basis $(\omega_1,\cdots,\omega_\rank)$ of $\frak b^*$ to the basis $(-\alpha_1,\cdots,-\alpha_\rank)$, that is $L_{ij}=-\alpha_i(H_j)$. Then $L$ is an integer matrix. Hence, we can define $E_d$, a rational map from $(\R^*)^\rank$ to $(\R^*)^d$ with $1\leq d\leq m$, which is given by $y=E_d(x)$ for $x\in( \R^*)^\rank$ where
\[y_i=\Pi_{1\leq j\leq \rank} x_j^{L_{ij}}. \] 

Fix an element $g$ in $G$. Let 
\begin{align*}
	X_g(n,h,\eta)&=(e^{\omega_1(\sigma(gh,\eta)-\kappa(g)-n\sigma_\mu)},\dots,e^{\omega_\rank(\sigma(gh,\eta)-\kappa(g)-n\sigma_\mu)}),
	\\
	\sigmah_g(h,\eta)&=(e^{-\alpha_1(\sigma(gh,\eta)-\kappa(g)-n\sigma_\mu)},\dots,e^{-\alpha_\rank(\sigma(gh,\eta)-\kappa(g)-n\sigma_\mu)})
\end{align*}
for $\eta$ in $\PP$ and $h$ in $G$. By definition, $E_dX_g(n,h,\eta)$ is the vector which is composed of the first $d$ components of $\sigmah_g(h,\eta)$, that is 
\begin{equation}\label{equ:ed sig}
p_d\sigmah_g(h,\eta)=E_dX_g(n,h,\eta), 
\end{equation}
where $p_d:\R^\rank\rightarrow \R^d$ is the map which takes a vector $x$ of $\R^\rank$ to the vector of $\R^d$ composed of the first $d$ components of $x$.
In the following argument $g$ is fixed or $g$ equals identity. Hence we will abbreviate $X_g,\sigmah_\g, \sigmah_e$ to $X,\sigmah, \sigmah_0$. 

We define an affine determinant $A_d$ on $(\R^d)^{d+1}$. 
For $d+1$ vectors $y^1,\cdots,y^{d+1}$ in $\R^d$, let $A_d$ be the determinant of the $(d+1)\times (d+1)$ matrix
$\begin{pmatrix}
	y^1 & \cdots & y^{d+1}\\
	1 & \cdots & 1
\end{pmatrix}$, which is the volume of the $d+1$-dimensional parallelogram generated by vectors $(y^i,1)$ for $i=1,\dots,d+1$. Let $e_i$ be the vector in $\R^d$ with only $i$-th coordinate non-zero and equal to $1$. By identifying $e_1\wedge\cdots\wedge e_d$ with number $1$, we can also define $A_d$ by
\[ A_d(y^1,\cdots,y^{d+1})=\sum_{1\leq i\leq d+1}(-1)^{i+d+1}y^1\wedge\cdots\wedge\widehat{y^i}\wedge\cdots \wedge y^{d+1}. \]
For $d+1$ vectors $x^1,\cdots,x^{d+1}$ in $\R^\rank$, let $B_d$ be the rational function defined by
\[B_d(x^1,\cdots, x^{d+1})=A_d(E_dx^1,\cdots, E_dx^{d+1}). \]

We introduce the notation 
\[\bfh{d+1}=(h_1,\dots,h_{d+1}),\] 
which is an element in $G^{ (d+1)}$. Let
\[A_d^n(\bfh{d+1},\eta):=B_d(X(n,h_1,\eta), \dots, X(n,h_{d+1},\eta)).\]

\begin{defi}\label{defi:prononcon} We say that $\mu$ satisfies the projective non-concentration (PNC) on dimension $d$, if for every $\epsilon>0$ there exist $c,C>0$ such that for all $n$ in $\bb N$, $\eta$ in $\PP$ and $g$ in $G$
	\[\sup_{a\in\bb R,v\in\bb S^{d-1}}\mu^{*n}\{h\in G| |\l v,\sigmah(h,\eta)\r-a|\leq e^{-\epsilon n} \}\leq Ce^{-c\epsilon n},\]
	where $v$ is regarded as a vector in $\R^d\times \{0\}^{m-d}\subset \R^m$.
\end{defi}
More geometrically, this is equivalent to say that the probability of $\sigmah(h,\eta)$ being close to an affine hyperplane is exponentially small.
\begin{defi}\label{defi:cocnoncon} We say that $\mu$ satisfies the weak non-concentration (WNC) on dimension $d$, if for every $\epsilon>0$ there exist $c,C>0$ such that for all $n$ in $\bb N$, $\eta$ in $\PP$ and $g$ in $G$
	\[(\mu^{*n})^{\otimes (d+2)}\{(\bfh{d+1},\ell)\in G^{(d+2)}||A^n_d(\bfh{d+1},\ell\eta)|\leq e^{-\epsilon n} \}\leq Ce^{-c\epsilon n}.\]
\end{defi}
\begin{defi}\label{defi:strong} We say that $\mu$ satisfies the strong non-concentration (SNC) on dimension $d$, if for every $\epsilon>0$ there exist $c,C>0$ such that for all $n$ in $\bb N$, $\eta$ in $\PP$ and $g$ in $G$
	\[(\mu^{*n})^{\otimes(d+1)}\{\bfh{d+1}\in G^{ (d+1)}|| A^n_d(\bfh{d+1},\eta)|\leq e^{-\epsilon n} \}\leq Ce^{-c\epsilon n}.\]
\end{defi}
We will proceed by induction. When $d=0$, we make the convention that $A_0^d=1$ and it is trivial that SNC holds. Then
\begin{itemize}
	\item SNC on dimension $d$ $\Rightarrow$ WNC on dimension $d$ (By definition)
	\item PNC on dimension $d$ $\Leftrightarrow$ SNC on dimension $d$ (Lemma \ref{lem:PNC}) 
	\item WNC on dimension $d$ $\Rightarrow$  PNC on dimension $d$ (Lemma \ref{lem:WNC})
	\item SNC on dimension $d-1$ $\Rightarrow$ WNC on dimension $d$ (Lemma \ref{lem:SNC}).
\end{itemize} 
In the above implications, the constants $C,c$ will change. We can conclude: 
\begin{prop}\label{prop:PNC}
	Let $\mu$ be a Zariski dense Borel probability measure on $G$ with exponential moment. Then $\mu$ satisfies PNC on dimension $m$.
\end{prop}
\subsection{Away from affine hyperplanes}
We need a lemma of linear algebra, which relates different non-concentrations. This lemma is already known from \cite[Lemma 7.5]{eskin2005growth}. For two subsets $A,B$ of a metric space $(X,d)$, the distance between $A$ and $B$ is defined as
\[d(A,B)=\inf_{x\in A, y\in B}d(x,y). \]
\begin{lem}\label{lem:affvol}
	Let $C>0,c>0$. Let $u_1, \cdots, u_{d+1}$ be vectors in $\R^d$ with length less than $C$. Consider the following conditions:
	\begin{itemize}
		\item[i.] There exists an affine hyperplane $l$ such that for $i=1,\dots,d+1$,
			\[d(u_i,l)\leq c. \]
		\item[ii.] We have
	\begin{equation*}
		\left\|\sum_{1\leq i\leq d+1}(-1)^iu_1\wedge \cdots \wedge \widehat{u_i}\wedge\cdots\wedge u_{d+1}\right\|\leq c,
	\end{equation*}
	where $\widehat{u_i}$ means this term is not in the wedge product.
	 \item[iii.] There exists $i$ in $\{1,\dots,d \}$ such that
	 \[ d(u_i,\sp_{\rm{aff}}(u_{d+1},u_1,\dots, u_{i-1}) )\leq c, \]
	 where $\sp_{\rm{aff}}$ is the affine subspace generated by the elements in the bracket.
\end{itemize} 
Then $i(c)\Rightarrow ii(2^{d+1}C^{d-1}c)$, $ii(c)\Rightarrow iii(c^{1/d})$ and $iii(c)\Rightarrow i(c)$.
\end{lem}
\begin{proof}
	We first transfer the affine problem to a linear problem. Let $v_i=u_i-u_{d+1}$ for $i=1,\dots,d$. Then $v_i$ are vectors with length less than $2C$. The above three conditions are equivalent to (with change of constants in $i$) 
		\begin{itemize}
			\item[i'.] There exists a linear subspace $l$ of codimension $1$ such that for $i=1,\dots,d$
			\[d(v_i,l)\leq c. \]
			\item[ii'.] We have
			\begin{equation*}
			\|v_1\wedge \cdots\wedge v_d\|\leq c.
			\end{equation*}
			\item[iii'.] There exists $i$ such that
			\[ d(v_i,\sp(v_1,\dots, v_{i-1}) )\leq c, \]
			where $\sp$ is the linear subspace generated by the elements in the bracket.
		\end{itemize} 
		
	 $iii'(c)\Rightarrow i'(c)$: Let  the hyperplane $l$ be $\sp(v_1,\cdots, \hat{v_i},\cdots, v_d)$. Then $i'(c)$ follows from $iii'(c)$.
	 
	 $i'(c)\Rightarrow ii'(2^dC^{d-1}c)$: Due to $i'$, 
	 the volume of the parallelogram generated by $\{v_i\}_{1\leq i\leq d}$ is less than $(2C)^{d-1}2c$, which is $ii'$.
	 
	 $ii'(c)\Rightarrow iii'(c^{1/d})$: Due to the same argument as in Lemma \ref{lem:volume area}, we have a formula for the volume,
	 \[\|v_1\wedge\cdots\wedge v_d\|=\Pi_{1\leq i\leq d}d(v_i,\sp(v_1,\dots,v_{i-1})),  \]
	 from which the result follows.
\end{proof}
As a corollary, we have the following, which is general and deals with random variables. 
\begin{cor}\label{cor:3.6}	
	 Let $X_1,\dots, X_{d+1}$ be i.i.d. random vectors in $\R^d$ bounded by $C>0$. Let $l$ be an affine hyperplane in $\R^d$. Then for any $c>0$, we have
	\begin{equation}\label{equ:projective from strong}
		\bb P\{d(X_1,l)<c \}^{d+1}\leq \bb P\{\|\sum(-1)^i X_1\wedge\cdots \wedge\hat{X}_i\wedge\cdots\wedge X_{d+1}\|<2^{d+1}C^{d-1}c \},
	\end{equation}
	and
	\begin{equation}\label{equ:strong from projective}
	\begin{split}
		&\bb P\{\|\sum(-1)^i X_1\wedge\cdots\wedge\hat{X}_i\wedge\cdots\wedge X_{d+1}\|<c \}\\
		&\leq\sum_{1\leq i\leq d}\bb P\{d(X_i,\sp_{\rm{aff}}(X_{d+1},X_1,\cdots,X_{i-1}))<c^{1/d} \}.
		\end{split}
	\end{equation}
\end{cor}
\begin{lem}\label{lem:PNC}
	PNC on dimension $d$ is equivalent to SNC on dimension $d$. 
\end{lem}
\begin{proof}
	Let $X_i=E_dX(n,h_i,\eta)$ for $i=1,\cdots, d+1$, where $h_i$ has distribution $\mucon{n}$. Due to Lemma \ref{lem:cocbou}, with a loss of exponentially small measure, we can suppose that the $X_i$ are bounded by $C=e^{\epsilontwo n}$, where $\epsilontwo =\epsilon/(2d)$. 
	
	Due to \eqref{equ:ed sig}, we have $\l v,\sigmah(h_i,\eta) \r=\l p_dv,E_dX(n,h_i,\eta)\r=\l p_dv,X_i\r$. PNC asks exactly that the probability that $X_i$ is close to a hyperplane is small. SNC means that $$|A_d^n(\bf h_{d+1},\eta)|=|A_d(X_1,\dots,X_{d+1})|=\|\sum(-1)^i X_1\wedge\cdots \wedge\hat{X}_i\wedge\cdots\wedge X_{d+1}\|$$
	is small. By \eqref{equ:projective from strong}, PNC on dimension $d$ follows from SNC on dimension $d$.
	
	By \eqref{equ:strong from projective}, SNC on dimension $d$ follows from PNC on dimension $d$.
\end{proof}
\begin{rem}\label{rem:fini SNC}
	We explain that SNC implies the stronger form of SNC, which will be used later. Let $\rm{O}(d)$ be the orthogonal group in dimension $d$. The stronger form of SNC says that for any $(\rho_1,\cdots,\rho_{d+1})\in \rm O(d)^{\times(d+1)}$, we have
	\[(\mu^{*n})^{\otimes(d+1)}\{\bfh{d+1}\in G^{ (d+1)}|| A_d(\rho_1E_dX(n,h_1,\eta),\dots,\rho_{d+1}E_dX(n,h_{d+1},\eta)|\leq e^{-\epsilon n} \}\leq Ce^{-c\epsilon n}.\]
	By Lemma \ref{lem:PNC}, SNC implies PNC. We adopt the notation in the proof of Lemma \ref{lem:PNC}. By Lemma \ref{lem:affvol} and the fact that $\rm{O}(d)$ preserves the distance,
	\begin{align*}
		&\bb P\{\|\sum(-1)^i \rho_1X_1\wedge\cdots\widehat{\rho_iX_i}\cdots\wedge \rho_{d+1}X_{d+1}\|<c \}\\
		\leq &\sum_{1\leq i\leq d}\bb P\{d(\rho_iX_i,l_i)<c^{1/d}\}=\sum_{1\leq i\leq d}\bb P\{d(X_i,\rho_i^{-1}l_i)<c^{1/d}\},
	\end{align*}
	where $l_i=\sp_{\rm{aff}}(\rho_{d+1}X_{d+1},\rho_1X_1,\cdots,\rho_{i-1}X_{i-1})$. Therefore SNC implies the stronger form of SNC.
\end{rem}
WNC is weaker than SNC, because WNC takes an extra average over $\ell$ with respect to measure $\mucon{n}$.
\begin{lem}\label{lem:WNC}
	WNC on dimension $d$ implies PNC on dimension $d$.
\end{lem}
%
 %
\begin{proof}
	Let $\delta=e^{-\epsilon n}$. We first prove the result for $2n$. Recall that $h$ is a random variable which takes values in $G$ with the distribution $\mucon{2n}$.
	Let $h=\ell_1\ell$ such that $\ell_1$ and $\ell$ both have distribution $\mu^{*n}$. Then the cocycle property implies $\sigmahh^{2n}(h,\eta)=\sigmahh^{2n}(\ell_1\ell,\eta)=\sigmah(\ell_1,\ell\eta)\sigmah_0(\ell,\eta)$.  Fubini's theorem implies
	\begin{align*}
	E:=&\sup_{a,v}\mucon{2n}\{h|\l v,\sigmahh^{2n}(h,\eta)\r\in B(a, \delta) \}\\
	\leq  &\int_G\sup_{a,v}\mucon{n}\{\ell_1|\l v,\sigmah(\ell_1,\ell\eta)\sigmah_0(\ell,\eta)\r\in B(a, \delta)\}\dd\mucon{n}(\ell).
	\end{align*}
	The cocycle property is crucial here. Fix $\ell$ and fix $a,v$. 
	We can write 	$$\l v,\sigmah(\ell_1,\ell\eta)\sigmah_0(\ell,\eta)\r=R\l {v'},\sigmah(\ell_1,\ell\eta)\r,$$
	where $R=\|v\sigmah_0(\ell,\eta)\|\geq \min_{1\leq j\leq d}|\sigmah_0(\ell,\eta)_j|$. Here $v'$ is a vector of norm 1, defined by $v'=v\cdot \sigmah_0(\ell,\eta)/R$, depending on $v,l$ and $\eta$. By Lemma \ref{lem:cocbou expoential} and Lemma \ref{lem:cocbou}, for $\ell$ outside an exponentially small set independent of $a,v$, we have $R\geq \delta^{1/2}$. Therefore the condition becomes $\l v',Y^n(\ell_1,\ell \eta)\r\in R^{-1}B(a,\delta)\subset B(a/R,\delta^{1/2})$ and we have 
	\begin{align}\label{equ:Eleq}
	E\leq \int_G\sup_{a,v}\mucon{n}\{\ell_1|\l v,\sigmah(\ell_1,\ell\eta)\r\in B(a, \delta^{1/2}) \}\dd\mucon{n}(\ell)+O_\epsilon(\delta^{c}),
	\end{align}
	where $c>0$ comes from the large deviation principle (Lemma \ref{lem:cocbou expoential}). By H\"older's inequality,
	\begin{equation}\label{equ:int g d+1}
	\begin{split}
		&\int_G\sup_{a,v}\mucon{n}\{\ell_1|\l v,\sigmah(\ell_1,\ell\eta)\r\in B(a, \delta^{1/2}) \}\dd\mucon{n}(\ell)\\
		&\leq \left(\int(\sup_{a,v}\mucon{n}\{\ell_1|\l v,\sigmah(\ell_1,\ell\eta)\r\in B(a, \delta^{1/2}) \})^{d+1}\dd\mucon{n}(\ell)\right)^{1/(d+1)}.
	\end{split}
	\end{equation}
	
	By the same argument as in the proof of Lemma \ref{lem:PNC}
	\begin{align*}
		&\sup_{a,v}\mucon{n}\{\ell_1|\l v,\sigmah(\ell_1,\ell\eta)\r\in B(a, \delta^{1/2}) \}^{d+1}\\
		&\leq (\mucon{n})^{\otimes(d+1)}\{(\bfh{d+1})||A^n_d(\bfh{d+1},\ell\eta)|\leq 2\delta^{1/4} \}+O_\epsilon(\delta^c).
	\end{align*}
	Therefore, by \eqref{equ:Eleq} and \eqref{equ:int g d+1}, we have
	\begin{align*}
	E^{d+1}\leq  (\mucon{n})^{\otimes(d+1)}\otimes\mucon{n}\{(\bfh{d+1},\ell)||A^n_d(\bfh{d+1},\ell\eta)|\leq 2\delta^{1/4} \}+O_\epsilon(\delta^{c}).
	\end{align*}
	The proof for $2n$ is complete.
	
	It remains to prove the same result for $2n+1$. Let $h=\ell_1\ell$ such that $\ell$ has distribution $\mucon{(n+1)}$ and $\ell_1$ has distribution $\mucon{n}$. Following the same argument, we have
	\begin{align*}
	E^{d+1}\leq  (\mucon{n})^{\otimes(d+1)}\otimes\mucon{(n+1)}\{(\bfh{d+1},\ell)||A^n_d(\bfh{d+1},\ell\eta)|\leq 2\delta^{1/4} \}+O_\epsilon(\delta^{c}).
	\end{align*}
	Since $\ell$ only changes the position $\eta$, the uniformity of WNC implies that 
	\begin{align*}
		&(\mucon{n})^{\otimes(d+1)}\otimes\mucon{(n+1)}\{(\bfh{d+1},\ell)||A^n_d(\bfh{d+1},\ell\eta)|\leq 2\delta^{1/4} \}\\
		&=  \int_{\ell_3\in G}(\mucon{n})^{\otimes(d+1)}\otimes\mucon{n}\{(\bfh{d+1},\ell_2)||A^n_d(\bfh{d+1},\ell_2(\ell_3\eta))|\leq 2\delta^{1/4} \}\dd\mu(\ell_3)\ll_\epsilon \delta^c.
	\end{align*}
	The proof is complete.
\end{proof}
\subsection{H\"older regularity}
In this section, we will prove
\begin{lem}\label{lem:SNC}
	SNC on dimension $d-1$ implies WNC on dimension $d$.
\end{lem}

Using other representations, we can get more information on the Iwasawa cocycle. This idea has already been used in \cite{aoun2013transience} for problem concerning transience of algebraic subvariety of split real Lie groups. It is also used in the work of Bourgain-Gamburd on the spectral gap of dense subgroups in $\rm{SU}(n)$, for establishing transience of subgroups. 

The key tool is the following estimate. See \cite[Proposition 14.3]{benoistquint} or \cite{guivarc1990produits} for example.
\begin{lem}\label{lem:lindev}
	Let $V$ be an irreducible representation of $G$. Let $\mu$ be a Zariski dense Borel probability measure on $G$ with exponential moment. For every $\epsilon>0$ there exist $c,C>0$ such that for all $n\in\bb N$, $v$ in $V$ and $f$ in $V^*$ we have
	\[\mucon{n}\{\ell\in G|\, |f(\ell v)|\leq \|f\|\|\ell v\|e^{-\epsilon n} \}\leq Ce^{-c\epsilon n}. \]
\end{lem}

In this part, we write $V_j=V_{\alpha_j}$ for the fixed representation in Lemma \ref{lem:tits} and we write $V_{j,\eta}$ for the image of $\eta\in\P$ in $\bp V_j$ for $j=1,\dots, \rank$. Let $v^j$ be a nonzero vector in $V_{j,\eta}$. For $\ell$ in $G$, we abbreviate $\rho_j(\ell)v^j$ to $\ell v^j$. Since $v^j$ lives in $V_j$, we use the same symbol $\|\cdot\|$ for norms on different $V_j$, which makes no confusion. For a vector $x$ in $\bb R^m$, we denote by $x_i$ the $i$-th coordinate. We use upper script to denote different vectors. We want to replace $\omega_j$ by $\chi_j$, because $\chi_j\sigma(g,\eta)$ has a nice interpretation using representations \eqref{equ:representation cartan}. Let $\chi_j^c=\chi_j-\omega_j$, which vanishes on $\frak b$.

Before proving Lemma \ref{lem:SNC}, we introduce some linear algebra. We want to construct a linear form.	
Recall that $E_d$ is a rational map, $A_d$ is the affine determinant, $B_d$ is the composition of $A_d$ and $E_d$ and $$A_d^n(\bfh{d+1},\eta):=B_d(X(n,h_1,\eta), \dots, X(n,h_{d+1},\eta)),$$
where 
\begin{equation}\label{equ:xnh eta}
X(n,h,\eta)=(e^{\omega_j(\sigma(gh,\eta)-\kappa(g)-n\sigma_\mu)})_{1\leq j\leq \rank}=\left(\frac{e^{\chi_j^c(-c(h)+n\sigma_\mu)}}{e^{\chi_j(\kappa(g)+n\sigma_\mu)}}\frac{\|ghv^j\|}{\|v^j\|}\right)_{1\leq j\leq \rank},
\end{equation}
and the second equality is due to \eqref{equ:representation cocycle} and
\begin{align*}
	\chi_j^c(\sigma(gh,\eta)-\kappa(g)-n\sigma_\mu)=\chi_j^c(c(gh)-c(g)-n\sigma_\mu)=\chi_j^c(c(h)-n\sigma_\mu).
\end{align*}
Let 
\begin{equation}\label{equ:xnh etaj}
 X^i(n,\eta):=X(n,h_i,\eta).
 \end{equation}
In order to use Lemma \ref{lem:lindev}, we need to linearise some function related to $A_d^n(\bfh{n+1},\eta)$ with $\bfh{n+1}$ fixed. We will multiply $B_d$ by its denominator and all the Galois conjugates to get a polynomial on $\|X^i_j\|^2$, which can be realized as a linear functional. 

The function $B_d$ can be seen as a rational function of $$(\underline{x} ):=(x^1,\cdots, x^{d+1})=(x^i_j)_{1\leq i\leq d+1,1\leq j\leq\rank}\text{, where }x^i=X^i(n,\eta).$$ 
By definition, $B_d$ has a special form. Each term in $B_d$ can be expressed as a quotient of two monomials. Let $D_d$ be the lowest common denominator of $B_d$ such that $D_dB_d$ is a polynomial on $(\underline{x})$. In other words, suppose that 
\[B_d(x^1,\cdots,x^{d+1})=\sum_{\bf n\in \bb Z^{\rank(d+1)}}b_{\bf n}\prod_{1\leq j\leq\rank,1\leq i\leq d+1}(x^i_j)^{n_{ij}}, \]
where $\bf n$ is a multi index and $b_{\bf n}$ is the coefficient. Let $q_{ij}=\sup_{\bf n\in \bb Z^{\rank(d+1)}}\{-n_{ij},0\}$ for $1\leq j\leq\rank,1\leq i\leq d+1$. Then $D_d(\underline{x})=\Pi_{1\leq j\leq\rank,1\leq i\leq d+1}(x^i_j)^{q_{ij}}$.

\begin{defi}\label{defi:mul homogeneous}
	Let $F$ be a polynomial on $(x^1,\cdots , x^k)$ where $x^1,\cdots, x^k$ are vectors in $\R^n$. Then we call $F$ a multi-homogeneous polynomial of degree $\bf q=(q_1,\cdots, q_n)\in\bb N^n$ if for $\xi$ in $(\R^*)^n$ we have
	\[F(\xi x^1,\cdots, \xi x^k)=\xi^{\bf q}F(x^1,\cdots, x^k), \]
	where $\xi^{\bf q}=\Pi_{1\leq j\leq n}\xi_j^{q_j}$.
\end{defi}

Let $\fini$ be the finite group $(\bb Z/2\bb Z)^{d(d+1)}$ which acts on $\R^{d(d+1)}$ by changing the sign. Let $(\underline{y}):=(y^1,\cdots ,y^{d+1})=(y^i_j)_{1\leq i\leq d+1,1\leq j\leq d}\in (\R^d)^{d+1}$.  
For $\rho\in \fini$, we write $\rho(\underline{y})$ for the action on the coefficient $y^i_j$, which is of dimension $d(d+1)$. Due to the definition of $\fini$, the product $\Pi_{\rho\in \fini}A_d\rho(y^1,\dots, y^{d+1})$ is invariant under the action $\fini$, hence it is a polynomial on $(y^i_j)^2$.
Let
\begin{equation}\label{equ:Fd definition}
F_d(x^1,\dots, x^{d+1}):=\prod_{\rho\in \fini} D_d(\underline{x})A_d\rho(E_dx^1,\dots, E_dx^{d+1}),
\end{equation}
then
\begin{lem}\label{lem:multi homogeneous}
	$F_d$ is a multi-homogeneous polynomial on $((x^1)^2, \cdots, (x^{d+1})^2)$ with degree $\bf q=(q_1,\cdots, q_\rank)\in \bb N^\rank$ such that $\xi^{\bf q}=\left(\det(E_d(\xi))D_d(\xi,\cdots,\xi)\right)^{2^{d(d+1)}}$.
\end{lem}
\begin{proof}
	We only need to verify that $F_d$ is a multi-homogeneous polynomial. The fact that the determinant is a multilinear function implies that for $\lambda$ and $y^i$ in $\R^d$
	\begin{equation}\label{equ:Ad multilinear}
	A_d(\lambda y^1,\cdots, \lambda y^{d+1})=\det(\lambda)A_d(y^1,\cdots,y^{d+1}),
	\end{equation}
	where $\det(\lambda)=\lambda_1\cdots \lambda_d$. The functions $E_d$ and $D_d$ are group morphisms due to definition. Hence we have 
	\begin{equation}\label{equ:Ed morphism}
	E_d(\xi x)=E_d(\xi)E_d(x)  \text{ and }D_d(\xi x^1,\cdots, \xi x^{d+1})=D_d(\xi,\cdots, \xi)D_d(x^1,\cdots, x^{d+1}).
	\end{equation}
	Therefore by \eqref{equ:Fd definition}, \eqref{equ:Ad multilinear} and \eqref{equ:Ed morphism}, for $\xi$ and $x^i$ in $\R^\rank$,
	\begin{align*}
	F_d(\xi x^1,\cdots, \xi x^{d+1})&=\prod_{\rho\in\fini}D_d(\xi x^1,\cdots,\xi x^{d+1})A_d\rho(E_d(\xi x^1),\cdots, E_d(\xi x^{d+1}))\\
	& =\prod_{\rho\in\fini} D_d(\xi x^1,\cdots,\xi x^{d+1})A_d\rho(E_d(\xi)E_d(x^1),\cdots, E_d(\xi)E_d(x^{d+1}))\\
	&=\prod_{\rho\in\fini} D_d(\underline{x})A_d\rho(E_d(x^1),\cdots, E_d(x^{d+1}))\det(E_d(\xi))D_d(\xi,\cdots,\xi)\\
	&=\xi^{\bf q}F_d(x^1,\cdots,x^{d+1}),
	\end{align*}
	where $\bf q$ is a vector in $\bb N^\rank$ such that $\xi^{\bf q}=\left(\det(E_d(\xi))D_d(\xi,\cdots,\xi)\right)^{|\Gamma|}$.
\end{proof}
For $\bfh{d+1}\in G^{ (d+1)}$ and $\eta$ in $\P$, we write
\[ F(\bfh{d+1},\eta)=F_d(X^1(n,\eta),\dots, X^{d+1}(n,\eta)). \] 
Fix $\bfh{d+1}$. By \eqref{equ:xnh eta} and \eqref{equ:xnh etaj}, $F$ is a function on $v^j$ for $1\leq j\leq \rank$. Recall that $v^j$ are vectors in $V_{j,\eta}$. Let 
\begin{equation*}
F_0(v^1,\cdots, v^\rank)=F(\bfh{d+1},\eta)\Pi_{1\leq j\leq\rank}\|v^j\|^{2q_j}.
\end{equation*}
Now, we want to explain how to realize $F_0$ as a linear functional. 
\begin{lem}\label{lem:linear functional}
	There exists a linear functional $F_1$ on the space $V_0=\bigotimes_{1\leq j\leq \rank}(Sym^2V_j)^{\otimes q_j}$ such that $$F_1(\otimes_j ((v^j)^2)^{\otimes q_j})=F_0(v^1,\cdots, v^\rank).$$
\end{lem}
\begin{proof}
Since $F_d$ is a multi-homogeneous polynomial (Lemma \ref{lem:multi homogeneous}), it is sufficient to prove that every monomial in $F_d$ has the same property. By Definition \ref{defi:mul homogeneous}, a monomial of $F_d$ is of the form
\[\Pi_{1\leq j\leq \rank}\Pi_{1\leq i\leq d+1}(x^i_j)^{2n_{ij}}, \]
with $n_{ij}\in\bb N$ and $\sum_{1\leq i\leq d+1} n_{ij}=q_j$. The term $\Pi \|v^j\|^{2q_j}$ is used to compensate $\|v^j\|$ in the denominator of $X^i_j$ in \eqref{equ:xnh eta}. Now, by multiplying $\|v^j\|$, we can view $X^i_j$ as $\|gh_iv^j\|$ with some coefficient.
By \eqref{equ:xnh eta} and $\|ghv^j\|^2=\l ghv^j,ghv^j \r$, the function $(X^i_j)^2$ is a linear functional on $Sym^2V_j$. Hence $\Pi_{1\leq i\leq d+1}(X^i_j)^{2n_{ij}}$ is a linear functional on $(Sym^2V_j)^{\otimes q_j}$. This is because if we have two linear functionals $f_1$ and $f_2$ on $W_1$ and $W_2$, then $f_1f_2$ is the linear functional on $W_1\otimes W_2$ given by $f_1f_2(w_1\otimes w_2)=f_1(w_1)f_2(w_2)$. Then by the same reason, the monomial $\Pi_{i,j}(X^i_j)^{2n_{ij}}$ is a linear functional on $V_0$. In order to express the linearity of $F_0$, we rewrite
\[F_1(\otimes_j ((v^j)^2)^{\otimes q_j}):=F_0(v^1,\cdots, v^\rank), \]
where $v^j$ is in $V_{j,\eta}$ and $F_1$ is understood as a linear functional on $V_0$.
\end{proof}
\begin{proof}[Proof of Lemma \ref{lem:SNC}]
	Recall $\expec=\max_{\alpha\in\Pi}e^{-\alpha\sigma_\mu n}$. Let $\delta=e^{-\epsilontwo  n}$, where the constant $\epsilontwo $ will be determined later depending on $\epsilon$. We suppose that $n$ is large enough such that $\delta\leq 1/2$, because for small $n$, WNC can be obtained by enlarging the constant $C$.

	\textbf{Step 1:}
	We take into account of measures. We want to reduce the condition of WNC on $A^n_d$ to $F$, which is essentially a linear functional.
	
	For this purpose, we will bound the measure of small $A^n_d$ by the measure of small $F$. 	
		In order to control $F/A^n_d(\bfh{d+1},\eta)$, we take $\bfh {d+1}$ which is $\eta$ good, that means for every $i$ in $\{1,\cdots,d+1 \}$, the group element $h_i$ is $(n,\epsilontwo ,\eta,\zeta^m_g)$ good (Definition \ref{defi:good element}). By Lemma \ref{lem:cocbou} and \eqref{equ:xnh eta}, for $1\leq i\leq d+1,1\leq j\leq \rank$
		\[|X^i_j|\leq \delta^{-1}. \]
	Since $F/A^n_d$ is a polynomial on $X^i_j$, for $\bfh {d+1}$ which is $\eta$ good, we have
	\begin{equation}\label{equ:fand}
	F/A^n_d=D_d\Pi_{\rho\in\fini,\rho\neq e}D_dA^n_d\rho\leq \delta^{-C_0}.
	\end{equation} 
	Hence by \eqref{equ:fand} and Lemma \ref{lem:cocbou expoential}, we have
	\begin{equation}\label{equ:mgood}
	\begin{split}
		M &:=(\mucon{n})^{\otimes(d+1)}\otimes\mucon{n}\{(\bfh {d+1},\ell)||A^n_d(\bfh{d+1},\ell\eta)|\leq e^{-\epsilon n} \}\\
		&\leq (\mucon{n})^{\otimes(d+1)}\otimes\mucon{n}\{\bfh {d+1} \text{ is }\ell\eta\text{ good, }\ell\in G ||A^n_d(\bfh{d+1},\ell\eta)|\leq e^{-\epsilon n} \}+O_{\epsilontwo }(\delta^c)\\
		&\leq  (\mucon{n})^{\otimes(d+1)}\otimes\mucon{n}\{\bfh {d+1} \text{ is }\ell\eta\text{ good, }\ell\in G ||F(\bfh{d+1},\ell\eta)|\leq e^{-\epsilon n}\delta^{-C_0} \}+O_{\epsilontwo }(\delta^c)\\
		&\leq  (\mucon{n})^{\otimes(d+1)}\otimes\mucon{n}\{(\bfh {d+1},\ell) ||F(\bfh{d+1},\ell\eta)|\leq e^{-\epsilon n}\delta^{-C_0} \}+O_{\epsilontwo }(\delta^c).
	\end{split}
	\end{equation}
	
	\textbf{Step 2:}
	Lemma \ref{lem:linear functional} tells us that
		\begin{equation*}
		F(\bfh{d+1},\eta)=F_1(\otimes_j ((v^j)^2)^{\otimes q_j})/\Pi \|v_j\|^{2q_j},
		\end{equation*}
		where $F_1$ is a linear functional on $V_0=\bigotimes_{j}(Sym^2V_j)^{\otimes q_j}$. To be more precise, $F_1$ will be restricted to a linear form on $W$, the unique irreducible representation of $V_0$ with maximal weight. (This is specific for real split Lie groups)
	
	\textbf{It remains to show that for most $\bfh {d+1}$ in $G^{(d+1)}$, the norm of $F_1$ is relatively large.} It is sufficient to find one $\eta$ such that $|F(\bfh{d+1},\eta)|$ is large. We will prove that $|D_dA_d\rho|$ is large for each $\rho$ in $\Gamma$, which implies that $|F(\bfh{d+1},\eta)|$ is large.
	
	Using the $(d+1)$-th column expansion of the matrix $\begin{pmatrix}
	y^1 & \cdots & y^{d+1}\\
	1 & \cdots & 1
	\end{pmatrix}$, we have
	\begin{equation}\label{equ:Ad y}
	\begin{split}
	 A_d(y^1,\cdots, y^{d+1})&=-A_{d-1}(r_dy^1,\cdots, r_dy^d)y^{d+1}_d+\text{ other terms}\\
	 &=\sum_{1\leq j\leq d}(-1)^{j+d+1}A_{d-1}(r_jy^1,\cdots, r_jy^d)y^{d+1}_j+\det(y^1,\cdots, y^d),
	 \end{split}
	 \end{equation}
	 where $r_j:\R^d\rightarrow \R^{d-1}$ is the map forgetting the $j$-th coordinate. Replacing $y^i$ by $E_dx^i$, due to $r_dE_dx^i=E_{d-1}x^i$, we obtain
	\begin{equation}\label{equ:Ad Ad-1}
	 A_d(E_dx^1,\cdots, E_dx^{d+1})=-A_{d-1}(E_{d-1}x^1,\cdots, E_{d-1}x^d)(E_dx^{d+1})_d+\text{ other terms}. 
	 \end{equation}
	 Using SNC on dimension $d-1$, we are able to give a lower bound of $A_{d-1}(E_{d-1}X^1,\cdots, E_{d-1}X^d)$ with a loss of exponentially small probability of $\bfh{d+1}$. But the problem is in other similar terms. Due to $y^{d+1}_j=\Pi_{1\leq i\leq \rank}(x^{d+1}_i)^{-\alpha_j(H_i)}$ and the structure of the root system, the degree of $x^{d+1}_d$ in $y^{d+1}_j=(E_dx^{d+1})_j$ is 
	\begin{equation}\label{equ:roots}
		-\alpha_d(H_d)=-2 \text{ and }-\alpha_j(H_d)\geq 0 \text{ for }j<d.
	\end{equation}
	Hence, we will find an $\eta$ with $X^{d+1}_d\leq \beta$, which makes the first term in \eqref{equ:Ad y} greater than $\delta^{C_0}\beta^{-2}$, and the other terms are less than $\delta^{-C_0}$.
	
	Now, here is the precise proof.	Take $h_{d+1}$ good, that means $h_{d+1}$ is $(n,\epsilontwo ,\zeta^m_g)$ good (Definition \ref{defi:good element'}). We take 
	\begin{equation}\label{equ:eta d}
	\eta=\ell_{h_{d+1}}^{-1}\eta_{\alpha_d}
	\end{equation} 
	as in Lemma \ref{lem:cocbou'}.
	By Lemma \ref{lem:cocbou'}
	\begin{equation}\label{equ:norlar1}
		X^{d+1}_j \in[\delta,\delta^{-1}] \text{ for } j\neq d\text{ and }
		X^{d+1}_d\leq \beta\delta^{-1}.
	\end{equation}
	Let $\fini_{d-1}=(\bb Z/2\bb Z)^{(d-1)d}$, seen as a subgroup of $\fini$, which acts on $\R^{(d-1)d}$. Then we demand that $\bfh{d}$ satisfies
	\begin{equation}\label{equ:norlar2}
		|A_{d-1}^n\rho(\bfh{d},\eta)|\geq \delta\text{ for all }\rho\in \fini_{d-1}\text{ and }\bfh{d} \text{ is }\eta\text{ good}.
	\end{equation}
	Recall that $\bfh{d}$ is $\eta$ good means that $h_i$ is $(n,\epsilontwo ,\eta,\zeta^m_g)$ good for $1\leq i\leq d$ and by Lemma \ref{lem:cocbou} and \eqref{equ:xnh eta}, this implies
	\begin{equation}\label{equ:norlar3}
		X^i_j(\eta)\in [\delta,\delta^{-1}], \text{ for }1\leq i\leq d,1\leq j\leq \rank.
	\end{equation}
	Recall that $W$ is the unique irreducible subrepresentation of $V_0$ with the highest weight.
	\begin{lem}\label{lem:norlar}
	If $h_{d+1}$ is good ($(n,\epsilontwo ,\zeta^m_g)$ good), $\eta$ is taken as in \eqref{equ:eta d} and the assumption \eqref{equ:norlar2} is satisfied for $\bf{h}_d$, then the operator norm satisfies \[\|F_1|_W\|\geq \delta^{C_0}.\]
	\end{lem}
	 \begin{proof}[Proof ot Lemma \ref{lem:norlar}]
	 	As we have already explained, it is sufficient to prove that, for $\rho$ in $\fini$, we have
	 		\[|D_dA^n_d\rho(\bfh{d},\eta)|\geq \delta^{C_0}. \]
	 	The proof is similar for an arbitrary $\rho$ in $\fini$, we will only prove the case $\rho=e$.	 
	 	
	 	By \eqref{equ:Ad y} and \eqref{equ:Ad Ad-1}
	 	\begin{equation}\label{equ:Ad y'}
		 	\begin{split}	 	
	 		 &D_d(\underline{x})A_d(E_dx^1,\cdots, E_dx^{d+1})=-A_{d-1}(E_{d-1}x^1,\cdots, E_{d-1}x^d)D_d(
	 		 \underline{x})(E_dx^{d+1})_d\\
	 		 &+\sum_{1\leq j<d}(-1)^{j+d+1}A_{d-1}(r_jE_dx^1,\cdots, r_jE_dx^d)D_d(\underline{x})(E_dx^{d+1})_j+D_d(\underline{x})\det(E_dx^1,\cdots, E_dx^d)
	 		\end{split}
	 	\end{equation}
	 	where $r_j:\R^d\rightarrow \R^{d-1}$ is the map forgetting the $j$-th coordinate. Since $x^{d+1}_d$ only appears in $E_dx^{d+1}$,	by \eqref{equ:roots}, we know that the degree of $x^{d+1}_d$ in $D_d$ equals $\alpha_d(H_d)=2$, which implies that $$D_d(X^1,\cdots,X^{d+1})\leq \delta^{-C_0}\beta^2.$$ 
	 	Hence by \eqref{equ:norlar1}-\eqref{equ:norlar3} and the property \eqref{equ:roots} that the degree of $X^{d+1}_d$ in $(E_dX^{d+1})_d$ is $-2$, the degree in $(E_dX^{d+1})_j$ is non negative for $j<d$, we have
	 	\begin{equation}\label{equ:DdEd}
	 	\begin{split}
	 	&D_d(E_dX^{d+1})_d\geq \delta^{C_0},\ |A_{d-1}(E_{d-1}X^1,\cdots, E_{d-1}X^d)|\geq \delta^{C_0},\\
	 	&D_d(E_dX^{d+1})_j\leq \delta^{-C_0}\beta^2,\ |A_{d-1}(r_jE_dX^1,\cdots, r_jE_dX^d)|\leq \delta^{-C_0}\text{ for }1\leq j<d\\
	 	& \text{ and }D_d\det(E_dX^1,\cdots, E_dX^d)\leq \delta^{-C_0}\beta^2.
	 	\end{split}
	 	\end{equation}
	 	By \eqref{equ:Ad y'} and \eqref{equ:DdEd}, we have
	 	\begin{equation*}
	 	|D_dA^n_d|\geq \delta^{C_0} -\delta^{-C_0}\beta^2\geq \delta^{C_0}.
	 	\end{equation*}
	 	The proof is complete by a good choice of $\epsilon$.
	 \end{proof}
	\textbf{Step 3.} We return to the proof of Lemma \ref{lem:SNC}. We write $\ell v$ for the vector $\otimes_j (\ell(v^j)^2)^{\otimes q_j}$ in $V_0$. Then $\R \ell v$ is exactly the image of $\ell\eta$ in $\bp W$. Using the Fubini theorem and \eqref{equ:mgood}, we have
	\begin{equation*}
	\begin{split}
	M&\leq \int \dd\mucon{n}(h_{d+1})\int\dd(\mucon{n})^{\otimes(d-1)}(\bfh d)\mucon{n}\left\{\ell\Big|\frac{|F_1(\ell v)|}{\|F_1|_W\|\|\ell v\|}\leq e^{-\epsilon n}\delta^{-C_0}\|F_1|_W\|^{-1}\right \}\\
	&+O_{\epsilontwo }(\delta^c).
	\end{split}
	\end{equation*}
	Using SNC on dimension $d-1$, for all $\rho\in \fini_{d-1}$, we have $(\mucon{n})^{\otimes(d-1)}\{(\bfh d)|| A_{d-1}^n\rho(\bfh{d},\eta)|\leq \delta \}=O_{\epsilontwo }(\delta^c)$. (This is a stronger form of SNC on dimension $d-1$. Due to $\fini_{d-1}\in \rm{O}(d-1)^{\times d}$, it follows from Remark \ref{rem:fini SNC} that SNC implies this stronger form.) 
	By Lemma \ref{lem:cocbou expoential}, the set where $h_{d+1}$ is not $(n,\epsilontwo ,\zeta^m_g)$ good and $\bfh{d}$ is not $\eta$ good have exponentially small measure. Hence
	\begin{equation}\label{equ:mg1}
		\begin{split}
		M&\leq \int_{good} \dd\mucon{n}(h_{d+1})\int_{\bfh{d} \text{ satisfying }\eqref{equ:norlar2}}\dd(\mucon{n})^{\otimes(d-1)}(\bfh d)\\&\mucon{n}\left\{\ell\Big|\frac{|F_1(\ell v)|}{\|F_1|_W\|\|\ell v\|}\leq e^{-\epsilon n}\delta^{-C_0}\|F_1|_W\|^{-1}\right \}
		+O_{\epsilontwo }(\delta^c).
		\end{split}
	\end{equation}
	Due to Lemma \ref{lem:norlar}, when $\epsilontwo $ is small enough with respect to $\epsilon$, we have ($\delta=e^{-\epsilontwo  n}$ and $\|F_1|_W\|\geq \delta^{C_0}$)
	\[e^{-\epsilon n}\delta^{-C_0}\|F_1|_W\|^{-1}\leq e^{-\epsilon n}\delta^{-C_0}\leq e^{-\epsilon n/2}. \]
	Using Lemma \ref{lem:lindev} with $V=W$, due to $\ell v$ in $W$ we conclude that, under the condition of Lemma \ref{lem:norlar},
	\begin{equation}\label{equ:mlf1}
		\mucon{n}\left\{\ell\Big|\frac{|F_1(\ell v)|}{\|F_1|_W\|\|\ell v\|}\leq e^{-\epsilon n}\delta^{-C_0}\|F_1|_W\|^{-1} \right\}\leq_\epsilon e^{-c\epsilon n}.
	\end{equation}
	By \eqref{equ:mg1} and \eqref{equ:mlf1}, the proof is complete.
\end{proof}
\subsection{Combinatoric tool}
\label{sec:com}
\begin{prop}\label{prop:sum-product}
	Fix $\kappa_1>0$. Let $C_0>0$. Then there exist $\epsilonthr $, $k\in\bb N$ depending only on $\kappa_1$ such that the following holds for $\fren$ large enough depending on $C_0$.  Let $\lambda_1,\dots \lambda_k$ be Borel measures on $([-\fren^{\epsilonfour },-\fren^{-\epsilonfour }]\cup[\fren^{-\epsilonfour },\fren^{\epsilonfour }])^\rank\subset \R^\rank$ where $\epsilonfour =\min\{\epsilonthr ,\epsilonthr \kappa_1\}/10k$, with total mass less than $1$. Assume that for all $\rho\in[\fren^{-2},\fren^{-\epsilonthr }]$ and $j=1,\dots,k$
	\begin{equation}\label{equ:non-con}
	\quad\sup_{a\in\bb R,v\in\bb S^{\rank-1}}(\pi_v)_*\lambda_j(B_\R(a,\rho))=\sup_{a,v}\lambda_j\{x| \,\l v, x\r\in B_\R(a,\rho) \}\leq C_0\rho^{\kappa_1}.
	\end{equation}	
	
	Then for all $\varsigma\in\R^\rank, \|\varsigma\|\in[\fren^{3/4},\fren^{5/4} ]$ we have
	\begin{equation*}
	\left|\int\exp(i\l\varsigma, x_1\cdots x_k\r)\dd\lambda_1(x_1)\cdots\dd\lambda_k(x_k) \right|\leq \fren^{-\epsilonthr }.
	\end{equation*}
\end{prop}
This is proved in \cite{li-sumproduct_2018}, based on a discretized sum-product estimate by He-de Saxc\'e \cite{he_sum-product_2018}. When $m=1$, this is due to Bourgain in \cite{bourgain2010discretized}. 
\subsection{Application to measures induced by the random walk}
\label{sec:application}
From Proposition \ref{prop:PNC}, we fix $\epsilontwo <\frac{1}{10}\min_{\alpha\in\Pi}\{\alpha\sigma_\mu\}$ and we can find $c_1$ such that PNC holds. Let $(\epsilontwo /2,c')$ be the constants in Lemma \ref{lem:cocbou expoential}. Take 
$$\kappa_0=\frac{1}{10}\min\{c_1,c' \}.$$ 
Using Proposition \ref{prop:sum-product} with $\kappa_1=\kappa_0$, we get two constants $\epsilonthr , \epsilonfour $.

For $g,h$ in $G$ and $\eta$ in $\PP$, recall that $\sigmah(h,\eta)=(e^{-\alpha(\sigma(gh,\eta)-\kappa(g)-n\sigma_\mu)})_{\alpha\in\Pi}\in \R^m$.
Let $\lambda_{{g},\eta}$ be a pushforward measure on $\R^\rank$ of $\mucon{n}$ restricted on a subset $G_{n,{g},\eta}$ of $G$, which is defined by $$\lambda_{{g},\eta}(E)=\mucon{n}\{h\in G_{n,{g},\eta}|\sigmah(h,\eta)\in E \}$$ 
for any Borel subset $E$ of $\R^\rank$, where 
\begin{equation}\label{equ:h0gx}
G_{n,{g},\eta}=\{h\in G| h\text{ is }(n,\epsilon,\eta,\zeta^m_g) \text{ good} \}
\end{equation}
and where the constant $\epss >0$ will be determined later.

PNC is only at one scale, we need to verify all the scales needed in the sum-product estimate. The idea is to separate the random variable and try to use PNC in some other scale, where we need the cocycle property to change scale.
\begin{prop}[Change scale]\label{prop:appmea}
	With $\epss $ small enough depending on $\epsilonfour \epsilontwo $, there exists $C_0$ independent of $n$ such that the measure $\lambda_{g,\eta}$ satisfies the conditions in Proposition \ref{prop:sum-product} with constant $\fren=e^{\epsilontwo  n}$ for all $n\in\bb N$.
\end{prop}
\begin{proof}
	We abbreviate $\lambda_{g,\eta}$ to $\lambda$. 	By taking $\epss $ small depending on $\epsilonfour \epsilontwo $, Lemma \ref{lem:cocbou} implies that the support of $\lambda$ is contained in the cube $[\fren^{-\epsilonfour },\fren^{\epsilonfour }]^\rank$.
	
	Then we verify \eqref{equ:non-con}. Let $\rho\in[\fren^{-2},\fren^{-\epsilonthr } ]$. 
		Let $n_1=[\frac{|\log\rho|}{2\epsilontwo }] $ and $n_2=n-n_1$. Then $n_1$ lies in $[\epsilonthr n/2,n]$. We separate $h=h_1h_2$ such that $h_1,h_2$ have distributions $\mucon{n_1}, \mucon{n_2}$, respectively. We have
		\begin{equation}\label{equ:cocycle sigmah}
		\sigmah(h,\eta)=\sigmahh^{n_1}(h_1,h_2\eta)\sigmahh^{n_2}_0(h_2,\eta).
		\end{equation}
		We will use this cocycle property and the support of $Y$ to change the scale. 
		
	For \eqref{equ:non-con}, due to the fact that the support of $\lambda$ is contained in $[\fren^{-\epsilonfour },\fren^{\epsilonfour }]^\rank$, for $w\in \bb S^{\rank-1}$
	\begin{equation}\label{equ:piwn}
	(\pi_w)_*\lambda(B(a,\rho))\leq\sup_{h_2,v}\mucon{n_1}\{h_1|\l v,\sigmahh^{n_1}(h_1,h_2\eta)\r\in R^{-1}B(a,\rho),\sigmah(h_1h_2,\eta)\in[\fren^{-\epsilonfour },\fren^{\epsilonfour }]^\rank  \},
	\end{equation}
	where $R=\|w\sigmahh^{n_2}_0(h_2,\eta) \|$ depends on $h_2$.
	\begin{itemize}
		\item If $R\geq \rho^{1/2}$, then $\rho R^{-1}\leq \rho^{1/2}\leq e^{-\epsilontwo n_1}$. It follows by PNC at scale $n_1$ that
		\begin{equation}\label{equ:piwn1}
		\mucon{n_1}\{h_1|\l v,\sigmahh^{n_1}(h_1,h_2\eta)\r\in B(a/R,e^{-\epsilontwo n_1}) \}\ll_{\epsilontwo }e^{-c_1\epsilontwo n_1}\leq\rho^{\kappa_0}.
		\end{equation}
		
		\item If $R\leq \rho^{1/2}$. There exists one coordinate $\alpha$ such that $|\sigmahh^{n_2}_0(h_2,\eta)_\alpha|\leq \rho^{1/2}$, which implies that $\sigmahh^{n_1}(h_1,h_2\eta)_\alpha=\sigmah(h,\eta)_\alpha/\sigmahh^{n_2}_0(h_2,\eta)_\alpha\geq \fren^{-\epsilonfour }\rho^{-1/2}$. Due to $n_1\geq \epsilonthr  n/2$ and $\epsilonthr \geq 10\epsilonfour $, we have $$ \fren^{-\epsilonfour }\rho^{-1/2}\geq e^{\epsilontwo(n_1-\epsilonfour n)}\geq e^{\epsilontwo(n_1/2+\epsilonthr n/4-\epsilonfour n)} \geq e^{\epsilontwo  n_1/2}.$$
		For such $h_2$, we have 
		\begin{equation}\label{equ:piwn2}
		\begin{split}
		&\mucon{n_1}\{h_1|\sigmahh^{n_1}(h_1h_2,\eta)\in[\fren^{-\epsilonfour },\fren^{\epsilonfour }]^\rank  \}
		\leq\sum_{\alpha\in\Pi} \mucon{n_1}\{h_1|\sigmahh^{n_1}(h_1,h_2\eta)_\alpha\geq e^{\epsilontwo n_1/2} \}.
		\end{split}
		\end{equation}
		It follows from Lemma \ref{lem:cocbou} that
		\begin{equation}\label{equ:piwn3}
			\begin{split}
			&\mucon{n_1}\{h_1|\sigmahh^{n_1}(h_1,h_2\eta)_\alpha\geq e^{\epsilontwo n_1/2} \}\\&\leq \mucon{n_1}\{h_1|\|\sigma(gh_1,h_2\eta)-\kappa(g)-n_1\sigma_\mu \|\geq \epsilontwo n_1/2\}
			\ll_{\epsilontwo }e^{-c'\epsilontwo n_1}\leq \rho^{\kappa_0}.
			\end{split}
		\end{equation}
	\end{itemize}
	By \eqref{equ:piwn}-\eqref{equ:piwn3}, for $\rho\in[\fren^{-2},\fren^{-\epsilonthr } ]$ we have
	\[ 	(\pi_w)_*\lambda(B(a,\rho))\ll_{\epsilontwo } \rho^{\kappa_0}. \]
	The proof is complete.
\end{proof}
\section{Proof of the main theorems}
\label{sec:proof}
In this section, we will use the results of Section \ref{sec:lie groups} and Section \ref{sec:noncon} to give the proofs of the main theorems. In Section \ref{sec:sumfou}, we will prove Theorem \ref{thm:foudec}, the simply connected case. For non simply connected case, please see Theorem \ref{thm:foudecsemi} in Appendix \ref{sec:semisimple}. Then in Section \ref{sec:fougap}-\ref{sec:expdec}, we will work on semisimple case and we prove all the other theorems in the introduction from Theorem \ref{thm:foudecsemi}.

We will add many assumptions on the elements of $G$ and $\P$. The assumptions seem complicated. In fact, the exact form is not really important. They are assumed to make the result work outside a set of exponentially small measure. These assumptions usually ask that the elements are away from certain closed subvarieties of $G$ or that they are not too far from the expected average value.
\subsection{$(C,r)$ good function}
\label{sec:C r good}
For a $C^1$ function $\varphi$ on the flag variety $\PP$, we first lift it to $\P_0=G/A_eN$. Let $\partial_\alpha \varphi=\partial_{Y_{\alpha}}\varphi$ be the directional derivative on $\P_0$. By Lemma \ref{lem:kmk} the action of the group $M$ only changes the sign of the directional derivative $\partial_\alpha\varphi$, hence $|\partial_\alpha\varphi|$ is actually a function on $\P$. Although $\partial_\alpha\varphi$ is not well-defined on $\P$, we can fix a local trivialization of the line bundle with fibre $\R Y_\alpha$ and define the directional derivative. This local definition will be used in G3. 

Recall that for $\eta,\eta'$ in $\P$ and a simple root $\alpha$, we have defined $d_\alpha(\eta,\eta')=d(V_{\alpha,\eta},V_{\alpha,\eta'})$ and $d(\eta,\eta')=\sup_{\alpha\in\Pi}d_\alpha(\eta,\eta')$. 
\begin{defi}\label{defi:C r good}
	Let $r$ be a continuous function on $\P$ and $C>1$. 
	Let $\open$ be the open set in $\P$, which is the $1/C$-neighbourhood of the support of $r$. Let $\varphi$ be a $C^2$ function on $\PP$. For a simple root $\alpha$, let $v_\alpha=\sup_{\eta\in \supp r}|\partial_\alpha\varphi(\eta)|$.
	We say that $\varphi$ is $(C, r)$ good if:
	\begin{itemize}
		\item[(G1)] For $\eta,\eta'$ in $\open$ such that $d(\eta,\eta')\leq 1/C$,
		\begin{equation}\label{equ:G1}
		|\varphi(\eta)-\varphi(\eta')|\leq C \sum_{\alpha\in\Pi}d_\alpha(\eta,\eta')v_\alpha,
		\end{equation}
		\item[(G2)] For every simple root $\alpha$ and for every $\eta$ in the support of $r$, we have 
		\begin{equation}\label{equ:G2}
		|\partial_\alpha \varphi(\eta)|\geq\frac{1}{C}v_\alpha,
		\end{equation}
		\item[(G3)] For $\eta, \eta'$ in $\open$ with $d(\eta,\eta')\leq 1/C$,
		\begin{equation}\label{equ:G3}
		|\partial_\alpha \varphi(\eta)-\partial_\alpha\varphi(\eta')|\leq Cd(\eta,\eta')v_\alpha.
		\end{equation}
		\item[(G4)] 
		\begin{equation}\label{equ:G4}
		\sup_{\alpha\in\Pi}v_\alpha\in[1/C,C].
		\end{equation}
	\end{itemize}
\end{defi}
\begin{rem}\label{rem:distance}
	The distance $d_\alpha$ depends on the representation $V_\alpha$. But for two different representation $(\rho,V),(\rho',V')$ such that $\Theta(\rho)=\Theta(\rho')=\{\alpha \}$, by Lemma \ref{lem:kvv}, when $C$ is large enough, two distances $d_V,d_{V'}$ are equivalent.

	In the above definition, the G3 assumption \eqref{equ:G3} is equivalent to the inequality on $\P_0$, that is
	\begin{equation}\label{equ:G3'}
		|\partial_\alpha \varphi(\k)-\partial_\alpha\varphi(\k')|\leq Cd_0(\k,\k')v_\alpha,
	\end{equation}
	for  $\k,\k'$ in $\pi^{-1}(J)$ with $d(\k,\k')\leq 1/C$.
\end{rem}
The G1 assumption is trivial for $\sltwo$, so it is new in higher dimension. 
The G2 and the G3 assumptions are natural generalizations of the case $\rank=1$, $\sltwo$. The G4 assumption is used to normalize the function. 

The role of $\open$ is to simplify the verification of $(C,r)$ goodness. With this definition, we only need to verify assumptions on a neighbourhood of the support of $r$.
\subsection{From sum-product estimates to Fourier decay}
\label{sec:sumfou}
In this subsection we will prove Theorem \ref{thm:foudec}, an estimate of Fourier decay, by using the results established in Section \ref{sec:lie groups} and Section \ref{sec:noncon}. 

Recall that in Section \ref{sec:application} we have fixed $(\epsilontwo ,c_1)$ for Proposition \ref{prop:PNC}, the constant $(\epsilontwo /2,c')$ in Lemma \ref{lem:cocbou expoential}, which means that $(\epsilon,c)$ in Lemma \ref{lem:cocbou expoential} equals $(\epsilontwo /2,c')$, and 
$$\kappa_0=\frac{1}{10}\min\{c_1,c' \}.$$
Take $k, \epsilonthr , \epsilonfour $ from Proposition \ref{prop:sum-product} with this $\kappa_0$. 
Let $\epss$ be a positive number to be determined later (the only constant which is not fixed yet).
The constant $\epsilonzer $ in the hypothesis of Theorem \ref{thm:foudec} is defined as 
\begin{equation}\label{equ:eps0}
\epsilonzer =\frac{\epss}{\max_{\alpha\in\Pi}\{(2k+1)\alpha\sigma_\mu+\epsilontwo \}+\epss}
\end{equation} 
which will be fixed once $\epss$ is fixed. 

Here, we define and give relations of different constants. Let $v$ be the vector in $\R^\rank$ whose components are $
 v_\alpha=\sup_{\eta\in\supp r}|\partial_\alpha \varphi(\eta)|,\text{ for }\alpha\in\Pi$.
Recall that $\xi>1$. Then by the G4 assumption \eqref{equ:G4}, we have
\begin{equation}\label{equ:valp0}
	\sup_{\alpha\in\Pi}v_\alpha\in[\xi^{-\epsilonzer },\xi^{\epsilonzer }].
\end{equation}
Let $n$ be the minimal integer such that
\begin{equation}\label{equ:valp}
e^{\epsilontwo n}\geq \xi\max_{\alpha\in\Pi}\{v_\alpha e^{-(2k+1)\alpha\sigma_\mu n}\}.
\end{equation}
The existence is guaranteed by the positivity of Lyapunov constant, that is $\alpha\sigma_\mu>0$ for $\alpha\in\Pi$ (Lemma \ref{lem:poslya}).
Let the regularity scale $\delta$ be given by
\[\delta=e^{-\epss n}<1/2,\]
where we take $\xi$ large enough depending on $\epsilon$ so that $n$ is large enough. Let the contraction scale $\expec$ be given by
\begin{equation*}
\expec_\alpha=e^{-\alpha\sigma_\mu n}, \expec=\max_{\alpha\in\Pi}\{\expec_\alpha \}. 
\end{equation*}
The point is that the contraction speed $\beta$ decides the magnitude of a term and $\delta$ is only an error term, much larger than $\beta$. We can take $\epsilon$ small so that 
\begin{equation}\label{equ:epsbeta}
\beta\leq \delta^{10}.
\end{equation} 
 
Let the frequency $\fren$ be defined by $\fren=e^{\epsilontwo n} $. By \eqref{equ:valp}, we have
\begin{equation}\label{equ:fren}
	\fren\geq \xi\max_{\alpha\in\Pi}\{v_\alpha\beta_\alpha^{2k+1}\}\geq C_ {\epsilontwo }\fren,
\end{equation}
where $C_{\epsilontwo }=e^{-\epsilontwo }\min_{\alpha\in\Pi}\{e^{-(2k+1)\alpha\sigma_\mu} \}$.
By \eqref{equ:valp0}, there exists $\alpha_o$ in $\Pi$ such that $v_{\alpha_o}\geq\xi^{-\epsilonzer }$. Then \eqref{equ:fren} and \eqref{equ:eps0} imply that
\[\xi\leq\fren v_{\alpha_o}^{-1}\expec_{\alpha_o}^{-2k-1}\leq \xi^{\epsilonzer }\fren\expec_{\alpha_0}^{-(2k+1)} \leq \xi^{\epsilonzer }e^{n\epss\frac{1-\epsilonzer }{\epsilonzer }}. \]
Hence the regularity scale satisfies
\begin{equation}\label{equ:regsca}
\xi^{\epsilonzer }\leq e^{\epsilon n}=\delta^{-1}.
\end{equation}
\textbf{Notation}: We introduce some notation which will be used throughout Section \ref{sec:sumfou}.
\begin{itemize}
	\item Let $\bf g=(g_0,\dots, g_k)$ be an element in $G^{ (k+1)}$.
	\item Let $\bf h=(h_1,\dots, h_k)$ be an element in $G^{ k}$.
	\item We write $\bfgh=g_0h_1\cdots h_kg_k\in G$  for the product of $\bf g,\bf h$.
	\item We write $T\bfgh=g_0h_1\cdots g_{k-1}h_k\in G$.
	\item For $l\in\bb N$, let $\mu_{l,n}$ be the product measure on $G^{ l}$ given by $\mu_{l,n}=(\mucon{n})^{\otimes l}$.
	\item Recall that for $g,h$ in $G$ and $\eta$ in $\PP$, we define $\sigmah_g(h,\eta)_\alpha=\exp(-\alpha(\sigma(gh,\eta)-\kappa(g)-n\sigma_\mu))$ and $\sigmah_g(h,\eta)=(\sigmah_g(h,\eta)_\alpha)_{\alpha\in\Pi}\in\R^\rank$.
	\item For $\k$ in $\P_0$, let $\tisig_g(h,\k)_\alpha=\alpha^\up(\sg(\ell_g^{-1},h\k))\sigmah_g(h,\eta)_\alpha$, where $\alpha^\up$ is the corresponding algebraic character of the simple root $\alpha$ and we make a choice of $\ell_g$ and $\eta=\pi(\k)$.
	\item For $g$ in $G$, $\k$ in $\P_0$ and $\eta=\pi(\k)$, let $\tilde\lambda_{g,\k}$ be the pushforward measure on $\R^\rank$ of $\mucon{n}$ restricted to a subset $G_{n,g,\eta}$ under the map $\tisig_{g}(\cdot, \k)$. In other words, for a Borel set $E$,
	$$\tilde\lambda_{g,\k}(E)=\mucon{n}\{h\in G_{n,g,\eta}|\tisig_{g}(h,\k)\in E \}.$$ 
	Recall that the set $G_{n,g,\eta}$ is defined by $G_{n,{g},\eta}=\{h\in G| h \text{ is }(n,\epsilon,\eta,\zeta^m_g) \text{ good}\}$.
	
	\item After fixing $\bf{g}$, we will also fix a choice of $k_{g_j}$, $\ell_{g_j}$ for $g_j$ and let $\k_{g_j}=k_{g_j}\k_o$, $m_j(h)=\sg(\ell_{g_{j-1}}^{-1},hk_{g_j})$ and  $\lambda_j=\tilde\lambda_{g_{j-1},\k_{g_j}}$, for $j=1,\dots, k$.
\end{itemize}
\begin{lem}\label{lem:tilde noncon}
The measure $\tilde\lambda_{g,\k}$ satisfies the same property \eqref{equ:non-con} as $\lambda_{g,\eta}$ with $C_0$ replaced by $2^mC_0$, where $\eta=\pi(\k)$.
\end{lem}
\begin{proof}
	Since the difference is only in the sign, we have
	\[(\pi_v)_*\tilde\lambda_{g,\k}(B_\R(a,\rho))\leq \sum_{f\in(\bb Z/2\bb Z)^\rank}(\pi_{fv})_*\lambda_{g,\eta}(B_\R(a,\rho)), \]
	where we identify $(\bb Z/2\bb Z)^\rank$ with $\{-1,1\}^\rank\subset \R^\rank$. The result follows from this inequality.	
\end{proof}
\textbf{First step:} 
For $\eta,\eta'$ in $\P$, let
\begin{equation}\label{equ:feta}
f(\eta,\eta')=	\int_G e^{i\xi(\varphi(g \eta)-\varphi(g \eta'))}r(g \eta)r(g \eta')\dd\mucon{(2k+1)n}(g).
\end{equation}
\begin{lem}
	We have
	\begin{equation}\label{equ:fxx'}
	\begin{split}
	\left|\int_\P e^{i\xi\varphi( \eta)}r( \eta)\dd\nu( \eta)\right|^2\leq\int_{\P^2} f( \eta, \eta')\dd\nu( \eta)\dd\nu( \eta').
	\end{split}
	\end{equation}
\end{lem}
\begin{proof}
	By the definition of $\mu$-stationary measure and the Cauchy-Schwarz inequality, 
	\begin{equation*}
	\begin{split}
	&\left|\int_\P e^{i\xi\varphi( \eta)}r( \eta)\dd\nu( \eta)\right|^2
	\\&=\left|\int_{\P\times G} e^{i\xi\varphi(g \eta)}r(g \eta)\dd\mucon{(2k+1)n}(g)\dd\nu( \eta)\right|^2\leq \int_G\left|\int_\P e^{i\xi\varphi(g \eta)}r(g \eta)\dd\nu( \eta)\right|^2\dd\mucon{(2k+1)n}(g)\\
	&=\int_{\P^2}\int_G e^{i\xi(\varphi(g \eta)-\varphi(g \eta'))}r(g \eta)r(g \eta')\dd\mucon{(2k+1)n}(g)\dd\nu( \eta)\dd\nu( \eta').
	\end{split}
	\end{equation*}
	The proof is complete.
\end{proof}
Recall that for $\eta,\, \eta'$ in $\PP$, we write $V_{\alpha,\eta},\,V_{\alpha,\eta'}$ for their images in $\bb PV_\alpha$ and $d_\alpha(\eta,\eta')=d(V_{\alpha,\eta},V_{\alpha,\eta'})$.
\begin{defi}[Good Position]
	Let $ \eta, \eta'$ be in $\PP$. We say that they are in good position if
	\begin{equation*}
	\forall \alpha\in\Pi,\ d_\alpha(\eta,\eta')\geq \delta.
	\end{equation*} 
\end{defi}
We fix $ \eta, \eta'$ in good position, which means that $ \eta, \eta'$ are far in all $\bb PV_\alpha$. We rewrite the formula.
\begin{lem} 
	We have
	\begin{equation}\label{equ:dxx'}
	\begin{split}
	\left|\int_\P e^{i\xi\varphi( \eta)}r( \eta)\dd\nu( \eta)\right|^2\leq \int_{\eta,\eta' \text{ good}}f( \eta, \eta')\dd\nu( \eta)\dd\nu( \eta')+O(\delta^c).
	\end{split}
	\end{equation}
\end{lem}
\begin{proof}
	By the regularity of stationary measure \eqref{equ:regularity stataionary measure}, we have
	\begin{equation}\label{equ:nualpha}
	\nu\{\eta'\in\P | d_\alpha(\eta,\eta')\leq\delta\}=\nu\{\eta'\in\P|d(V_{\alpha,\eta},V_{\alpha,\eta'})\leq\delta\}\leq C\delta^c.
	\end{equation}
	Therefore by \eqref{equ:nualpha} and Fubini's theorem,
	\begin{equation*}
	\nu\otimes\nu\{(\eta,\eta')\in \P^2|\,d_\alpha(\eta',\eta)<\delta \} =\int_{\eta\in\P}\nu\{\eta'\in\P | d_\alpha(\eta,\eta')\leq\delta\}\dd\nu(\eta)\ll \delta^c. 
	\end{equation*}
	Summing over simple roots $\alpha$, we obtain the result by $\|r\|_\infty\leq 1$.
\end{proof}
\textbf{Second step:}
The purpose of this part is to prove that the Schottky type property is almost preserved by iteration, on the complement of an exponentially small set.

We fix $g_j$ for $j=0,\dots,k-1$ which satisfies
\begin{equation}\label{equ:g2l}
\|\kappa( g_j)-n\sigma_\mu\|\leq \epsilon n/C_A.
\end{equation}
Recall that $C_A$ is a constant in Definition \ref{defi:good element}. We also demand that
\begin{equation}\label{equ:g2j+1}
\begin{split}
h_{j+1} \text{ is } (n,\epsilon,\eta^M_{g_{j+1}},\zeta^m_{g_j}) \text{ good.}
\end{split}
\end{equation}

Recall that the Cartan subspace $\frak a$ is equipped with the norm induced by the Killing form, and with this norm $\frak a$ is isomorphic to the euclidean space $\R^{r}$.
\begin{lem}\label{lem:mainapprox}
	Suppose that $\bf{g}, \bf{h}$ satisfy the above conditions \eqref{equ:g2l} and \eqref{equ:g2j+1}. Then the action of $\tbfgh$ on $b^M_{V_\alpha,g_k}(\delta)$ is $\beta_\alpha^{2k}\delta^{-C_0}$ Lipschitz and
	\begin{equation}
	 \label{equ:exp leq}
	e^{-\alpha\sigma(g_0h_1,x^M_{g_1})}\cdots e^{-\alpha\sigma(g_{k-1}h_k,x^M_{g_k})}\leq \beta_\alpha^{2k}\delta^{-C_0},
\end{equation} 
	for every $\alpha$ in $\Pi$. For $t\in b^M_{g_k}(\delta)$, let $t_j=g_jh_{j+1}\cdots h_kt$ for $j=0,\dots ,k$, where we let $t_k=t$. Then
	\begin{align}\label{equ:xlb}
	&t_j\in b^M_{g_j}(\beta\delta^{-2})\subset b^M_{g_j}(\delta),
	\\ \label{equ:sigma-sigma}
	&\|\sigma(g_jh_{j+1},t_{j+1})-\sigma(g_jh_{j+1},\eta^M_{g_{j+1}})\|\leq\expec\delta^{-C_0}.
	\end{align}
\end{lem}
\begin{rem}
	In \eqref{equ:xlb}, we need $\beta<\delta^3$, which is true due to \eqref{equ:epsbeta}.
	The contraction constant $\expec$ here is a little different from the gap $\gap(g_j)$, but $\gap(g_j)/\expec$ is in the interval $[\delta^{C_0},\delta^{-C_0}]$ by Lemma \ref{lem:cocbou}. Hence they are of the same size and we will not distinguish them.
	
		The intuition here is that by controlling $\kappa(g),\eta^M_g,\zeta^m_g$, all the other positions or lengths will also be controlled, which is similar to what happened in hyperbolic dynamics.
\end{rem}
\begin{proof}[Proof of Lemma \ref{lem:mainapprox}]
	By Lemma \ref{lem:cocbou}, we have \eqref{equ:exp leq} from \eqref{equ:g2j+1} for all $\alpha$ in $\Pi$ at the same time.
	
	We use induction to prove the inclusion. For $j=k$, it is due to the hypothesis of Lemma \ref{lem:mainapprox}.
	Suppose that the property holds for $j+1$. By definition, $t_j=g_jh_{j+1}t_{j+1}$. We abbreviate $g_j,h_{j+1}, t_{j+1},\eta^M_{g_{j+1}}$ to $g, h,\eta,\eta'$. The condition becomes
	$$d(\eta,\eta')\leq \delta, \|\kappa(g)-n\sigma_\mu\|\leq \epsilon n/\constant \text{ and }h\text{ is }(n,\epsilon,\eta',\zeta^m_g) \text{ good.}$$
	By Lemma \ref{lem:cocbou}, we have $\gap(h)\leq \beta\delta^{-1}$. By Lemma \ref{lem:gBmgmul}, due to $\eta\in B(\eta',\delta)\subset B^m_{h}(\delta)$, we have $h\eta\in b^M_{h}(\expec/\delta^2)\subset B^m_{g}(\delta)$. Therefore $gh\eta\in b^M_{g}(\expec/\delta^2)$, which is the inclusion condition.
	
	The Lipschitz property for simple root can be obtained by using Lemma \ref{lem:gBmg} $2k$ times, first the inclusion, then the Lipschitz property. 
	
	Then we prove \eqref{equ:sigma-sigma} and we keep the notation $g,h,\eta,\eta'$. We have
	\begin{align*}
	\|\sigma(gh,\eta)-\sigma(gh,\eta')\|\ll \|\sigma(g,h\eta)-\sigma(g,h\eta') \|+\|\sigma(h,\eta)-\sigma(h,\eta') \|.
	\end{align*}
	By the same argument, due to Lemma \ref{lem:gBmgmul} and $\eta,\eta'\in B(\eta',\beta/\delta^2)\subset B^m_{h}(\delta)$, we have $h\eta,h\eta'\in b^M_h(\beta/\delta^2)\subset B^m_g(\delta)$. Therefore by the  Lipschitz property of Lemma \ref{lem:gBmgmul}
	$$	\|\sigma(gh,\eta)-\sigma(gh,\eta')\|\ll (d(\eta,\eta')+d(h\eta,h\eta'))\delta^{-1}\leq (2\expec/\delta^2)\times\delta^{-1}= 2\expec/\delta^3.$$
	The proof is complete.
\end{proof}
\begin{lem}
	Suppose that $\bf{g}, \bf{h}$ satisfy the conditions \eqref{equ:g2l} and \eqref{equ:g2j+1}. Let $s$ be in $\{\k\in \P_0|d_0(\k,\k_{g_k})\leq\delta  \}$ (the distance $d_0$ is defined in Appendix \ref{sec:equi distance}). Let $s_j=g_jh_{j+1}\cdots h_k s$ for $j=0,\dots, k$, where we let $s_k=s$. We have
	\begin{equation}\label{equ:s0g0m0}
		\sg(s_0,k_{g_0})=\Pi_{1\leq j\leq k}\sg(\ell_{g_{j-1}}^{-1},h_jk_{g_j})=\Pi_{1\leq j\leq k}m_j(h_j).
	\end{equation}
\end{lem}
\begin{proof}
	We let $\eta=\pi(s)$, then $\eta$ is in $b^M_{g_k}(\delta)$.
	By \eqref{equ:xlb} with $j=1$ and \eqref{equ:g2j+1} with $j=0$, Lemma \ref{lem:ghkk'} implies
	\begin{equation*}
		\sg(s_0,k_{g_0})=\sg(k_{g_0},g_0h_1s_1)=\sg(\ell_{g_0}^{-1},h_1k_{g_1})\sg(s_1,k_{g_1}).
	\end{equation*}
	Iterating this formula, we obtain the result.
\end{proof}
\textbf{Third step:} Here we mimic the proof of \cite{bourgain2017fourier}, where they heavily use the properties of Schottky groups and symbolic dynamics. But in our case, the group is much more complicated from the point of view of dynamics. We use the large deviation principle to get a similar formula.

By a very careful control of $g_{l}$, with a loss of an exponentially small measure, we are able to rewrite the formula in a form to use the sum-product estimate. The key point is that by controlling the Cartan projection and the position of $\eta^M_{g_l}$ and $\zeta^m_{g_l}$ of each $g_{l}$, we are able to get a good control of their product $\bfgh$.

We should notice that the element $g_j$ will be fixed, and we will integrate first with respect to $h_j$. This gives the independence of the cocycle $\sigma(g_{j-1}h_j,\eta^M_{g_j})$, that is for different $j$ they are independent, which is an important point to apply sum-product estimates. 

We return to \eqref{equ:dxx'}. We call $\bf g$ ``good" with respect to $\eta,\eta'$ if
\begin{equation}\label{equ:g good}
\begin{split}
&\bf g \text{ satisfies \eqref{equ:g2l}, }g_k \text{ satisfies the conditions in Lemma \ref{lem:chapoi}},\ \eta^M_{g_0}\in\supp r \\
 &\text{ and }\delta(\eta,\zeta^m_{g_k}),\delta(\eta',\zeta^m_{g_k}), \delta(V_{\alpha,\eta}\wedge V_{\alpha,\eta'},y^m_{\wedge^2\rho_\alpha g_k})\geq 4\delta.
\end{split}
\end{equation}
\begin{lem}\label{lem:gketagketa}
	If $\eta$ and $\eta'$ are in good position and $\bf g$ is ``good", then $g_k\eta,g_k\eta'$ are in $b^M_{g_k}(\delta)$, and for $\alpha\in\Pi$ the $d_\alpha$ distance between $g_k\eta$ and $g_k\eta'$ is almost $\beta_\alpha$, that is
	\begin{equation*}
		d_\alpha(g_k\eta,g_k\eta')\in\expec_\alpha[\delta^{C_0},\delta^{-C_0}].
	\end{equation*}
\end{lem}
\begin{proof}
	The inclusion is due to Lemma \ref{lem:gBmgmul}. Since $g$ is good \eqref{equ:g good}, by \eqref{equ:gxgx'l} we have the lower bound and by the Lipschitz property in Lemma \ref{lem:gBmg} we have the upper bound.
\end{proof}
For $ \eta, \eta'$ in $\P$, we can rewrite the formula of $f(\eta,\eta')$ as
\begin{equation}\label{equ:feta'}
	f(\eta,\eta')=\int e^{i\xi(\varphi(\bfgh \eta)-\varphi(\bfgh \eta'))}r(\bfgh \eta)r(\bfgh \eta')\dd\mu_{k,n}(\bf h)\dd\mu_{k+1,n}(\bf g). 
\end{equation}
We say that $\bf h$ is $\bf g$-regular if $\bf h$ satisfies \eqref{equ:g2j+1}. Let 
\begin{equation*}
f_{\bf g}(\eta,\eta')=\int_{\bf g-regular} e^{i\xi(\varphi(\bfgh \eta)-\varphi(\bfgh \eta'))}\dd\mu_{k,n}(\bf h).
\end{equation*}
\begin{lem}\label{lem:g good}
	There exists $c$ depending on $\epsilon$, such that
	For $\eta,\eta'$ in $\PP$ 
	\begin{equation}\label{equ:fxx'd}
	\begin{split}
	|f(\eta,\eta')|
	\leq \int_{\bf g ``good"}|f_{\bf g}(\eta,\eta')|\dd\mu_{k+1,n}(\bf g)+O_\epsilon(\delta^c),
	\end{split}
	\end{equation}
	if $\epss $ is small enough with respect to $\gamma$, that is $\epss\leq \min_{\alpha\in\Pi}\{\alpha\sigma_\mu\gamma/(2+2\gamma) \}$.
\end{lem}
\begin{proof}
	Let
	\begin{equation*}
	\tilde f_{\bf g}(\eta,\eta')=\int_{\bf g-regular} e^{i\xi(\varphi(\bfgh \eta)-\varphi(\bfgh \eta'))}r(\bfgh \eta)r(\bfgh \eta')\dd\mu_{k,n}(\bf h).
	\end{equation*}
	We call $\bf g$ ``semi-good" if $\bf g$ satisfies \eqref{equ:g good} except the assumption of $\eta^M_{g_0}\in\supp r$ in \eqref{equ:g good}. By large deviation principle (Proposition \ref{prop:lardev1}, Proposition \ref{prop:large deviatio flag}, Lemma \ref{lem:xx'w}), we conclude that 
	\begin{equation}\label{equ:not semigood}
	\mu_{k+1,n}\{g \text{ not ``semi-good" } \}\leq O_\epsilon(\delta^c).
	\end{equation}
	Then by \eqref{equ:feta'}, Lemma \ref{lem:cocbou expoential} and \eqref{equ:not semigood}, 
	\begin{equation}\label{equ:fxx'g}
	\begin{split}
	|f(\eta,\eta')|\leq \int_{\bf g}|\tilde f_{\bf g}(\eta,\eta')|\dd\mu_{k+1,n}(\bf g)+O_\epsilon(\delta^c)
	\leq \int_{\bf g ``semi-good"}|\tilde f_{\bf g}(\eta,\eta')|\dd\mu_{k+1,n}(\bf g)+O_\epsilon(\delta^c).
	\end{split}
	\end{equation}
	
	By Lemma \ref{lem:gketagketa}, \eqref{equ:xlb} with $j=0$ and $c_\gamma(r)\leq \xi^{\epsilonzer }\leq \delta^{-1}$,
	\begin{align*}
	|r(\eta^M_{g_0})^2-r(\bfgh \eta)r(\bfgh \eta')|\leq 2\|r\|_\infty c_\gamma(r)(\expec\delta^{-2})^\gamma\leq 2\expec^{\gamma}\delta^{-1-2\gamma}\leq 2\delta,
	\end{align*}
	if $\epss $ is small enough with respect to $\gamma$. Hence
	\begin{equation}\label{equ:tilf}
	\begin{split}
	|\tilde f_{\bf g}(\eta,\eta')|&\leq\left|\int_{\bf g-regular} e^{i\xi(\varphi(\bfgh \eta)-\varphi(\bfgh \eta'))}r(\eta^M_{g_0})^2\dd\mu_{k,n}(\bf h)\right|+O(\delta^c)\\
	&\leq r(\eta^M_{g_0})^2|\fgeven(\eta,\eta')|+O(\delta^c).
	\end{split} 
	\end{equation}
	If $r(\eta^M_{g_0})\neq 0$, then that $\bf{ g}$ is ``semi-good" implies $\bf g$ is ``good".
	Combined with \eqref{equ:fxx'g} and \eqref{equ:tilf}, by $\|r\|_\infty\leq 1$, we have
	\begin{align*}
	|f(\eta,\eta')|&\leq \int_{\bf g ``semi-good"} \left(r(\eta^M_{g_0})^2|\fgeven(\eta,\eta')|+O(\delta^c)\right)\dd\mu_{k+1,n}(\bf g)+O_\epsilon(\delta^c)\\
	& \leq \int_{\bf g ``good"}|f_{\bf g}(\eta,\eta')|\dd\mu_{k+1,n}(\bf g)+O_\epsilon(\delta^c).
	\end{align*}
	The proof is complete.
\end{proof}
	Recall that $\beta$ is the magnitude which is really small, $\delta$ is only an error term and $\fren$ is the frequency for applying the sum-product estimate, which lies between $\delta^{-1}$ and $\beta^{-1}$.
\begin{prop}\label{prop:mainapprox} Let $I_\fren=[\fren^{3/4},\fren^{5/4}]$. The following formula is true for $\eta,\eta'$ in good position and $\bf g$ ``good",
	\begin{equation}\label{equ:fgxx'}
	\begin{split}
	|\fgeven(\eta,\eta')|\leq\sup_{\|\varsigma\|\in I_\fren } \left|\int e^{i\l\varsigma, x_1\cdots x_k\r}\dd\lambda_1(x_1)\cdots \lambda_k(x_k) \right|+O(\expec\delta^{-C_0}\fren),
	\end{split}
	\end{equation}
	when $\epss $ is small enough with respect to $\epsilontwo $.
\end{prop}
\begin{rem}
	This is the most complicated step, where the difficulty comes from higher rank. We need to use the technique of changing flags to find the direction of slowest contraction speed, where we can use Newton-Leibniz's formula. Since the action of the sign group $M$ is non trivial on the slowest directions, we also carefully treat the sign.
\end{rem}
\begin{proof}
	The element $\eta,\eta'$ and $\bf g$ are already fixed. Since $g_k$ satisfies the conditions in Lemma \ref{lem:chapoi}, we obtain two chains
	$(\eta=\eta_0,\eta_1,\dots ,\eta_\nupione)$ and $(\eta'=\eta_0',\eta_1',\dots,\eta_\nupitwo')$ as in Lemma \ref{lem:chapoi}. Then we write
	\begin{equation}\label{equ:ghetaeta'}
		\begin{split}
		&\varphi(\bfgh\eta)-\varphi(\bfgh\eta')=\sum_{0\leq j\leq \nupione-1}(\varphi(\bfgh\eta_{j})-\varphi(\bfgh\eta_{j+1}))\\
		&-\sum_{0\leq j\leq \nupitwo-1}(\varphi(\bfgh\eta'_{j})-\varphi(\bfgh\eta'_{j+1}))
		+\left( \varphi(\bfgh \eta_\nupione)-\varphi(\bfgh\eta_\nupitwo')\right) ,
		\end{split}
	\end{equation}
	The terms for different $j$ and for $\eta,\eta'$ are similar. We fix $j$ and we simplify $\alpha(\eta_j,\eta_{j+1})$ to $\alpha$. 
	
	\textbf{We compute the term $\varphi(\bfgh\eta_j)-\varphi(\bfgh\eta_{j+1})$.} In order to treat the sign, we will work on $\P_0=G/A_eN$. Recall that $\pi:\P_0\rightarrow \P$ is the projection and we use $\k=k\k_o$ to denote the element $kA_eN$ in $\P_0$.
	
	By Lemma \ref{lem:getajgetal} and \eqref{equ:g good}, we know that $g_k\eta_j,g_k\eta_{j+1}$ are in $b^M_{g_k}(\delta)$, then they satisfy the condition of Lemma \ref{lem:mainapprox}. Let $\k_0,\k_1$ be preimages of $g_k\eta_{j}$ and $g_k\eta_{j+1}$ in $\P_0$ such that $\sg(\k_0,\k_1)=e$. Notice that $\k_0,\k_1$ are in the same $\alpha$-circle.
	By Lemma \ref{lem:chapoi} \eqref{equ:djgj} and Lemma \ref{lem:gketagketa}
	\begin{equation*}
	d(g_k\eta_{j},g_k\eta_{j+1})=d_\alpha(g_k\eta,g_k\eta')+O(\beta e^{-\alpha\kappa(g_k)}\delta^{-C_0})\in \expec_\alpha[\delta^{C_0},\delta^{-C_0}].
	\end{equation*} 
	Due to $\sg(\k_0,\k_1)=e$, the arc-length distance also satisfies
	\begin{equation}\label{equ:tilxx'}
	d_A(\k_0,\k_1)=\arcsin d(g_k\eta_{j},g_k\eta_{j+1})\in  \expec_\alpha[\delta^{C_0},\delta^{-C_0}].
	\end{equation}
	
	Now, we lift $\varphi$ to $\P_0$, becoming a right $M$-invariant function. By abuse of notation, we also use $\varphi$ to denote the lifted function. 
	Let $\gamma$ be an arc connecting $\k_0,\k_1$ with unit speed in the $\alpha$-circle with length less than $\pi/2$. 
	Without loss of generality, we suppose that $\gamma$ is in the positive direction (If not, we add a minus sign on the right hand side of \eqref{equ:varphigh}. The sign only depends on $z_0,z_1$, which is independent of $\bf h$).
	By Newton-Leibniz's formula \eqref{equ:newlei}, we have
	\begin{equation}\label{equ:varphigh}
	\varphi(\tbfgh \k_0)-\varphi(\tbfgh \k_1)=\int_0^u \partial_\alpha\varphi(\tbfgh\gamma(t))e^{-\alpha\sigma(\tbfgh,\gamma(t))}\dd t,
	\end{equation}
	where $u=d_A(\k_0,\k_1)$. Fix a time $t$ in $[0,u]$, let $s_j=g_jh_{j+1}\cdots h_k\gamma(t)$. Then $\pi(\gamma(t))$ is in $b^M_{g_k}(\delta)$, because $g_k\eta_j$ and $g_k\eta_{j+1}$ are in $b^M_{g_k}(\delta)$ and by \eqref{equ:tilxx'}. By \eqref{equ:xlb}, the element $\pi(s_0)$, the image of $s_0=\tbfgh\gamma(t)$ in $\P$, is in $b^M_{g_0}(\beta\delta^{-C_0})$.
	
	Recall that we have made a choice of the Cartan decomposition of every $g_j$ for $0\leq j\leq k$. In particular, $k_{g_0}$ is given in the decomposition of $g_0=k_{g_0}a_{g_0}\ell_{g_0}\in KA^+K$. Let $m_0=\sg(s_0,k_{g_0})$ and $\underline{s_0}=s_0m_0$, then $\sg(\underline{s_0},k_{g_0})=e$. By Lemma \ref{lem:kmk}, 
	\begin{equation}\label{equ:phim0}
		\partial_\alpha\varphi_{s_0}=\partial_\alpha\varphi_{\underline{s_0}m_0}=\alpha^\up(m_0)\partial_\alpha\varphi_{\underline{s_0}}.
	\end{equation}
	By Lemma \ref{lem:p0 p} and $\pi (s_0),\pi (\k_{g_0})=\eta^M_{g_0}$ in $b^M_{g_0}(\beta\delta^{-C_0})$, we have 
	\begin{equation}\label{equ:d0 s0 g0}
	d_0(\underline{s_0},\k_{g_0})\leq d(\pi (s_0)),\pi (\k_{g_0}))<\beta\delta^{-C_0}.
	\end{equation} 
	Due to $\bf g$ good \eqref{equ:g good}, we have $\eta^M_{g_0}\in\supp r$. By the G2 assumption \eqref{equ:G2}, we have $|\partial_\alpha \varphi(\k_{g_0})|\geq \delta v_\alpha$. By \eqref{equ:d0 s0 g0}, the point $\pi (s_0)$ is in $\open$, the $\delta$ neighbourhood of $\supp r$. By the G3 assumption \eqref{equ:G3},
	$|\partial_\alpha\varphi(\underline{s_0})-\partial_\alpha\varphi(\k_{g_0})|\leq \delta^{-1} v_\alpha d_0(\underline{s_0},\k_{g_0}) $, which implies $$\partial_\alpha\varphi(\underline{s_0})/\partial_\alpha\varphi(\k_{g_0})\in[1-\expec\delta^{-C_0},1+\expec\delta^{-C_0}]. $$
		By Lemma \ref{lem:mainapprox} \eqref{equ:sigma-sigma}, we have
		\begin{equation}\label{equ:varphis0}
		(1-\expec\delta^{-C_0})e^{-O(\expec/\delta)} \leq \frac{\partial_\alpha\varphi(\underline{s_0})e^{-\alpha\sigma(g_0h_1,s_1)}\cdots e^{-\alpha\sigma(g_{k-1}h_k,s_k)}}{\partial_\alpha\varphi(\k_{g_0})e^{-\alpha\sigma(g_0h_1,x^M_{g_1})}\cdots e^{-\alpha\sigma(g_{k-1}h_k,x^M_{g_k})}}\leq(1+\expec\delta^{-C_0})e^{O(\expec/\delta)}. 
		\end{equation}
	By \eqref{equ:exp leq}, we have
	\begin{equation*}
		B_\alpha:=e^{-\alpha\sigma(g_0h_1,x^M_{g_1})}\cdots e^{-\alpha\sigma(g_{k-1}h_k,x^M_{g_k})}\leq \beta_\alpha^{2k}\delta^{-C_0}.
	\end{equation*}
	Together with \eqref{equ:tilxx'}-\eqref{equ:varphis0}
	\begin{equation}\label{equ:varphigh'}
	|\varphi(\bfgh \eta_{j})-\varphi(\bfgh \eta_{j+1})-d_A(\k_0,\k_1)\alpha^\up(m_0)\partial_\alpha\varphi(\k_{g_0})B_\alpha|\leq \beta\expec_\alpha^{2k+1}\delta^{-C_0} v_\alpha. 
	\end{equation}
	
	\textbf{We deal with the error term which comes from the process of changing flags}. 
	By Lemma \ref{lem:chapoi} \eqref{equ:geta1 geta2} and \eqref{equ:g2l}, we have $$d_\alpha(g_k\eta_\nupione,g_k\eta_\nupitwo')\leq \max_{\alpha'\in\Pi}\{e^{-\alpha'\kappa(g_k)} \}e^{-\alpha\kappa(g_k)} \delta^{-2}\leq \beta_\alpha\beta\delta^{-C_0}. $$
	 Then the Lipschitz property in Lemma \ref{lem:mainapprox}  implies that
	\[d_\alpha(\bfgh\eta_\nupione,\bfgh\eta_\nupitwo')\leq \beta_\alpha^{2k}\delta^{-C_0}d_\alpha(g_k\eta_\nupione,g_k\eta_\nupitwo')\leq \beta_\alpha^{2k+1}\beta\delta^{-C_0}, \]
	Due to \eqref{equ:xlb} in Lemma \ref{lem:mainapprox} and Lemma \ref{lem:getajgetal}, the two points $\bfgh\eta_\nupione,\bfgh\eta_\nupitwo'$ are in $\open$, the $\delta$ neighbourhood of $\supp r$. Due to the G1 assumption \eqref{equ:G1},
	\begin{equation*}
		|\varphi(\bfgh \eta_\nupione)-\varphi(\bfgh\eta_\nupitwo')|\leq\delta^{-1} \sum_{\alpha}v_\alpha d_\alpha(\bfgh\eta_\nupione,\bfgh\eta_\nupitwo').
	\end{equation*}
	 Therefore
	\begin{equation}\label{equ:error}
		|\varphi(\bfgh \eta_\nupione)-\varphi(\bfgh\eta_\nupitwo')|\leq\delta^{-C_0} \beta\sum_\alpha v_\alpha \beta_\alpha^{2k+1}.
	\end{equation}
	
	\textbf{We collect information for different simple roots.} Recall that for a fixed $g$ in $G$ and for $h\in G$, $\k\in \P_0$, we have defined $\tisig_{g}(h,\k)_\alpha=e^{-\alpha(\sigma(gh,\k)-\kappa(g)-n\sigma_\mu)}\alpha^\up(\sg(\ell_g,hk))$. Let 
	$$\varsigma_\alpha:=\frac{\xi d_A(\k_0,\k_1)\alpha^\up(m_0)\partial_\alpha\varphi(\k_{g_0})B_\alpha}{\Pi_{l=1}^{k}\tisig_{{g_{l-1}}}(h_l,\k_{g_l})_\alpha} .$$ 
	Let $\varsigma=(\varsigma_\alpha)_{\alpha\in\Pi}\in\R^\rank$. Hence by \eqref{equ:ghetaeta'}, \eqref{equ:varphigh'}, \eqref{equ:error} and \eqref{equ:fren}
	\begin{equation}\label{equ:xivar}
	|\xi(\varphi(\bfgh x)-\varphi(\bfgh x'))-\l\varsigma,\Pi_{l=1}^{k}\tisig_{{g_{l-1}}}(h_l,\k_{g_l})\r|\leq \expec\delta^{-C_0}\sum_\alpha\beta_\alpha^{2k+1}v_\alpha \xi\ll \beta\delta^{-C_0}\fren.
	\end{equation}
	
	We want to verify that $\|\varsigma\|\in I_\fren$. By \eqref{equ:s0g0m0}, we have
	$$\varsigma_\alpha=\xi d_A(\k_0,\k_1)\partial_\alpha \varphi(\k_{g_0})\expec_\alpha^{k}e^{-\alpha\kappa(g_0)-\cdots-\alpha\kappa(g_{k-1}) }. $$
	By \eqref{equ:g2l}, \eqref{equ:tilxx'}, \eqref{equ:g good} and \eqref{equ:G2} we have $|\varsigma_\alpha|\in  \xi v_\alpha\beta_\alpha^{2k+1} [\delta^{C_0},\delta^{-C_0}]$. 
	Therefore by \eqref{equ:fren},
	\begin{equation*}
	\|\varsigma\|\in\sup_\alpha \xi v_\alpha \beta_\alpha^{2k+1}[\delta^{C_0},\delta^{-C_0}]\in\fren [\delta^{C_0},\delta^{-C_0}]\subset[\fren^{3/4},\fren^{5/4}]=I_\fren.
	\end{equation*}
	
	By definition, the distribution of  $\tisig_{{g_{l-1}}}(h_l,\k_{g_l})$, where $h_l$ satisfies \eqref{equ:g2j+1} with distribution $\mucon{n}$, is the measure $\lambda_l$. 
	Finally, due to $|e^{ix}-e^{iy}|\leq |x-y|$ for $x,y\in\bb R$, inequality \eqref{equ:xivar} implies \eqref{equ:fgxx'}. 
\end{proof}
\textbf{Fourth step:}
We are able to apply sum-product estimates.
\begin{proof}[Proof of Theorem \ref{thm:foudec}] 
	For $l=1,2,\dots k$, Proposition \ref{prop:appmea} and Lemma \ref{lem:tilde noncon} tell us that with $\epss $ small enough depending on $\epsilonfour \epsilontwo$, there exists $C_0$ such that the measures $\lambda_l$ satisfy the assumptions in Proposition \ref{prop:sum-product} with $\fren$.

Proposition \ref{prop:sum-product} implies that for $\fren$ large enough,
	\[\left|\int\exp(i\l\varsigma, x_1\cdots x_k\r)\dd\lambda_1(x_1)\dots\dd\lambda_k(x_k)\right|\leq \fren^{-\epsilonthr }. \]
Then by \eqref{equ:dxx'}, \eqref{equ:fxx'd} and \eqref{equ:fgxx'}, we have
	\[	\left|\int e^{i\xi\varphi( \eta)}r( \eta)\dd\nu( \eta)\right|^2\leq O_\epsilon(\delta^c)+O(\expec\delta^{-C_0}\fren)+\fren^{-\epsilonthr }. \]
	Due to $\expec\delta^{-C_0}\fren=\max_{\alpha\in\Pi} e^{(-\alpha\sigma_\mu+O(1)\epss +\epsilontwo )n}$, take $\epss $ small enough. The proof is complete. 
\end{proof}
%
%
\subsection{From Fourier decay to spectral gap}
\label{sec:fougap}
From now on, we will only consider algebraic $\R$-split semisimple groups without simply connected condition.
In this section, we will prove Theorem \ref{thm:spegaprep} and Theorem \ref{thm:spegap} by using Theorem \ref{thm:foudecsemi}.
\subsubsection*{Derivative of the cocycle}
This part is devoted to the derivative of the cocycle. The results of this part imply that for most $g,h$ in $G$, the difference of the Iwasawa cocycle $\sigma(g,\cdot)-\sigma(h,\cdot)$ satisfies the $(C,r)$ good condition in Definition \ref{defi:C r good} (See Lemma \ref{lem:xgh}). Since the $\alpha$-bundle is trivial on $\P_0$, we will work on $\P_0$. We need to lift the Iwasawa cocycle $\sigma$ to $\P_0$ and we use the same notation $\sigma$.

Let $(\rho,V)$ be an irreducible representation of $G$ with highest weight $\chi$. Let $\alpha$ be a simple root. Let $e_1$ be a unit vector of highest weight in $V$ and let $e_2=Y_\alpha e_1$. The cocycle $\sigma_{V}$ is defined on $\bb PV$. After lifting to $\P_0$, we have $\sigma_V(g,z)=\frac{\|\rho(g)v\|}{\|v\|}$ for $g$ in $G$ and $z=kz_o$ in $\P_0$ with $v=\rho(k)e_1$. We will abbreviate $\rho g$ to $g$ in the proof, because $(\rho,V)$ is the only representation to be studied in this part. 
\begin{lem}\label{lem:derivative cocycle}
	For $\k=k\k_o$ in $\P_0$,
	we have
	\[\partial_\alpha \sigma_V(g,\k)=\frac{\l \rho gv,\rho gu\r}{\|\rho gv\|^2}, \]
	where $v=ke_1$ and $u=ke_2$. 
\end{lem}
\begin{proof}
	Without loss of generality, we suppose that $\k=\k_o$.
	By definition, we have
	\begin{align*}
	\partial_{Y_\alpha}\sigma_V(g,z_o)&=\partial_t\sigma_V(g,\exp(tY_\alpha)\k_o)|_{t=0}=\partial_t\left(\log\frac{\|g\exp(tY_\alpha)e_1\|}{\|\exp(tY_\alpha)e_1\|} \right)\Big|_{t=0}\\
	&=\frac{\l ge_1,gY_\alpha e_1 \r}{\|ge_1\|^2}-\frac{\l e_1,Y_\alpha e_1 \r}{\|e_1\|^2}.
	\end{align*}
	Since the norm is good, eigenvectors of different weights are orthogonal, hence we have $\l e_1,Y_\alpha e_1\r=0$. The result follows.
\end{proof}
From this lemma, we know that the derivative of the cocycle $\sigma_V$ in the direction $Y_\alpha$ is nonzero only if $\chi-\alpha$ is a weight of $V$. 
We fix the distance $d_0$ on $\P_0$, which is defined in Appendix \ref{sec:equi distance}.
\begin{lem}\label{lem:sigma lip iwasawa}
	Let $\delta<1/2$. Let $\widetilde{B^m_{V,g}(\delta)}$ be the preimage of $B^m_{V,g}(\delta)\subset\bp V$ in $\P_0$. For $\k=k\k_o\in \widetilde{B^m_{V,g}(\delta)}$,
	\begin{equation*}
	|\partial_\alpha \sigma_V(g,\k)|\leq \delta^{-C_0}.  
	\end{equation*}
	We also have
	\begin{equation*}
	Lip_{\P_0}(\partial_\alpha \sigma_V(g,\cdot)|_{\widetilde{B^m_{V,g}(\delta)}})\leq \delta^{-C_0}.
	\end{equation*}
\end{lem}
\begin{proof}	
	By Lemma \ref{lem:derivative cocycle}, the hypothesis that $\R ke_1\in B^m_{V,g}(\delta)$ and \eqref{equ:coccar},
	\[|\partial_\alpha\sigma_V(g,\k)|=\left|\frac{\l gke_1,gke_2\r}{\|gke_1\|^2}\right|\leq \frac{\|Y_\alpha\|\|g\|^2\|e_1\|^2 }{\|g\|^2\delta^2\|e_1\|^2 }. \]
	Since the operator norm of $Y_\alpha$ is bounded, we have
	\[ |\partial_\alpha\sigma_V(g,\k)|\leq \delta^{-C_0}.\]
	
	The estimate of the Lipschitz norm is more complicated. Let $v=ke_1,v'=k'e_1,u=ke_2,u'=k'e_2$. We have
	\[|\partial_\alpha\sigma_V(g,\k)-\partial_\alpha\sigma_V(g,\k')|=\frac{	|\l gv,gu\r\|gv'\|^2-\l gv',gu'\r\|gv\|^2|}{\|gv\|^2\|gv'\|^2}. \]
	By the same argument, due to $\R v=\R ke_1\in B^m_{V,g}(\delta)$, we use \eqref{equ:coccar} to give a lower bound of the denominator, that is $$\|gv\|^2\|gv'\|^2\geq \delta^4\|g\|^4\|v\|^2\|v'\|^2=\delta^4\|g\|^4\|e_1\|^4.$$ 
	Use the difference to give an upper bound of the numerator, that is
	\begin{align*}
	&|\l gv,gu\r\|gv'\|^2-\l gv',gu'\r\|gv\|^2|\\
	&\ll \|g\|^3\|e_1\|^3(\|gv-gv'\|+\|gu-gu'\|) \ll \|g\|^4\|v\|^3(\|v-v'\|+\|u-u'\|) .
	\end{align*}
	Therefore we have
	\[|\partial_\alpha\sigma_V(g,\k)-\partial_\alpha\sigma_V(g,\k')|\ll \delta^{-C_0}(\|ke_1-k'e_1\|+\|ke_2-k'e_2\|). \]
	Then by Lemma \ref{lem:equivalent distance}, the proof is complete.
\end{proof}

Let $\wedge^2Sym^2V$ be the exterior square of the symmetric square of $V$. It is a linear space generated by vectors of the form $v_1v_2\wedge v_3v_4$ where $v_i$ is in $V$, for $i=1,2,3,4$. For $g,h$ in $GL(V)$, let $F_{g,h}$ be the linear functional on $\wedge^2Sym^2V$, whose action on the vector $v_1v_2\wedge w_1w_2$ is defined by
\[F_{g,h}(v_1v_2\wedge w_1w_2)=\l hv_1,hv_2 \r\l gw_1,gw_2\r-\l gv_1,gv_2\r\l hw_1,hw_2\r. \]
This formula is well defined because $v_1,v_2$ and $w_1,w_2$ are symmetric, respectively. We also have $F_{g,h}(v_1v_2\wedge w_1w_2)=-F_{g,h}(w_1w_2\wedge v_1v_2)$. 
Since the vectors of form $v_1v_2\wedge w_1w_2$ generate the space $\wedge^2Sym^2V$, the linear form $F_{g,h}$ is uniquely defined.

Suppose that $V$ is a super proximal representation of $G$ with highest weight $\chi $ (Definition \ref{defi:super proximal}). Let $\alpha$ be the unique simple root such that $\chi-\alpha$ is a weight of $V$. The space $\wedge^2Sym^2V$ may be reducible. The two highest weights of $Sym^2V$ are $2\chi,\ 2\chi-\alpha$, whose eigenspaces have dimension 1. Hence, the highest weight of $\wedge^2Sym^2V$ is $4\chi-\alpha$, and the eigenspace has dimension 1. Let $W$ be the irreducible subrepresentation of $\wedge^2Sym^2V$ with the highest weight $\chi_1:=4\chi-\alpha$. Recall that $V_{\chi_1,\eta}$ is the image of $\eta\in\P$ in $\bp W$. In the following lemma, we abbreviate $\rho(g),\rho(h)$ to $g,h$.
\begin{lem}\label{lem:-derivative}
	Let $\delta<1/2$.  Let $V$ be a super proximal representation of $G$ and let $\alpha$ be the unique simple root such that $\chi-\alpha$ is a weight of $V$. If $g,h$ in $G$ and $\k=k\k_o\in\P_0, \eta=\pi(\k)$ satisfy
	\begin{itemize}
		\item[(1)] $\ell_{h}^{-1}V^{\chi},\ell_{ h}^{-1} V^{\chi-\alpha}\in B^m_{V,g}(\delta),\gamma_{1,2}(  g)\leq \delta^3$,
		\item[(2)] $\delta(V_{\chi_1,\eta},F_{  g,  h}|_W)>\delta$ and $V_{\chi ,\eta}\in B^m_{V,g}(\delta)\cap B^m_{V,h}(\delta)$,
	\end{itemize}
	then
	\begin{equation*}
	|\partial_\alpha (\sigma_V(g,\k)-\sigma_V(h,\k))|\geq \delta^{C_0}.
	\end{equation*}
\end{lem}
\begin{rem}	
	This is similar to the non local integrability property as defined in \cite{dolgopyat1998decay} \cite{naud2005expanding} and \cite{stoyanov2011spectra}. Although the above two conditions are complicated, we will see later that outside a set of exponentially small measure, all pairs $g,h$ satisfy these conditions.
	
	The key idea here is to use other representations to linearise polynomial functions on $V$. As long as the function is linear, we will have a good control of it. Another point is that the image of $\P$ stays in the same irreducible subrepresentation.
\end{rem}
\begin{proof}[Proof of Lemma \ref{lem:-derivative}]
	By Lemma \ref{lem:derivative cocycle}, let 
	\begin{equation}\label{equ:L kvu}
	L:=\partial_\alpha (\sigma_V(g,\k)-\sigma_V(h,\k))=\frac{F_{g,h}(v^2\wedge vu)}{\|gv\|^2\|hv\|^2},
	\end{equation}
	where $v=ke_1$ and $u=kY_\alpha e_1$ as in Lemma \ref{lem:derivative cocycle}.
	\begin{lem}\label{lem:Fgh}
		If $g,h$ satisfy assumption $(1)$, then the operator norm satisfies
		\[\|F_{g,h}|_W\|\geq \delta^{C_0}\|g\|^2\|h\|^2. \]
	\end{lem}
	\begin{proof}
		Using the Cartan decomposition and good norm, we can suppose that $h$ is diagonal and $h=\diag(a_1,a_2,\cdots, a_n )$ with $a_1\geq a_2\geq \dots\geq a_n$. By Definition \ref{defi:super proximal}, we know that $he_1=a_1e_1$ and $he_2=a_2e_2$. Assumption $(1)$ becomes
		\begin{equation}\label{equ:assum}
		\delta(\R e_1,y^m_g),\delta(\R e_2,y^m_g)>\delta,\gamma_{1,2}(g)\leq \delta^3.
		\end{equation}
		In \eqref{equ:L kvu}, let $\k=\k_o$, then $v=e_1, u=e_2$, which make
		\[ \l hv,hu\r=\l a_1e_1,a_2e_2\r=0.\] 
		Therefore, due to 
		\[\ \l v_1,v_2\r\geq \|v_1\|\|v_2\|-\|v_1\wedge v_2\|, \]
		for $v_1,v_2$ in $V$, we have
		\begin{align*}
		F_{g,h}(e_1^2\wedge e_1e_2)=a_1^2\l ge_1,ge_2\r&\geq a_1^2(\|ge_1\|\|ge_2\|-\|ge_1\wedge ge_2\|).
		\end{align*}
		Then \eqref{equ:coccar} and \eqref{equ:assum} imply 
		\[ F_{g,h}(e_1^2\wedge e_1e_2)\geq\|h\|^2\|g\|^2(\delta^2-\gamma_{1,2}(g)).\]
		The proof is complete.
	\end{proof}
	By Definition \ref{defi:super proximal}, the representation $\wedge^2 Sym^2V$ is a proximal representation.  Due to $\R (v^2\wedge vu)=\R k(e_1^2\wedge e_1e_2)=k V^{\chi_1}$, the line $\R(v^2\wedge vu)$ is contained in the $K$-orbit of the subspace of highest weight $V^{\chi_1}$. Since $V^{\chi_1}$ is in $W$, we see that $v^2\wedge vu$ is also in $W$.  By \eqref{equ:L kvu},
	\[L=\frac{F_{g,h}(v^2\wedge vu)}{\|F_{g,h}|_W\|}\frac{\|g\|^2\|h\|^2}{\|gv\|^2\|hv\|^2}\frac{\|F_{g,h}|_W\|}{\|g\|^2\|h\|^2}. \]
	When $\eta$ satisfies assumption $(2)$, Lemma \ref{lem:-derivative} follows by applying \eqref{equ:coccar} to $\|gv\|^2,\|hv\|^2$ and by Lemma \ref{lem:Fgh}.
\end{proof}
\subsubsection*{Proof of the spectral gap}
Here we will prove the theorem of uniform spectral gap. The first part is classic, where we use some ideas of Dolgopyat \cite{dolgopyat1998decay} to transform the problem to an effective estimate in Proposition \ref{prop:L1pbf}, see also \cite{naud2005expanding} and \cite{stoyanov2011spectra}. The key observation is that this effective estimate (Proposition \ref{prop:L1pbf}) can be obtained by the Fourier decay, regarding the difference of cocycle as a function on $\PP$.
The intuition here is from Lemma \ref{lem:-derivative}. When $g,h$ are in general position and $\eta$ not too close to $\zeta^m_g,\zeta^m_h$, the difference $\varphi(\eta)=\sigma(g,\eta)-\sigma(h,\eta)$ will be $(C,r)$ good (Definition \ref{defi:C r good}). But in order to accomplish this, we need some sophisticated cutoff, which makes the proof complicated.

Recall the family of representations $\{V_\alpha\}_{\alpha\in\Pi}$ defined in Lemma \ref{lem:tits}. This family is super proximal by Lemma \ref{lem:super proximal} and with highest weight $\chi_{\alpha}$ equal to a multiple of fundamental weight $\tilde{\omega}_\alpha$.
We are in semisimple case and we know that $\frak b^*=\frak a^*$. 
We can write $\vartheta$ in $\frak a^*$ as a linear combination of weights, $\{\chi_\alpha|\alpha\in \Pi\}$, that is
\[\vartheta=\sum_{\alpha\in\Pi}\vartheta_\alpha\chi_\alpha. \]
Set $|\vartheta|=\max_{\alpha\in\Pi}|\vartheta_\alpha|$. 

We want to treat the spectral gap on the flag variety $\P$ and the projective space $\bp V$ at the same time, where $V$ is an irreducible representation of $G$. Let $X$ be $\P$ or $\bp V$. Let $\sigma:G\times X\rightarrow E$ be the cocycle, which is 
\begin{itemize}
	\item given by the Iwasawa cocycle $\sigma$ and $E=\frak a$ when $X=\P$, 
	\item given by $\sigma_V$ (defined in \eqref{equ:sigma V}) and $E=\R$ when $X=\bp V$.
\end{itemize}
 Let $E_{\bb C}=E\otimes_\R\bb C$ and $E^*_{\bb C}$ be the dual space of $E_{\bb C}$. For $z\in E^*_{\bb C}$, write $z=\realpart+i\vartheta$, where $\realpart,\,\vartheta$ are elements in $E^*$.  Recall that the transfer operator $P_z$ is defined as follows: for $|\realpart|$ small enough and for $f$ in $C^0(X)$, $x$ in $X$
\[P_zf(x)=\int_{G}e^{z\sigma(g,x)}f(gx)\dd\mu(g), \]
where $z=\realpart+i\vartheta$.
Recall that for $f$ in $C^\gamma(X)$ we defined $c_\gamma(f)=\sup_{x\neq x'}\frac{|f(x)-f(x')|}{d(x,x')^\gamma}$ and $|f|_\gamma=|f|_\infty+c_\gamma(f)$. 

We state our main result of this section
\begin{thm}\label{prop:spegap}
	Let $\mu$ be a Zariski dense Borel probability measure on $G$ with a finite exponential moment. For $\gamma>0$ small enough, there exist $\rho<1, C>1$ such that for all $\vartheta$ and $\realpart$ in $E^*$ with $|\vartheta|$ large enough, $|\realpart|$ small enough and $f$ in $C^{\gamma}(X)$, $n$ in $\bb N$ we have
	\begin{equation*}
	|P^{n}_{\realpart+i\vartheta}f|_{\gamma}\leq C|\vartheta|^{2\gamma}\rho^n|f|_{\gamma}.
	\end{equation*}
\end{thm}
\begin{rem}
	Here we should be careful that the distances on $\bp V$ and $\P$ are defined in \eqref{equ:distance x x'} and \eqref{equ:distance eta eta'}. They are not the Riemannian distances defined in the introduction. But on a compact Riemannian manifold, different Riemannian distances are equivalent. In particular, every Riemannian distance on $\P$ is equivalent to the $K$-invariant Riemannian distance on $\P$. By Lemma \ref{lem:equivalence distance P}, we know it is equivalent to the distances defined in \eqref{equ:distance eta eta'}. The case of the projective space $\bp V$ is similar. Hence, the norms $|\cdot|_\gamma$ induced by different distances are equivalent.
\end{rem}
\begin{rem}\label{rem:norm}
	We explain here that the result holds for any norm.
	
	If we have another norm $\|\cdot \|_1$ on $V$. Let $\sigma_1$ be the new cocycle defined with respect to the norm $\|\cdot\|_1$. Let $\psi(x)=\log\frac{\|v\|_1}{\|v\|}$ for $x=\R v$ in $\bp V$. Then \[\sigma_1(g,x)=\sigma_V(g,x)+\psi(gx)-\psi(x),\] 
	which means the difference of two cocycles is a coboundary. This function $\psi$ is Lipschitz, due to equivalence of norms on finite-dimensional vector spaces. Let $T_zf(x)=e^{z\psi(x)}f(x)$. By Lipschitz property of $\psi$, we have
	\[|T_zf|_\gamma\leq Ce^{C|\realpart|}|z|^{\gamma}|f|_\gamma, \]
	where $C$ depends on $|\psi|_{Lip}$. 
	We know that
	\[P_{z\sigma_1}=T_z^{-1}P_{z\sigma_V}T_z, \]
	hence the same spectral gap property also holds for the norm $\|\cdot\|_1$ with different constants.
\end{rem}
Theorem \ref{thm:spegaprep} and Theorem \ref{thm:spegap} follow directly from Theorem \ref{prop:spegap}. 

We start with standard \textit{a priori} estimates. When $z=0$, we will write $P$ for $P_0$.
\begin{prop}\label{prop:spectral real}
	For every $\gamma>0$ small enough, there exist $C>0$ and $0<\rho<1$ such that for all $f$ in $C^\gamma(X)$, $|\realpart|$ small enough and $n\in\bb N$
	\begin{align}
		&\label{equ:expmom}
		|P^n_zf|_\infty\leq C^{|\realpart|n}|f|_\infty,\\
	&\label{equ:infpnf}|P^nf|_{\infty}\leq\left|\int_Xf\dd\nu\right|+C\rho^n|f|_{\gamma},\\
	&\label{equ:gampnf}
	c_\gamma(P^n_zf)\leq C(C^{|\realpart|n}|\vartheta|^{\gamma}|f|_{\infty}+\rho^nc_\gamma(f)).
	\end{align}
\end{prop}
Inequality \eqref{equ:expmom} is a consequence of exponential moment and the H\"older inequality.
For \eqref{equ:infpnf}, please see \cite[V, Thm.2.5]{bougerol1985products} and \cite[Prop 11.10, Lem.13.5]{benoistquint} for more details. This inequality \eqref{equ:infpnf} is a consequence of the fact that the action of $G$ on $X$ is contracting.
The third inequality \eqref{equ:gampnf} is called the Lasota-Yorke inequality, whose proof is classic.

We reduce Theorem \ref{prop:spegap} to Proposition \ref{prop:L1pbf}. The reduction is standard, using Proposition \ref{prop:spectral real}. Please see \cite{dolgopyat1998decay} for more details. For $f$ in $C^\gamma(X)$, we define another norm $|f|_{\gamma,\vartheta}=|f|_{\infty}+c_{\gamma}(f)/|\vartheta|^\gamma$ for $\vartheta\neq 0$. 
\begin{prop}\label{prop:L1pbf}
	For every $\gamma>0$ small enough, for $|\vartheta|$ large enough and $|\realpart|$ small enough, there exist $\epsilontwo ,\ \Cone >0$ such that
	for $f$ in $C^\gamma(X)$ and $|f|_{\gamma,\vartheta}\leq 1$, we have 
	\begin{equation}\label{equ:L1pbf}
	\int \left|P^{[\Cone \ln |\vartheta|]}_{\realpart+i\vartheta}f\right|^2\dd\nu\leq e^{-\epsilontwo \ln|\vartheta|}.
	\end{equation}
\end{prop}

Now we will distinguish two cases. \textbf{We claim that the case $\bp V$ is a corollary of the case $\P$ up to a constant.} Recall that the stationary measure on $\bp V$ is written as $\nu_V$. Let $f$ be a function in $C^\gamma(\bp V)$ and $|f|_{\gamma,\vartheta}\leq 1$.
The estimate only depends on the value of $f$ on the support of the stationary measure $\nu_V$. By Lemma \ref{lem:stauni}, the stationary measure on $\bp V$ is the pushforward measure of the stationary measure $\nu$ on $\P$. Hence we can define the function $\tilde{f}$ on $\P$ by 
\[\tilde{f}(\eta)=f(V_{\chi,\eta}), \]
where $\chi$ is the highest weight of $V$. Then by $\sigma_V(g,V_{\chi,\eta})=\chi\sigma(g,\eta)$ (see \eqref{equ:representation cocycle}),
\begin{equation*}
\int \left|P^{[\Cone \ln|\vartheta|]}_{\realpart+i\vartheta}f \right|^2\dd\nu_V=\int \left|P^{[\Cone \ln|\vartheta|]}_{(\realpart+i\vartheta)\chi}\tilde f \right|^2\dd\nu.
\end{equation*}
We will verify that $\tilde{f}$ satisfies $|\tilde{f}|_{\gamma,\vartheta}\ll 1$. By \eqref{equ:profla}, for two distinct points $\eta,\eta'$ in $\P$ we have
\begin{align*}
	\frac{|\tilde{f}(\eta)-\tilde{f}(\eta')|}{d(\eta,\eta')^\gamma}=\frac{|\tilde{f}(\eta)-\tilde{f}(\eta')|}{d(V_{\chi,\eta},V_{\chi,\eta'})^\gamma} \frac{d(V_{\chi,\eta},V_{\chi,\eta'})^\gamma}{d(\eta,\eta')^\gamma}\ll\frac{|{f}(V_{\chi,\eta})-{f}(V_{\chi,\eta'})|}{d(V_{\chi,\eta},V_{\chi,\eta'})^\gamma} =|f|_\gamma.
\end{align*}
Hence with different constants, we can deduce the case $\bp V$ from the case $\P$.

We only need to prove Proposition \ref{prop:L1pbf} for the case $\P$.
\begin{proof}[From Fourier decay to Proposition \ref{prop:L1pbf}]
	We need to reduce \eqref{equ:L1pbf} to Fourier decay (Theorem \ref{thm:foudecsemi}). Let
	\begin{equation}\label{equ:C1definition}
	n=[\Cone \log|\vartheta|] \text{ and }\delta=e^{-\epss  n}
	\end{equation}
	 (with $\Cone \geq\max_{\alpha\in\Pi} \{1/\alpha\sigma_\mu\}+1$ and $\epss >0$ to be determined later), and let $G_{n,\epss,\alpha }$ be the subset of $G\times G$ defined as the set of couples which satisfy Lemma \ref{lem:-derivative} (1) with $V=V_\alpha$. Let
	 $$G_{n,\epss}=\{g\in G|\|\kappa(g)-n\sigma_\mu\|\leq n\epsilon \}^2\cap (\cap_{\alpha\in\Pi}G_{n,\epss,\alpha})\subset G\times G.$$
	Let 
	\[ A_{g,h}:=\int_Xe^{z\sigma(g,\eta)+\bar z\sigma(h,\eta)}f(g\eta)\bar f(h\eta)\dd\nu(\eta). \]
	Then
	\begin{equation}\label{equ:pnzf}
	\begin{split}
	&\int |P^n_zf|^2\dd\nu=\int e^{z\sigma(g,\eta)+\bar z\sigma(h,\eta)}f(g\eta)\bar f(h\eta)\dd\nu(\eta)\dd\mu^{*n}(g)\dd\mu^{*n}(h)\\
	&=\int_{G_{n,\epss }}A_{g,h}\dd\mu^{*n}(g)\dd\mu^{*n}(h)+\int_{G_{n,\epss }^c}A_{g,h}\dd\mu^{*n}(g)\dd\mu^{*n}(h).
	\end{split}
	\end{equation}
	
	\textbf{We first compute the term with $(g,h)$ outside of $G_{n,\epss}$,} where the behaviour is singular. By the Cauchy-Schwarz inequality,
	\begin{equation}\label{equ:gncA}
		\left|\int_{G_{n,\epss }^c}A_{g,h}\dd\mu^{*n}(g)\dd\mu^{*n}(h)\right|^2\leq \mucon{n}\otimes\mucon{n}(G_{n,\epsilon}^c)\int |A_{g,h}|^2\dd\mu^{*n}(g)\dd\mu^{*n}(h).
	\end{equation}
	By large deviation principle (Proposition \ref{prop:lardev1}, Proposition \ref{prop:large deviation projective}), the set $G_{n,\epsilon}^c$ has exponentially small $\mucon{n}\otimes\mucon{n}$ measure, that is
	\begin{equation}\label{equ:mu Gnc}
		\mucon{n}\otimes\mucon{n}(G_{n,\epsilon}^c)\ll_\epsilon\delta^c.
	\end{equation}
	By $\|f\|_\infty\leq 1$ and \eqref{equ:expmom}, we have
	\begin{equation}\label{equ:A gh2}
	 \int |A_{g,h}|^2\dd\mu^{*n}(g)\dd\mu^{*n}(h)\leq |P_{2\realpart}^n\B|_\infty^2\leq C^{4n|\realpart|}.
	 \end{equation}
	When $|\realpart|$ is small enough depending on $\epsilon$, by \eqref{equ:gncA}, \eqref{equ:mu Gnc} and \eqref{equ:A gh2}
	\begin{equation}\label{equ:gncA2}
		\int_{G_{n,\epss }^c}A_{g,h}\dd\mu^{*n}(g)\dd\mu^{*n}(h)\ll_\epsilon \delta^{c/2}\leq |\vartheta|^{-c\epsilon/(2\Cone )}.
	\end{equation}

\textbf{We compute the major term, that is $(g,h)$ in $G_{n,\epss}$.} 
We want to use Theorem \ref{thm:foudecsemi} to control this part with $\varphi=|\vartheta|^{-1}\vartheta(\sigma(g,\eta)-\sigma(h,\eta))$ and a suitable $r$. In order to apply Theorem \ref{thm:foudecsemi}, we need that $\varphi$ is $(C,r)$ good, which will be accomplished by multiplying by a smooth cutoff. The most important is the G2 assumption \eqref{equ:G2}, which will be verified with the help of Lemma \ref{lem:-derivative}. Hence we want that $r$ vanishes when $\eta$ does not satisfy Lemma \ref{lem:-derivative} (2).

Let $X_{g,h,\alpha}$ be the subset of $\P$, defined as the set of elements which satisfy Lemma \ref{lem:-derivative} (2) with $V=V_\alpha$. Let $X_{g,h}=\bigcap_{\alpha\in\Pi}X_{g,h,\alpha}$.
	Let $\mollifier$ be a smooth function on $\bb R$ such that $\mollifier|_{[0,\infty)}=1$, $\mollifier$ takes values in $[0,1]$, $\supp\mollifier\subset[-1,\infty)$ and $|\mollifier'|\leq 2$. Set $\mollifier_\delta(x)=\mollifier(x/\delta)$ for $x\in \R$.
	Let $\sigma_\alpha:=\sigma_{V_\alpha}=\chi_\alpha\sigma$,
	\begin{equation}\label{equ:varphi eta}
	\varphi(\eta)=|\vartheta|^{-1}\vartheta(\sigma(g,\eta)-\sigma(h,\eta))=|\vartheta|^{-1}\sum_{\alpha\in\Pi}\vartheta_\alpha(\sigma_\alpha(g,\eta)-\sigma_\alpha(h,\eta))
	\end{equation}
	and 
	\begin{equation}\label{equ:r eta}
	r(\eta)=f(g\eta)\bar f(h\eta)e^{\realpart(\sigma(g,\eta)+\sigma(h,\eta))}\prod_{\alpha\in\Pi} \mollifier_{\alpha}(\eta),
	\end{equation}
	where
	\begin{align*}
		 \mollifier_{\alpha}(\eta)&=\mollifier_\delta(4\delta_\alpha(\eta,\zeta^m_g)-4\delta)\mollifier_\delta(4\delta_\alpha(\eta,\zeta^m_h)-4\delta)\mollifier_\delta(4\delta(V_{4\chi_\alpha-\alpha,\eta},F_{\rho_\alpha g,\rho_\alpha h})-4\delta),
	\end{align*}
	where $\delta_\alpha$ is defined to be $$\delta_\alpha(\eta,\zeta^m_g)=\delta(V_{\alpha,\eta},y^m_{\rho_\alpha(g)}).\nomentry{$\delta_\alpha(\eta,\zeta)$}{}$$ 
	The choice of $ \mollifier_{\alpha}$ is sophisticated. We only need to keep in mind that they come from Lemma \ref{lem:-derivative}. Then $e^{i|\vartheta|\varphi}r(\eta)$ equals $e^{z\sigma(g,\eta)+\bar z\sigma(h,\eta)}f(g\eta)\bar f(h\eta)$ on $X_{g,h}$.
	\begin{lem}\label{lem:xgh}
		Let $\epsilonzer ,\epsilonone $ be given by Theorem \ref{thm:foudecsemi}. Let $(g,h)$ be in $G_{n,\epsilon}$. With 
		$\epsilon$ small enough depending on $\epsilonzer $ and $|\realpart|$ small enough depending on $\epss $ and $\epsilonone $, for $\varphi,r$ defined in \eqref{equ:varphi eta} and \eqref{equ:r eta} we have that $\varphi$ is $(|\vartheta|^{\epsilonzer },r)$ good and  $c_\gamma(r) \leq |\vartheta|^{\epsilonzer }$, $|r|_\infty\leq |\vartheta|^{\epsilonone /2}$.
	\end{lem}
	By Lemma \ref{lem:xgh}, we can fix a value of $\epsilon$ and functions $\varphi$ and $r|\vartheta|^{-\epsilonone /2}$ satisfying the condition in Theorem \ref{thm:foudecsemi}. (Theorem \ref{thm:foudecsemi} still holds when $r$ is a complex function) Hence Theorem \ref{thm:foudecsemi} implies
	\begin{equation}\label{equ:etheta}
	\left|\int e^{i |\vartheta| \varphi(\eta)}r(\eta)\dd\nu(\eta)\right|\leq |\vartheta|^{-\epsilonone}\|r\|_\infty\leq |\vartheta|^{-\epsilonone /2} .
	\end{equation}
	Due to the definition of $G_{n,\epsilon}$, the difference between $A_{g,h}$ and $\int e^{i |\vartheta| \varphi(\eta)}r(\eta)\dd\nu(\eta)$ is bounded by 
	\begin{equation}\label{equ:Xgh}
	\nu(X_{g,h}^c)e^{|\realpart|(\|\kappa(g)\|+\|\kappa(h)\|)}\leq e^{2n|\realpart|(\|\sigma_\mu\|+\epsilon)}\sum_{\alpha\in\Pi}\nu(X_{g,h,\alpha}^c) \leq |\vartheta|^{c\epss/2\Cone}\sum_{\alpha\in\Pi}\nu(X_{g,h,\alpha}^c),
	\end{equation}
	if $|\realpart|$ is small enough. Using the regularity of stationary measure \eqref{equ:regularity stataionary measure} with $V=W_\alpha$, the irreducible subrepresentation of $\wedge^2Sym^2V_\alpha$ with the highest weight, we have
	\begin{equation}\label{equ:Xgh1}
	\nu\{\eta\in\P| \delta(V_{4\chi_\alpha-\alpha,\eta},F_{\rho_\alpha g,\rho_\alpha h})<\delta\}\ll_\epsilon e^{-c\epsilon n}.
	\end{equation}
	Using the regularity of stationary measure \eqref{equ:regularity stataionary measure} with $V=V_\alpha$, we obtain
	\begin{equation}\label{equ:Xgh2}
		\nu\{\eta\in\P| V_{\alpha,\eta}\in B^m_h(\delta)\cup B^m_g(\delta)\}\ll_\epsilon e^{-c\epsilon n}.
	\end{equation}
	Hence by \eqref{equ:Xgh}-\eqref{equ:Xgh2}, we have
	 \begin{equation}\label{equ:Xghc}
	 \nu(X_{g,h}^c)\ll_\epsilon e^{- c \epss n}=|\vartheta|^{-c\epss/\Cone }.
	 \end{equation}
	For $(g,h)$ in $G_{n,\epss}$, by \eqref{equ:etheta}, \eqref{equ:Xgh} and \eqref{equ:Xghc}
	\[A_{g,h}\ll |\vartheta|^{-\epsilonone /2}+|\vartheta|^{-c\epss/2\Cone }.\]
	Combined with \eqref{equ:pnzf} and \eqref{equ:gncA2}, the proof is complete by setting $\epsilontwo=\min\{\frac{\epsilonone}{2},\frac{c\epsilon}{4\Cone} \}$.
\end{proof}
It remains to prove Lemma \ref{lem:xgh}.
\begin{proof}[Proof of Lemma \ref{lem:xgh}]
	\textbf{We first verify that $\varphi$ is $(|\vartheta|^{\epsilonzer },r)$ good.}
	Since $\epsilon$ will be taken small enough, we can suppose $|\vartheta|^{-\epsilonzer }\leq \delta/4$. Let $\open$ be the $|\vartheta|^{-\epsilonzer }$ neighbourhood of $\supp r$. Then for $\eta\in\open$, we have $\delta_\alpha(\eta,\zeta^m_g)\geq \delta/2$ for $\alpha$ in $\Pi$.
	
	The function $\varphi$ is a sum of functions, each one is the lift of a function on $\bp V_\alpha$ for some simple root $\alpha$. We write $\varphi=\sum_{\alpha\in\Pi}\varphi_\alpha$ where $\varphi_\alpha(\eta)=|\vartheta|^{-1}\vartheta_\alpha(\sigma_\alpha(g,\eta)-\sigma_\alpha(h,\eta))$. By Lemma \ref{lem:lift varphi}, that is $\partial_{\alpha'}\varphi_\alpha=0$ for $\alpha'\neq \alpha$, in order to verify the $(|\vartheta|^{\epsilonzer },r)$ good condition, it is enough to verify G1-G3 assumptions \eqref{equ:G1}-\eqref{equ:G3} for $\varphi_\alpha$ and the G4 assumption \eqref{equ:G4} for $\varphi$. Since G1-G3 are linear, we can forget the coefficients $|\vartheta|^{-1}\vartheta_\alpha$ in $\varphi_\alpha$.
	
	Now, we verify G1-G3 assumptions. We fix a simple root $\alpha$ and consider $\varphi=\varphi_\alpha=\sigma_\alpha(g,\cdot)-\sigma_\alpha(h,\cdot)$. 
	 Recall that $v_\alpha=\sup_{\eta\in\supp r}|\partial_\alpha\varphi(\eta)|$. Since $\open$ satisfies the hypothesis of Lemma \ref{lem:sigma lip iwasawa} with $V=V_\alpha$, we have
	\begin{equation}\label{equ:lip phi alpha}
		v_\alpha,Lip_{\P_0}(\partial_\alpha\varphi|_{\pi^{-1}\open})<\delta^{-C_0}.
	\end{equation}
	Since $(g,h)\in G_{n,\epsilon}$ satisfies Lemma \ref{lem:-derivative}(1) and the support of $r$ satisfies Lemma \ref{lem:-derivative}(2), for $\eta$ in the support of $r$, by Lemma \ref{lem:-derivative},
	\begin{equation*}
		|\partial_\alpha\varphi(\eta)|>\delta^{C_0}\geq \delta^{C_0}v_\alpha
	\end{equation*}
	which is the G2 assumption \eqref{equ:G2}. This also implies
	\begin{equation}\label{equ:phi geq}
	v_\alpha>\delta^{C_0},
	\end{equation}
	the G4 assumption \eqref{equ:G4}. By \eqref{equ:lip phi alpha}, we have the G3 assumption \eqref{equ:G3}. Let $$\varphi_1(x)=\sigma_{V_\alpha}(\rho_{\alpha}(g),x)-\sigma_{V_\alpha}(\rho_{\alpha}(h),x)$$ 
	be a function on $\bp V_\alpha$, then $\varphi_1(V_{\alpha,\eta})=\varphi(\eta)$. Since $\open$ satisfies the hypothesis of Lemma \ref{lem:gBmg}, this Lemma implies	
	\[\frac{|\varphi(\eta)-\varphi(\eta')|}{d_\alpha(\eta,\eta')}= \frac{|\varphi_1(V_{\alpha,\eta})-\varphi_1(V_{\alpha,\eta'})|}{d(V_{\alpha,\eta},V_{\alpha,\eta'})}\leq |Lip_{\bp V_\alpha}\varphi_1|<\delta^{-C_0}\leq \delta^{-C_0}v_\alpha, \]
	which is the G1 assumption \eqref{equ:G1}. 
	
	For general $\varphi$, it remains to verify the G4 assumption \eqref{equ:G4}. There exists a simple root $\alpha$ such that $|\vartheta_\alpha|=|\vartheta|$. Since $\varphi_\alpha$ satisfies the G4 assumption and due to Lemma \ref{lem:lift varphi} we have $|\partial_\alpha\varphi|=|\partial_\alpha\varphi_\alpha|$, the function $\varphi$ also satisfies the G4 assumption.
	
	\textbf{Finally, we verify the terms $c_\gamma(r)$ and $|r|_\infty$.} 
		\begin{lem}\label{lem:frho}
		For $0<\gamma\leq 1$, let $f, \mollifier$ be two $\gamma$-H\"older functions on a compact metric space $X$. Then
		\begin{equation*}
		c_\gamma(\mollifier f)\leq c_\gamma(\mollifier)\|f|_{\supp \mollifier}\|_\infty+|\mollifier|_\infty c_\gamma(f|_{\supp \mollifier}).
		\end{equation*}
	\end{lem}
	\noindent The proof of Lemma \ref{lem:frho} is elementary.
	Recall that 
		$$r(\eta)=f(g\eta)\bar f(h\eta)e^{\realpart(\sigma(g,\eta)+\sigma(h,\eta))}\prod_{\alpha\in\Pi} \mollifier_{\alpha}(\eta).$$
	For the infinity norm, due to $(g,h)\in G_{n,\epss}$, we have
	\begin{equation*}
		|r|\leq e^{|\realpart|(\|\kappa(g)\|+\|\kappa(h) \|)}\leq e^{|\realpart|(2\|\sigma_\mu\|+2\epsilon)n}\leq |\vartheta|^{|\realpart|\Cone (2\|\sigma_\mu\|+2\epsilon)}.
	\end{equation*}
	Take $|\realpart|$ small enough, then $|r|_\infty\leq |\vartheta|^{\epsilonone /2}$.
	
	For the term $c_\gamma(r)$, we only need to verify that each term in the formula of $r$ has a bounded $c_\gamma$ value. Due to Lemma \ref{lem:frho}, we only need to verify the $c_\gamma$ value on $X_{g,h}$.
	\begin{itemize}
		\item Since the action of $g$ on $X_{g,h}$ is contracting, by Lemma \ref{lem:gBmgmul}, we have
		\[c_\gamma(f(g\cdot)|_{X_{g,h}})\leq c_\gamma(f)(Lip\  g|_{X_{g,h}})^\gamma\leq (|\vartheta|\beta\delta^{-2})^\gamma. \]
		Due to \eqref{equ:C1definition}, we have $\log \beta=-n\min_{\alpha\in\Pi}\alpha\sigma_\mu<-n/\Cone \leq -\log|\vartheta|$. Therefore $c_\gamma(f(g\cdot)|_{X_{g,h}})\leq \delta^{-C_0}$.
		\item Due to 
		\[|e^a-e^b|\leq\max\{e^a,e^b \}|a-b|^\gamma \]
		for all $a,b$ in $\R$ and $0\leq \gamma\leq 1$, by Lemma \ref{lem:gBmgmul},
		\[c_\gamma(e^{\realpart\sigma(g,\cdot)}|_{X_{g,h}})\leq e^{|\realpart|\|\kappa(g)\|}(Lip\realpart\sigma(g,\cdot)|_{X_{g,h}})^\gamma\leq e^{|\realpart|(\|\sigma_\mu\|+\epsilon)n}|\realpart|^\gamma\delta^{-\gamma}. \]
		Hence when $|\realpart|$ is small enough depending on $\sigma_\mu$, we obtain $c_\gamma(e^{\realpart\sigma(g,\cdot)}|_{X_{g,h}})\leq \delta^{-C_0}$. 
		\item In $c_\gamma( \mollifier_{\alpha})$, the only term we need to be careful about is $\mollifier_\delta(4\delta(V_{4\chi_\alpha-\alpha,\eta},F_{\rho_\alpha g,\rho_\alpha h})-4\delta)$. By Lemma \ref{lem:profla}, we have $d(V_{4\chi_\alpha-\alpha,\eta},V_{4\chi_\alpha-\alpha,\eta'})\ll d(\eta,\eta')$. Hence the $c_\gamma$ value of this term is also bounded by $\delta^{-C_0}$.
	\end{itemize}
	The proof is complete.
\end{proof}

\subsection{Exponential error term}
\label{sec:expdec}
In this section, we will prove Theorem \ref{thm:renewal} that the speed of convergence in the renewal theorem is exponential using our result on the spectral gap (Theorem \ref{prop:spegap}).
Recall that $X=\bp V$, where $(\rho, V)$ is an irreducible representation of $G$ with highest weight $\chi$. We have defined a renewal operator $R$ as follows: For a positive bounded Borel function $f$ on $\bb R$, a point $x$ in $X$ and a real number $t$, we set
\begin{align*}
Rf(x,t)=\sum_{n=0}^{+\infty}\int_Gf(\sigma_V(\rho(g),x)-t)\dd\mu^{*n}(g)
\end{align*}
and 
\begin{equation*}
R_Pf(x,t)=\sum_{n=0}^{+\infty}\int_Gf(\log\|\rho(g)\|-t)\dd\mu^{*n}(g).
\end{equation*}

Recall $P_z$ is the transfer operator defined by $P_zf(x)=\int_G e^{z\sigma_V(\rho(g),x)}f(gx)\dd\mu(g)$. For $\eta>0$, let $\ceta=\{z\in\bb C|\,|\Re z|<\eta \}$.
Using the analytical Fredholm theorem, we summarize the property of $P_z$.
\begin{prop}\label{prop:invrsetran}
	Under the same assumptions as in Theorem \ref{thm:renewal}, for any $\gamma>0$ small enough, there exists $\eta>0$ such that on $\ceta$ the transfer operator $P_z$ is a bounded operator on $C^\gamma(X)$ and depends analytically on $z$. Moreover there exists an analytic operator $U(z)$ on $\ceta$ such that the following holds on $\ceta$
	\begin{equation*}
	(I-P_z)^{-1}=\frac{1}{\sigma_{V,\mu}z}N_0+U(z),
	\end{equation*}
	where $N_0$ is the operator defined by $N_0f=\int_X f\dd\nu_V$ and $\sigma_{V,\mu}=\chi\sigma_{\mu}>0$ is the Lyapunov constant. There exists $C>0$ such that for $z\in\ceta$
	\begin{equation}\label{equ:uz}
	\|U(z)\|_{C^\gamma\rightarrow C^\gamma}\leq C(1+|\Im z|)^{2\gamma}.
	\end{equation}
\end{prop}
This is a generalization of \cite[Prop. 4.1]{li2017fourier} and \cite[Theorem 4.1]{boyer2016renewalrd}, and the proof is exactly the same. The main difference is that the spectral radius of $P_z$ is bounded below $1$ on $\ceta$ (except at $0$), due to Theorem \ref{prop:spegap}. From this we have the analytic continuation of $U(z)$ to $\ceta$ and the bound on the operator norm of $U(z)$. 

Now, we give the precise statement and the proof of Theorem \ref{thm:renewal}.
\begin{prop}\label{prop:reniwa}
	Under the same assumptions as in Theorem \ref{thm:renewal}, there exists $\epsilon>0$ such that for $ f\in C_c^{2}(\bb R)$, we have
	\begin{equation*}
	Rf(x,t)=\frac{1}{\sigma_{V,\mu}}\int_{-t}^{\infty}f(u)\dd u+e^{-\epsilon|t|}O(e^{\epsilon|\supp f|}(| f''|_{L^1}+| f|_{L^1})),
	\end{equation*}
	where $|\supp f|$ is the supremum of the absolute value of  $x$ in $\supp f$.
\end{prop}
\begin{proof}
	By the same computation as in \cite[Lemma 4.5]{li2017fourier} and \cite[Prop. 4.14]{boyer2016renewalrd}, we have
	\begin{align*}
	Rf(x,t)=\frac{1}{\sigma_{V,\mu}}\int_{-t}^{\infty}f(u)\dd u+\lim_{s\rightarrow 0^+}\frac{1}{2\pi}\int e^{-it\xi}\hat{f}(\xi)U(s+i\xi)\B(x)\dd\xi,
	\end{align*}
	where $\B(x)$ is the constant function with value $1$ on $X$ and $\hat{f}$ is the Fourier transform of $f$ given by $\hat{f}(\xi)=\int e^{i\xi u}f(u)\dd u$. Hence, we only need to control the error term.
	
	By Proposition \ref{prop:invrsetran}, we know that $U(z)$ is analytical on $\ceta$ and uniformly bounded by $(1+|\Im z|)^{2\gamma}$. Since $f$ is a compactly supported smooth function, the Fourier transform $\hat{f}$ is an analytic function on $\bb C$.  As $|\hat{ f}(\pm i\epsilon+\xi)|\leq e^{\epsilon|\supp f|}\frac{1}{|\xi|^2}| f''|_{L^1}$, and $|\hat{ f}(\pm i\epsilon+\xi)|\leq e^{\epsilon|\supp f|}| f|_{L^1}$ for $\epsilon>0$ and $\xi$ in $\R$, we have 
	\begin{equation}\label{equ:hat f}
	|\hat{ f}(\pm i\epsilon+\xi)|\leq e^{\epsilon|\supp f|}\frac{1}{1+|\xi|^2}(| f''|_{L^1}+| f|_{L^1}).
	\end{equation}
	By \eqref{equ:uz}, \eqref{equ:hat f} and the dominant convergence theorem, we have
	\begin{equation}\label{equ:s+}
	\lim_{s\rightarrow 0^+}\frac{1}{2\pi}\int e^{-it\xi}\hat{f}(\xi)U(s+i\xi)\B(x)\dd\xi=\frac{1}{2\pi}\int e^{-it\xi}\hat{f}(\xi)U(i\xi)\B(x)\dd\xi. 	
	\end{equation}
	
	\begin{lem}\label{lem:reed}\cite[Thm.\uppercase\expandafter{\romannumeral9}14]{reed1975methods}
		If $T$ is in $\cal S'(\bb R)$, the space of tempered distributions, the distribution $T$ has analytic continuation to $|\Im \xi|<a$ and $\sup_{|b|<a}\int|T(ib+y)|\dd y< \infty$, then $\check{T}$, the inverse Fourier transform of $T$, is a continuous function. For all $b<a$, let $C_b=\max\int|T(\pm ib+y)|\dd y$. We have
		\begin{equation*}
		|\check T(t)|\leq C_be^{-b|t|}.
		\end{equation*}
	\end{lem}
	
	Using Lemma \ref{lem:reed} with $T(\xi)=\hat{f}(\xi)U(i\xi)\B(x)$, we have
	\begin{align}\label{equ:check t}
	\left|\int\hat{ f}(\xi)U(i\xi)\bm 1(x)e^{-it\xi}\dd\xi\right|&=| \check{T}(t)|\leq e^{-\epsilon|t|}\max|T(\pm i\epsilon+\xi)|_{L^1(\xi)}.
	\end{align}
	for $\epsilon< \eta$.
	By \eqref{equ:uz} and \eqref{equ:hat f}, we have
	\begin{equation}\label{equ:max t}
	\begin{split}
	\max|T(\pm i\epsilon+\xi)|_{L^1(\xi)}&\leq e^{\epsilon|\supp f|}\int \frac{1}{1+|\xi|^2}(| f''|_{L^1}+| f|_{L^1})|U(\mp \epsilon+i\xi)\bm 1(x)|\dd\xi\\
	&\ll_\gamma e^{\epsilon|\supp f|}(| f''|_{L^1}+| f|_{L^1}).
	\end{split}
	\end{equation}
	Combining \eqref{equ:s+}, \eqref{equ:check t} and \eqref{equ:max t}, we obtain the result.
\end{proof}
\begin{prop}
	Under the same assumptions as in Theorem \ref{thm:renewal}, there exists $\epsilon>0$ such that for $ f\in C_c^{2}(\bb R)$, we have
	\begin{equation*}
	R_Pf(x,t)=\frac{1}{\sigma_{V,\mu}}\int_{-t}^{\infty}f(u)\dd u+e^{-\epsilon|t|}O(e^{\epsilon|\supp f|}(| f''|_{L^1}+| f|_{L^1})).
	\end{equation*}
\end{prop}
\begin{proof}
	The ideal of the proof is the same as \cite[Lemma 4.11 or Proposition 4.28]{li2017fourier}, where we only need to replace the error term by the error term in the above Proposition \ref{prop:reniwa}. 
	
	We summarize the main idea here. By the large deviation principle, the main contribution of the renewal sum is given by $n$ in a small interval containing $t/\sigma_{V,\mu}$. Since the norm is good, we have the interpretation of the norm by the Cartan projection \eqref{equ:representation cartan}. Then we use \cite[Lemma 17.8]{benoistquint} to replace the norm by the cocycle $\sigma_V$ for each $n$ in the small interval. The proof is complete by applying Proposition \ref{prop:reniwa}.
\end{proof}

\section{Appendix}
\subsection{Non simply connected case}
\label{sec:semisimple}
We explain here how to get Theorem \ref{thm:foudecsemi} for connected algebraic semisimple Lie groups defined and split over $\R$ from Theorem \ref{thm:foudec} for connected $\R$-split reductive $\R$-groups whose semisimple part is simply connected, which is proved in Section \ref{sec:proof}. 

See \cite{margulis91discrete} and \cite[\S 22]{borel1990linear} for more facts about algebraic groups and the central isogeny.
\begin{lem}\label{lem:surjective}
	Let $\bf G'$ be a connected algebraic semisimple Lie groups defined over $\R$. Then there exist a connected reductive $\R$-group $\bf G$ with simply connected derived group $\scr D\bf G$ and an algebraic group morphism $\psi:\bf G\rightarrow \bf G'$ which is surjective between real points. Moreover, the restriction of $\psi$ to $\scr D\bf G$ gives a central isogeny from $\scr D\bf G$ to $\bf G'$ and the connected centre of $\bf G$ is $\R$-split.
\end{lem}
\begin{proof}
	Let $\bf A'$ be a maximal $\R$-split torus of $\bf G'$. Let $\bf G_1$ be a cover of $\bf G'$ which is simply connected and let $f$ be the isogeny map from $\bf G_1$ to $\bf G'$. Let $\bf A_1$ be the preimage of $\bf A'$ in $\bf G_1$, which is a maximal $\R$-split torus of $\bf G_1$ \cite[Theorem 22.6 (ii)]{borel1990linear}. Let $\bf N=\ker f\cap \bf A_1$, then $\bf A'$ is isomorphic to $\bf A_1/\bf N$ as torus. Consider the action of $\bf A_1$ on $\bf G_1$ via conjugation, that is for $s\in \bf A_1$ and $g\in \bf G_1$ we define $Int_s(g)=s^{-1}gs$. Since the kernel of $f$ is in the centre of $\bf G_1$, the conjugation action of $\bf N$ on $\bf G_1$ is trivial. By \cite[Corollary 6.10]{borel1990linear}, the quotient group $\bf A'\simeq \bf A_1/\bf N$ acts $\R$-morphically on $\bf G_1$.
	\begin{center}
		\begin{tikzcd}
			\bf A_1\times \bf G_1 \arrow[d, "f\times Id"']  \arrow[r, "Int"] & \bf G_1 \arrow[dd,"f"] \\
			\bf A'\times \bf G_1 \arrow[ru,dotted] \arrow[d, "Id\times f"'] &  \\
			\bf A'\times \bf G' \arrow[r,"Int"]& \bf G'
		\end{tikzcd}
	\end{center}	
	
	Hence, we can define the semidirect product $\bf G=\bf A'\ltimes \bf G_1$, given by the action of $\bf A'$ on $\bf G_1$. The derived group $[\bf G,\bf G]$ equals $\bf G_1$, which is simply connected. The group $\bf G$ is defined over $\R$, because $\bf A',\bf G_1$ and $\psi$ are also. The restriction of the action of $\bf A'$ on $\bf A_1$ is trivial and $\bf A'\times \bf A_1$ is a maximal $\R$-split torus of $\bf G$. 
	Hence the group $\bf G$ is a connected reductive $\R$-group. 
	
	We only need to find the surjective morphism $\psi$. Let $\bf A'\ltimes \bf G'$ be the semidirect product given by the conjugation action of $\bf A'$ on $\bf G'$. As $\bf A'$ is a subgroup of $\bf G'$, this semidirect product is a product. Hence $(s,g)\mapsto sg$ from $\bf A'\ltimes\bf G'$ to $\bf G'$ is a group morphism.
	We have a group morphism
	\begin{align*}
	\bf G=\bf A'\ltimes \bf G_1\rightarrow &\bf A'\ltimes \bf G'\rightarrow \bf G'\\
	\psi:	(s,g)\mapsto& (s,f(g))\mapsto sf(g).	 
	\end{align*}
	It is well-known that the real part of a semisimple simply connected group $G_1$ is connected in the analytic topology. (See for example \cite{steinberg1968}) Let $(G')^o$ be the analytic connected component of the identity element in $G'$. Then the image of real points of $\bf G$ under $\psi$ is $A'(G')^o$, which is equal to $ G'$ by a theorem of Matsumoto \cite{matsumoto64gpreel} (\cite[Théorème 14.4]{boreltits65reductif}).
\end{proof}
\begin{exa}
	When $\bf G'=\bf{PGL}_2$, the above construction gives $\bf G=\bf{GL}_2=\bf{GL}_1\ltimes \bf{SL}_2$ and the map $\psi$ is the quotient map from $\bf{GL}_2$ to $\bf{PGL}_2$.
\end{exa}

Let $G'=\bf G'(\R)$ be the group of real points of a connected algebraic semisimple Lie groups defined and split over $\R$. Recall that $\mu$ is a Zariski dense Borel probability measure on $G'$ with a finite exponential moment. If $\bf G'$ is simply connected, then Theorem \ref{thm:foudec} holds for $\bf G'$.
If not, let $G=\bf G(\R)$ be as in Lemma \ref{lem:surjective}. 
Recall that $\psi$ is a group morphism from $G$ to $G'$, 
\begin{lem}\label{lem:psimu}
There exists a Zariski dense Borel probability measure $\tilde{\mu}$ on $G$ with a finite exponential moment such that 
\begin{equation}\label{equ:psimu}
\psi_*\tilde{\mu}=\mu.
\end{equation}
\end{lem}
The proof of Lemma \ref{lem:psimu} will be given at the end this section.
We will explain why the result also holds for $G'$ and $\mu$. We state the non simply connected version of Theorem \ref{thm:foudec} here.

\begin{thm}[Fourier decay]\label{thm:foudecsemi}
	Let $\bf G'$ be a connected algebraic semisimple Lie group defined and split over $\R$ and let $G'=\bf G'(\R)$ be its group of real points.
	Let $\mu$ be Zariski dense Borel probability measure on $G'$ with finite exponential moment. Let $\nu$ be the $\mu$-stationary measure on the flag variety $\P$.
	
	For every $\gamma>0$, there exist $\epsilonzer > 0,\epsilonone >0$ depending on $\mu$ such that the following holds. For any pair of real functions $\varphi\in C^2(\PP)$, $r\in C^\gamma(\PP)$ and $\xi>0$ such that $\varphi$ is $(\xi^{\epsilonzer },r)$ good, $\|r\|_\infty\leq 1$ and $c_\gamma(r) \leq \xi^{\epsilonzer }$,
	then 
	\begin{equation*}
	\left|\int e^{i\xi \varphi(\eta)}r(\eta)\dd\nu(\eta)\right|\leq \xi^{-\epsilonone } \quad\text{ for all } \xi \text{ large enough}.
	\end{equation*}
\end{thm}
\begin{proof}
	We only need to prove that the flag varieties of $G$ and $G'$ are isomorphic.
	By Lemma \ref{lem:surjective}, since $\bf G'$ is $\R$-split, $\bf G$ is also $\R$-split. By Remark \ref{rem:distance}, the distance $d_\alpha$ on $\P$ for $\bf G$ and $\bf G'$ are equivalent. By Lemma \ref{lem:psimu},  we obtain that $\nu$ is also the $\tilde{\mu}$-stationary measure and we can apply Theorem \ref{thm:foudec} to $G,\tilde{\mu}$.
	
	Now we will prove that the flag varieties of $G$ and $G'$ are isomorphic.
	By \cite[Prop.20.5]{borel1990linear}, we know that $(\bf G_2/\bf P_2)(\R)=\bf G_2(\R)/\bf P_2(\R)$ for any connected reductive $\R$-group $\bf G_2$ and its parabolic $\R$-subgroup $\bf P_2$. Hence it is sufficient to prove for a minimal parabolic $\R$-subgroup $\bf P$ of $\bf G$ and $\bf P'$ its image in $\bf G'$ that
	\begin{equation}\label{equ:gp}
		\bf G/\bf P\simeq \bf G'/\bf P'.
	\end{equation}
	As if \eqref{equ:gp} holds, then $\bf P'$ is also a parabolic subgroup by definition and it is minimal because $\bf P$ is.
	Due to \cite[Thm.11.16]{borel1990linear}, the normalizer of a parabolic subgroup is itself. Then the centre of $\bf G$ is contained in the parabolic group $\bf P$. It suffices to prove that $\ker \psi$ is in the centre, then 
	\[\bf G/\bf P\simeq (\bf G/\ker\psi)/(\bf P/\ker\psi)\simeq \bf G'/\bf P'.\]
	
	By \cite[Prop.14.2]{borel1990linear}, we know that $\bf G=\bf C\cdot\scr D\bf G$, where $\bf C$ is the connected centre and $\bf C\cap\scr D\bf G$ is finite. Since $\bf G'$ is semisimple, the connected centre $\bf C$ is in $\ker\psi$. As the restriction of $\psi$ on $\scr D\bf G$ to $\bf G'$ is a central isogeny, hence $\ker\psi\cap \scr D\bf G$ is in the centre of $\scr D\bf G$, which is also in the centre of $\bf G$. Therefore the kernel of $\psi$ is in the centre of $\bf G$. The proof is complete.
\end{proof}
It remains to prove Lemma \ref{lem:psimu}.
\begin{proof}[Proof of Lemma \ref{lem:psimu}]
	We will first construct a measure $\tilde{\mu}_1$ which has a finite exponential moment. In the construction of Lemma \ref{lem:surjective}, there exists a finite subgroup $F$ of $A'$ such that $\psi$ from $F\ltimes G_1$ to $G'$ is already surjective. Let $F_1$ be the kernel of this covering, which is finite. Then there exists a Borel probability measure $\tilde{\mu}_1$ on the subgroup $F\ltimes G_1$ of $G$ which is $F_1$-left invariant and the pushforward measure is $\mu$.

	The moment condition is also satisfied. Because $\psi$ induce an isomorphism between $\frak a_1$ to $\frak a'$ (Recall the notation in Section \ref{sec:lie groups}) and this isomorphism identifies the Cartan projections $\kappa(g)$ and $\kappa(\psi(g))$. 
	Let $mg$ be an element in $F\ltimes G_1$ with $m\in F$ and $g\in G_1$, then
	by the sub additivity of the Cartan projection (\cite[Corollary 8.20]{benoistquint}),
	\[\|\kappa(mg)\|\leq \|\kappa(m)\|+\|\kappa(g)\|= \|\kappa(m)\|+\|\kappa(\psi(g))\|\leq \|\kappa(m)\|+\|\kappa(\psi(m))\|+\|\kappa(\psi(mg)) \|. \]
	Hence
	\[\int_G e^{\epsilon\|\kappa(g)\|}\dd\tilde{\mu}_1(g)\ll\int_{G}e^{\epsilon\|\kappa(\psi(g))\|}\dd\tilde\mu_1(g)= \int_{G'}e^{\epsilon\|\kappa(g')\|}\dd\mu(g'). \]
	
	In order to get a Zariski dense measure $\tilde{\mu}$, we replace the above measure $\tilde{\mu}_1$ by $\tilde{\mu}=\frac{1}{2}(\tilde{\mu}_1+c_*\tilde{\mu}_1)$, where $c$ is an element in the connected centre $C$ such that the group $C_1=\l c\r$ generated by $c$ is Zariski dense in $C$. Due to $\dd\psi|_{\frak c}=0$, the connected centre $C$ is in the kernel of $\psi$. Hence $\psi_*(c_*\tilde{\mu}_1)=\psi_*\tilde{\mu}_1$.
	This measure $\tilde{\mu}$ satisfies \eqref{equ:psimu} and has a finite exponential moment. We will prove that it is also Zariski dense.
	
	Let $\bf H$ be the Zariski closure of $\Gamma_{\tilde{\mu}}$, the group generated by the support of $\tilde{\mu}$. Let $\frak h$ be the Lie algebra of $H$. Since the group $\bf G$ is a connected $\R$-group, it is sufficient to prove that $\frak h=\frak g$. 
	Recall that $\frak g=\frak c\oplus\scr D\frak g$. Due to $c$ in $H$, the Zariski closure of $C_1$ is also in $H$. Hence $\frak h\supset \frak c$. For the semisimple part, consider the adjoint action of $\Gamma_{\tilde{\mu}}$ on $\scr D\frak g$. Because the group $\Gamma_\mu$ is Zariski dense in $\bf G'$, the adjoint action of $\Gamma_\mu$ on $\frak g'$ is irreducible. The map $\dd\psi|_{\scr D\frak g}:\scr D\frak g\rightarrow \frak g'$ is an isomorphism of Lie algebras. By 
	\[\dd\psi(\ad_gX)=\ad_{\psi(g)}\dd\psi X\text{ for }X\in\scr D\frak g, \]
	we obtain that the action of $\Gamma_{\tilde{\mu}}$ on $\scr D\frak g$ is irreducible. Since $\frak h\cap \scr D\frak g$ is nonzero 
	and $\Gamma_{\tilde{\mu}}$-invariant, we know that $\frak h\cap\scr D\frak g=\scr D\frak g$. Therefore $\frak h=\frak g$. The proof is complete.
\end{proof}

\subsection{Equivalence of distances}
\label{sec:equi distance}
We consider connected $\R$-split reductive $\R$-groups.
\begin{defi}
	Let $(X,d)$ be a metric space. Let $d'$ be another metric on $X$. We say that $d,d'$ are equivalent if there exist $c,C>0$ such that for all $x_1,x_2$ in $X$
	\[ cd(x_1,x_2)\leq d'(x_1,x_2)\leq Cd(x_1,x_2). \]
\end{defi}
Recall that $\P_0$ is the homogeneous space $G/A_eN$, on which the compact group $K$ acts simply transitively. Recall that $\{V_\alpha \}_{\alpha\in\Pi}$ is the family of representations fixed in Lemma \ref{lem:tits}. We will define three distances on $\P_0$. Due to the fact that $\P_0$ is homeomorphic to $K$, a distance on $\P_0$ is also a distance on $K$ and we will continue our argument on $K$.
Let $k,k'$ be two points in $K$. If they are not in the same connected component, we define their distance as 1. From now on, we always suppose that $k,k'$ are in the connected component $K^o$.
\begin{itemize}
		\item $d_0(k,k')=\sup_{\alpha\in\Pi}\|kv_\alpha-k'v_\alpha \|/\sqrt{2}$, where $v_\alpha$ is a unit vector in $V_\alpha$ with highest weight. This is also the distance induced by the embedding of $\P_0$ into $\Pi_{\alpha\in\Pi}\bb SV_\alpha$. \nomentry{$d_0(k,k')$}{}
	\item $d_1(k,k')$ is the distance induced by the bi-invariant Riemannian metric on $K$.
\end{itemize}
We can easily verify that they are distances.
\begin{lem}\label{lem:equivalent distance}
	The two distances $d_0$ and $d_1$ on $\P_0$ are equivalent.
\end{lem}
Before proving Lemma \ref{lem:equivalent distance}, we need a lemma.
\begin{lem}\label{lem:kvv}
	Let $(\rho,V)$ be an irreducible representation with highest weight $\chi$, which satisfies $\chi(H_\alpha)>0$ for only one simple root $\alpha$. Then there exists $t_0>0$ such that the following holds. Let $Z$ be a unit vector in $\frak k$, given by $Z=\sum_{\alpha\in R^+}c_\alpha K_\alpha$. 
	Let 
		\[Z_\alpha=\sum_{\beta\geq\alpha,\beta\in R^+}c_\beta K_\beta.\]
	Then for $0<t<t_0$, $k=\exp(tZ)$ and a unit vector $v$ with highest weight, we have
	\[\|kv-v\|\asymp t\|Z_\alpha\|. \]
\end{lem}
\begin{proof}	
For a positive root $\beta$, let
	\[A_{\beta}:=\dd\rho(K_\beta)v=\dd\rho(Y_\beta)v. \]
	Consider the representation of $\frak{s}_\beta=\{Y_\beta,X_\beta,H_\beta \}\simeq \frak{sl}_2$. Due to the classification of the representations of $\frak{sl}_2$, the vector $A_{\beta}$ is non-zero if and only if $\chi(H_\beta)>0$.

	Fix an inner product $(\cdot,\cdot)$ on $\frak a^*$ which is invariant under the Weyl group, then we can identify $H_\beta$ with $2\frac{\beta}{(\beta,\beta)}$, that is (see \cite[V. 5]{serre2012complex} for example)
	\[\chi(H_\beta)=(\chi,2\frac{\beta}{(\beta,\beta)}). \]
	By hypothesis, $(\chi,\alpha)>0$ for only one simple root $\alpha$, this implies that $\chi(H_\beta)=2(\chi,\beta)/(\beta,\beta)>0$ if and only if $\beta\geq \alpha$ and $\beta$ is a positive root. 
	 Therefore only the vectors $\{A_\beta \}_{\beta\geq\alpha,\beta\in R^+}$ are non-zero. They are also orthogonal since they are of different weights. When $t$ is small enough, by Lipschitz's property we conclude
	\[\|kv-v\|=\|\exp(tZ)v-v\|\asymp t\|\dd\rho(Z)v \|=t\|\sum_{\beta\geq\alpha,\beta\in R^+}c_\beta A_\beta\|\asymp t\|Z_\alpha\|.\]
	The proof is complete.
\end{proof}
\begin{proof}[Proof of Lemma \ref{lem:equivalent distance}]	
	Due to the embedding $\P_0$ to $\Pi_{\alpha\in\Pi}\bb SV_\alpha$ and since the Riemannian metric on $\Pi_{\alpha\in\Pi}\bb SV_\alpha$ induces a metric on $\P_0$ which is equivalent to $d_0$, we know that for any $x_1,x_2$ in $\P_0$, 
	\[d_0(x_1,x_2)\ll d_1(x_1,x_2). \]
	
	We observe that the two distances are left $K$-invariant. It is sufficient to prove the equivalence for $k'$ equal to the identity $e$.
	
	Fix $\epsilon$ small depending on $K$. Let $B_\epsilon$ be the neighbourhood of $e$ given by  $\{k\in K|d_1(k,e)<\epsilon \}$. Then $B_\epsilon^c$ is a compact subset of $K$. Consider the function $f(k)=\frac{d_0(k,e)}{d_1(k,e)}$ for $k\in B_\epsilon^c$. Then $f$ is a positive continuous function on $B_\epsilon^c$. The compactness of $B_\epsilon^c$ implies that it has positive minimum on $B_\epsilon^c$. Hence there exists $c>0$ such that for $k$ outside of $B_\epsilon$ 
	\[d_0(k,e)\geq cd_1(k,e). \]
	
	Finally, we only need to consider a small neighbourhood of the identity. We take $\epsilon$ small such that the exponential map at $e$ is bi-Lipschitz. Suppose that $k=\exp(tZ)$ with $Z$ a unit vector in $\frak k$ and $t>0$. 
	We can decompose $Z$ as in Lemma \ref{lem:kvv}.
	There exists $\alpha\in \Pi$ such that $\|Z_\alpha\|\gg 1$. By Lemma \ref{lem:kvv},
	we have
	\[\|kv_\alpha-v_\alpha\|\asymp t\|Z_\alpha \|\gg t.  \]
	Then we have $d_0(k,e)\gg d_1(k,e)$. The proof is complete.
\end{proof}

\begin{lem}\label{lem:equivalence distance P}
	The $K$-invariant Riemannian distance on $\P$ is equivalent to the distance defined in \eqref{equ:distance eta eta'}.
\end{lem}
\begin{proof}
	By $\P=\P_0/M$ and since the group $M$ is a subgroup of $K$ which preserves the distance, let $d_1$ also be the quotient Riemannian distance on $\P$.
	By the same argument of the proof as in Lemma \ref{lem:equivalent distance}, it is sufficient to prove a small neighbourhood of $\eta_0$.
	For any two points $\eta$, $\eta'$ in this small neighbourhood, we can find $\k$, $\k'$ in $\P_0$ such that $\pi(\k)=\eta$, $\pi(\k')=\eta'$ and $d_1(\k,\k')=d_1(\eta,\eta')$. Due to $d_1(\k,\k')$ small, we see that $d_0(\k,\k')$ is less than $1$. Hence by  
	\[\|v_1-v_2\|\asymp \|v_1\wedge v_2\| \]
	for two unit vectors $v_1,v_2$ with $\|v_1-v_2\|<\sqrt{2}$,
	we have
	\[d(\eta,\eta')\asymp d_0(\k,\k').\] 
	By Lemma \ref{lem:equivalent distance}, we have $d_0(\k,\k')\asymp d_1(\k,\k')=d_1(\eta,\eta')$. The proof is complete.
\end{proof}

Recall the definition of the sign function $\sg$ in Section \ref{sec:sign group}. 
\begin{lem}\label{lem:p0 p}
	Suppose in addition that semisimple part of the group $\bf G$ is simply connected. Let $\k=k\k_o,\k'=k'\k_o$ be two points in $\P_0$, then
	\[ \sqrt{2}d_0(\k,\k')\geq d(\pi(\k),\pi(\k')). \]
	We have
		\[\sg(\k,\k')=e\Longleftrightarrow d_0(\k,\k')< 1. \]
	If $\sg(\k,\k')=e$, then
	\[ d(\pi(\k),\pi(\k'))\geq d_0(\k,\k'). \]
\end{lem}
\begin{proof}
	Suppose that the angle between $kv_\alpha$ and $k'v_\alpha$ is $\vartheta\in [0,\pi)$, then $\|kv_\alpha-k'v_\alpha\|=2\sin\frac{\vartheta}{2}$ and $d(V_{\alpha,k\eta_o},V_{\alpha,k'\eta_o})=\|kv_\alpha\wedge k'v_\alpha\|=\sin\vartheta=2\sin\frac{\vartheta}{2}\cos\frac{\vartheta}{2}\leq 2\sin\frac{\vartheta}{2}$, which implies the first inequality.
	
	The assumption $d_0(\k,\k')< 1$ is equivalent to that for every simple root $\alpha$,  the angle $\vartheta$ is less than $\pi/2$, which is equivalent to $\sg(\k,\k')=e$ due to Lemma \ref{lem:isomorphism M Z}.
	
	If $m(\k,\k')=e$, then for every simple root $\alpha$, the angle $\vartheta$ is less than $\pi/2$. Hence $\sin\vartheta=2\sin\frac{\vartheta}{2}\cos\frac{\vartheta}{2}\geq \sqrt{2}\sin\frac{\vartheta}{2}$, which implies the result.
\end{proof}

\noindent Jialun LI\\
Universit\"at Z\"urich\\
jialun.li@math.uzh.ch
\newpage
\printnomenclature
\addcontentsline{toc}{section}{List of symbols}
\clearpage

\end{document}